\documentclass[12pt]{amsart}
\usepackage{enumerate,amssymb,amsfonts,latexsym,psfrag}
\usepackage{amscd,amsmath}
\usepackage[dvips]{graphicx}
\newtheorem{thm}{Theorem}[section]
\newtheorem{cor}[thm]{Corollary}
\newtheorem{lem}[thm]{Lemma}
\newtheorem{prop}[thm]{Proposition}
\theoremstyle{remark}
\newtheorem{rem}{Remark}[section]

\newtheorem{clm}[thm]{Claim}

\theoremstyle{definition}
\newtheorem{defn}{Definition}[section]

\numberwithin{equation}{section}
\numberwithin{figure}{section}

 \newcommand{\pichere}[2]
{\begin{center}\includegraphics[width=#1\textwidth]{#2}\end{center}}

\font\nt=cmr7

\def\note#1
{\marginpar
{\nt $\leftarrow$
\par
\hfuzz=20pt \hbadness=9000 \hyphenpenalty=-100 \exhyphenpenalty=-100
\pretolerance=-1 \tolerance=9999 \doublehyphendemerits=-100000
\finalhyphendemerits=-100000 \baselineskip=6pt
#1}\hfuzz=1pt}

\newcommand{\bignote}[1]{\begin{quote} \sf #1 \end{quote}}

\newcommand{\di}{\partial}

\newcommand{\ra}{\rightarrow}

\newcommand{\diam}{\operatorname{diam}}
\newcommand{\dist}{\operatorname{dist}}

\newcommand{\inter}{\operatorname{int}}

\newcommand{\tl}{\tilde}

\newcommand{\transverse}
 {\kern .7em\makebox[0pt][c]{\raisebox{.2ex}{$|$}}\kern -.6em\cap}

\newcommand{\tangent}
 {\kern .7em\makebox[0pt][c]{\raisebox{.77ex}{$--$}}\kern -.6em\cap}

\def\loc{{\mathrm{loc}}}

\newcommand{\eps}{{\varepsilon}}

\newcommand{\de}{{\delta}}
\newcommand{\la}{{\lambda}}
\newcommand{\La}{{\Lambda}}

\newcommand{\Om}{{\Omega}}

\newcommand{\bbe}{{\mbox{\boldmath$\beta$} }}

\newcommand{\bg}{{\boldsymbol{\gamma}}}

\newcommand{\AAA}{{\mathcal A}}

\newcommand{\CC}{{\mathcal C}}

\newcommand{\II}{{\mathcal I}}
\newcommand{\FF}{{\mathcal F}}

\newcommand{\HH}{{\mathcal H}}
\newcommand{\KK}{{\mathcal K}}

\newcommand{\OO}{{\mathcal O}}
\newcommand{\PP}{{\mathcal P}}

\newcommand{\TT}{{\mathcal T}}

\newcommand{\UU}{{\mathcal U}}
\newcommand{\VV}{{\mathcal V}}
\newcommand{\WW}{{\mathcal W}}

\newcommand{\R}{{\mathbb R}}

\def\BPhi{{\boldsymbol{\BPhi}}}
\def\B0{{\mathbf{0}}}

\newcommand{\Jac}{\operatorname{Jac}}

\newcommand{\Dom}{\operatorname{Dom}}

\newcommand{\Domain}{\operatorname{Dom}}
\newcommand{\Image}{\operatorname{Im}}
\newcommand{\graph}{\operatorname{graph}}
\newcommand{\Trap}{\operatorname{Trap}}
\newcommand{\Orb}{\operatorname{Orb}}

\catcode`\@=12

\def\Empty{}
\newcommand\oplabel[1]{
  \def\OpArg{#1} \ifx \OpArg\Empty {} \else
  	\label{#1}
  \fi}
		
\newcommand{\comm}[1]{}
\newcommand{\comment}[1]{}

\begin{document}

\bigskip\bigskip

\title[H\'enon renormalization]{Renormalization in the H\'enon family, II: \\
 The heteroclinic web\\}
\author{M. Lyubich, M. Martens}

\address {Stony Brook University and University of Toronto}


\date{\today}

\begin{abstract} We study highly dissipative  H\'enon maps 
$$
F_{c,b}: (x,y) \mapsto (c-x^2-by, x)
$$ 
with zero entropy.
They form a region  $\Pi$ in the parameter plane bounded on the left by the curve $W$ 
of infinitely renormalizable maps.
We prove that Morse-Smale maps are dense in $\Pi$,
but there exist infinitely many different topological types
of such maps (even away from $W$). 
We also prove that in the infinitely renormalizable case,
the average Jacobian $b_F$ on the attracting Cantor set $\OO_F$ is a topological invariant.
These results come from the analysis of the heteroclinic web 
of the saddle periodic points based on the renormalization theory. 
Along these lines, we show that the unstable manifolds of the periodic points
form a lamination outside $\OO_F$ if and only if there are no heteroclinic tangencies.

\end{abstract}

\maketitle
\thispagestyle{empty}

\def\IMSmarkvadjust{0 pt}
\def\IMSmarkhadjust{0 pt}
\def\IMSmarkhpadding{0 pt}
\def\IMSpubltext{Published in modified form:}
\def\SBIMSMark#1#2#3{
 \font\SBF=cmss10 at 10 true pt
 \font\SBI=cmssi10 at 10 true pt
 \setbox0=\hbox{\SBF \hbox to \IMSmarkhpadding{\relax}
                Stony Brook IMS Preprint \##1}
 \setbox2=\hbox to \wd0{\hfil \SBI #2}
 \setbox4=\hbox to \wd0{\hfil \SBI #3}
 \setbox6=\hbox to \wd0{\hss
             \vbox{\hsize=\wd0 \parskip=0pt \baselineskip=10 true pt
                   \copy0 \break%
                   \copy2 \break%
                   \copy4 \break}}
 \dimen0=\ht6   \advance\dimen0 by \vsize \advance\dimen0 by 8 true pt
                \advance\dimen0 by -\pagetotal
	        \advance\dimen0 by \IMSmarkvadjust
 \dimen2=\hsize \advance\dimen2 by .25 true in
	        \advance\dimen2 by \IMSmarkhadjust

%
%
  \openin2=publishd.tex
  \ifeof2\setbox0=\hbox to 0pt{}
  \else 
     \setbox0=\hbox to 3.1 true in{
                \vbox to \ht6{\hsize=3 true in \parskip=0pt  \noindent  
                {\SBI \IMSpubltext}\hfil\break
                \input publishd.tex 
                \vfill}}
  \fi
  \closein2
  \ht0=0pt \dp0=0pt
 \ht6=0pt \dp6=0pt
 \setbox8=\vbox to \dimen0{\vfill \hbox to \dimen2{\copy0 \hss \copy6}}
 \ht8=0pt \dp8=0pt \wd8=0pt
 \copy8
 \message{*** Stony Brook IMS Preprint #1, #2. #3 ***}
}

\SBIMSMark{2008/2}{April 2008}{}

\setcounter{tocdepth}{1}
\tableofcontents

\section{Introduction}

The renormalization theory for the H\'enon family began with the works of Collet, Eckman and Koch and
 Gambaudo, van Strien and Tresser \cite{CEK}, \cite{GST}.
In this paper we continue our exploration of renormalization of H\'enon maps
started in \cite{CLM}. As Part I was mostly concerned with geometric properties of 
the Cantor attractor $\OO_F$, here we focus on global topological properties of the maps in question
that are essentially determined by the structure of the web of the stable/unstable manifolds
of saddle periodic points (we call it the ``heteroclinic web'').
As in Part I, we have encountered  here some surprising phenomena.   

In the one-dimensional situation,
all infinitely renormalizable maps with the same combinatorics are topologically equivalent.   
It is not anymore the case in the H\'enon family; in fact,   all infinitely renormalizable
H\'enon maps near the Feigenbaum point are topologically distinct.  
More generally, the average Jacobian $b_F$ of a H\'enon-like map is a topological invariant:
varying $b_F$ leads to bifurcations in the heteroclinic web (\S 9). 

Along these lines, we carry out a detailed analysis of the heteroclinic web.
In particular, we show that the unstable manifolds form a lamination (outside the attractor $\OO_F$)
if and only if there are no heteroclinic tangencies (\S\S 4,6). We also show that the orbit of the ``tip''
of $\OO_F$ (a counterpart of the critical value of one-dimensional maps)   
is topologically distinguished: it is respected by topological conjugacies (\S 5). 

\begin{figure}[htbp]
\begin{center}
\psfrag{a}[c][c] [0.7] [0] {\Large $a$}
\psfrag{t}[c][c] [0.7] [0] {\Large $t$} 

\psfrag{0}[c][c] [0.7] [0]{\Large $0$} 
\psfrag{In}[c][c] [0.7] [0]{\Large $I_m$} 
\psfrag{Mn}[c][c] [0.7] [0] {\Large $M_m$}
\psfrag{K}[c][c] [0.7] [0] {\Large $\KK_{\cdot, \cdot}$} 
\psfrag{Kkn}[c][c] [0.7] [0] {\Large $\KK_{k,n}$} 
\psfrag{W}[c][c] [0.7] [0] {\Large $W$} 
\psfrag{jkn}[c][c] [0.7] [0] {\Large $j_{k}(F_{n-1})=j$}

\pichere{0.9}{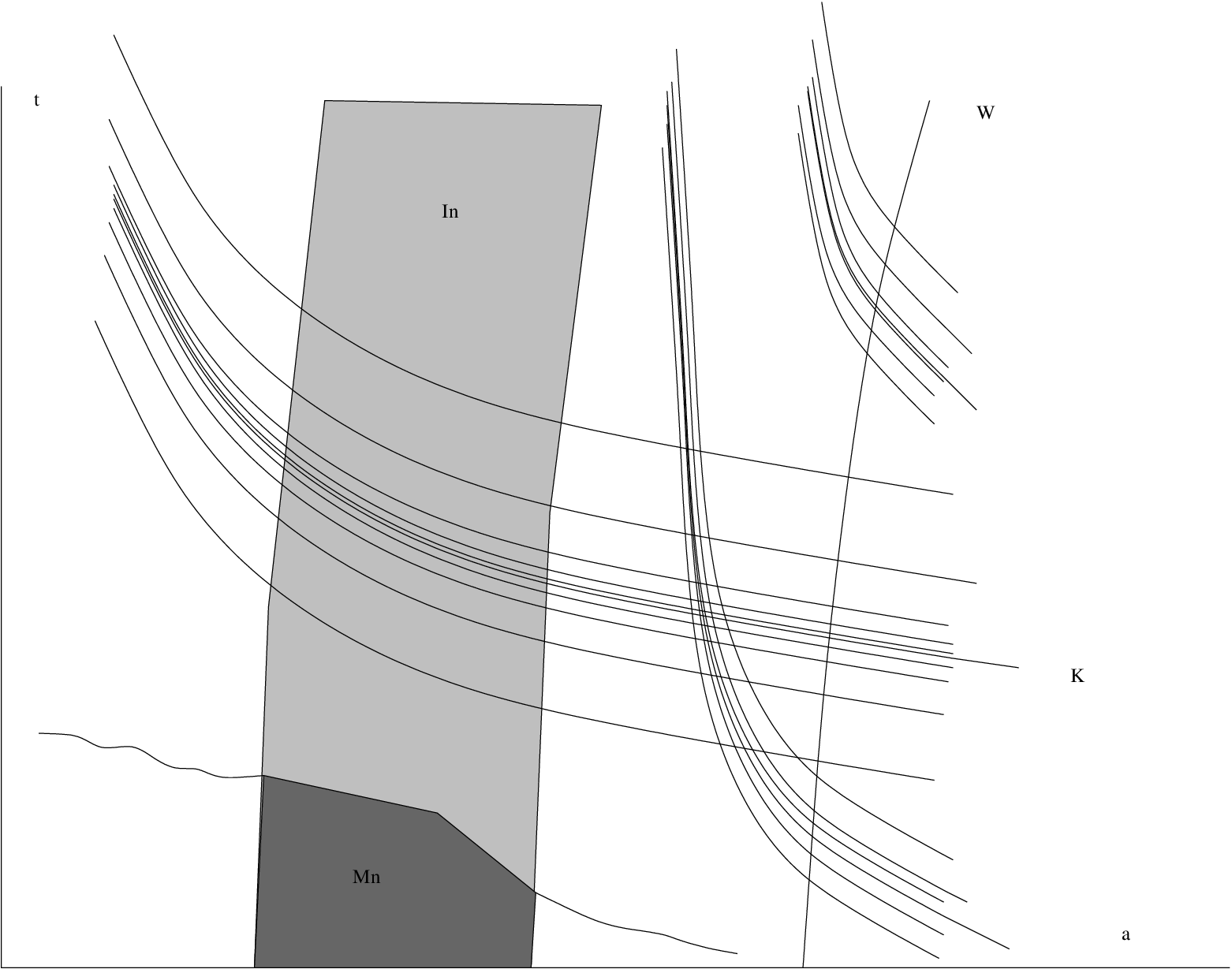} \caption{Bifurcation pattern }
 \label{bif}
\end{center}
\end{figure}

Infinitely renormalizable maps in question separate the regions with regular
(zero entropy) and chaotic (positive entropy) dynamics, \cite{GST}.  
Because of the Newhouse phenomenon, hyperbolic maps are not dense in the chaotic region.
However, it is conceivable that they are dense in the regular region. 
We confirm this conjecture in a narrow strip to the left of the curve of infinitely renormalizable maps: Morse-Smale maps are dense over there (\S 10).
However, the situation is quite intricate,
as  there are infinitely many different types  of Morse-Smale maps in this region. 
In particular, a boundary arc of a Morse-Smale component which is not related to the longest periodic orbit
 is accumulated by infinitely many different Morse-Smale components. 

\begin{rem} The maps in the region we consider, a narrow strip to the left of the curve of infinitely renormalizable maps, do not have homoclinic intersections. Previous 
results in \cite{PS} and \cite{C} imply that maps in this region can be $C^1$ 
approximated by Morse-Smale maps. 
\end{rem}

The results of \S 10 are illustrated in Figure \ref{bif}. It shows an 
artist's impression of parts of the
bifurcation pattern of a H\'enon family in a neighborhood of $W$, the curve of 
infinitely renormalizable maps. 
The strip $I_m$ consists of the  $m$-times renormalizable maps.  
The bottom Morse-Smale component $M_m$ is attached to 
the $m$-times renormalizable 
unimodal maps. 
The curves $\KK_{\cdot , \cdot }$ illustrate loci of heteroclinic tangencies.

\centerline{\bf Acknowledgements}

\bigskip

We thank Andre de Carvalho and Charles Tresser for many inspiring H\'enon discussions. 
We also thank all the institutions and foundations that have supported us in the course of this work: 
Simons Mathematics and Physics  Endowment,  Fields Institute, NSF, NSERC, University of Toronto.

\section{Preliminaries}\label{prelim}

The precise definitions and proofs of the following statements can be found in part I, see \cite{CLM}, of this series on renormalization of H\'enon maps.

\bigskip

Let $\Omega^h, \Omega^v\subset \mathbb{C}$ be neighborhoods of $[-1,1]\subset \mathbb{R}$ and $\Omega=\Omega^h\times \Omega^v$. Let $B=[-1,1]\times [-1,1]$ and $\overline{\epsilon}>0$.  The set
 $\HH_\Omega(\overline{\epsilon})$ consists of maps $F:B\to B$ of the following form.
$$
F(x,y)=(f(x)-\epsilon(x,y), x),
$$
where $f:[-1,1]\to [-1,1]$ is a unimodal map which admits a holomorphic extension to $\Omega^h$ and $\epsilon:B\to \mathbb{R}$ admits  a holomorphic extension to $\Omega$ and finally, $|\epsilon|\le \overline{\epsilon}$. The critical point $c$ of $f$ is non degenerate, $Df(c)<0$. A map in  $\HH_\Omega(\overline{\epsilon})$ is called a {\it H\'enon-like} map. Observe that H\'enon-like maps map vertical lines to horizontal lines.

A unimodal map $f:[-1,1]\to [-1,1]$ with critical point $c\in [-1,1]$ is {\it renormalizable} if $f^2:[f^2(c),f^4(c)]\to [f^2(c),f^4(c)]$ is unimodal and 
$[f^2(c),f^4(c)]\cap f([f^2(c),f^4(c)])=\emptyset$. The renormalization of $f$ 
is the affine rescaling of $f^2|([f^2(c),f^4(c)]$, denoted by $Rf$. The domain of $Rf$ is again $[-1,1]$.  The renormalization operator $R$ has a unique fixed 
point $f_*:[-1,1]\to [-1,1]$.  The introduction of \cite{FMP} presents the history of renormalization of unimodal maps and describes the main results.

The {\it scaling factor} of this fixed point $f_*$ is
$$
\sigma=\frac{|[f_*^2(c),f_*^4(c)]|}{|[-1,1]|}.
$$  
We will also need $\lambda=1/\sigma=2.6\dots$.

A H\'enon map is renormalizable if there exists a domain $D\subset B$ such that
$F^2:D\to D$. The construction of the domain $D$ is inspired by renormalization of unimodal maps. In particular, it is a topological construction.  However, for small $\overline{\epsilon}>0$ the actual domain $A\subset B$, used to renormalize as was done in \cite{CLM}, has an analytical definition. The precise definition can be found in \S 3.5 of part I. If the renormalizable H\'enon maps is given by  $F(x,y)=(f(x)-\epsilon(x,y))$ then the 
domain $A\subset B$, an essentially vertical strip, is bounded by two curves of the form 
$$
f(x)-\epsilon(x,y)=\text{Const}.
$$
These curves are graphs over the $y$-axis with a slope of the order 
$\overline{\epsilon}>0$. 
The domain $A$ satisfies similar combinatorial properties as the domain of renormalization of a unimodal map:
$$
F(A)\cap A=\emptyset,
$$
and
$$
F^2(A)\subset A.
$$
Unfortunately, the restriction $F^2|A$ is not a H\'enon-like map 
as it does not map vertical lines into horizontal lines. 
This is the reason why the coordinated change needed to define the renormalization of $F$ is not an affine map, 
but it rather has the following form. Let
$$
H(x,y)=(f(x))-\epsilon(x,y),y)
$$
and
$$
G=H\circ F^2\circ H^{-1}.
$$
The map $H$ preserves horizontal lines and it is designed in such a way
 that the map $G$ maps vertical lines into horizontal lines. 
Moreover, $G$ is well defined on a rectangle $U\times [-1, 1]$ of 
full height. Here $U\subset [-1,1]$ is an interval of length $2/|s|$ 
with $s<-1$. 
Let us  rescale the domain of $G$ by the $s$-dilation $\Lambda$, 
such that the rescaled domain is of the form $[-1,1]\times V$,
where $V\subset \mathbb{R}$ is an interval of length $2/|s|$. Define the renormalization of $F$ by
$$
RF= \Lambda\circ G\circ \Lambda^{-1}.
$$
Notice that $RF$ is well defined on the rectangle $[-1,1]\times V$.
The coordinate  change $\phi= H^{-1}\circ \La^{-1}$ maps this rectangle 
onto the
topological rectangle  $A$ of full height.

The set of $n$-times renormalizable maps is denoted by $\HH^n_\Omega(\overline{\epsilon})\subset \HH_\Omega(\overline{\epsilon})$. If 
$F\in \HH^n_\Omega(\overline{\epsilon})$ we 
use the notation
$$
F_n=R^nF.
$$
The set of infinitely renormalizable maps is denoted by
$$
\II_\Omega(\overline{\epsilon})=\bigcap_{n\ge 1} 
\HH^n_\Omega(\overline{\epsilon}).
$$
The collection of maps in $\HH^n_\Omega(\overline{\epsilon})$ which have a periodic attractor of period $2^n$ is denoted by $\II^n_\Omega(\overline{\epsilon})$.

The renormalization operator acting on $\HH^1_\Omega(\overline{\epsilon})$,
$\overline{\epsilon}>0$ small enough, has a unique fixed point
 $F_*\in \II_\Omega(\overline{\epsilon})$. It is the degenerate map 
$$
F_*(x,y)=(f_*(x), x).
$$
This renormalization fixed point is hyperbolic and the stable manifold has codimension one. Moreover,
$$
W^s(F_*)=\II_\Omega(\overline{\epsilon}).
$$

If we want to emphasize that some set, say $A$, 
is associated with a certain map $F$ we 
use notation like $A(F)$.

The coordinate change which conjugates $F_k^2|A(F_k)$ to $F_{k+1}$ is
 denoted 
by
\begin{equation}\label{phi}
          \phi^k_v=(\Lambda_k \circ H_k)^{-1}: \Domain(F_{k+1})\to A(F_k). 
\end{equation}
Here $H_k$ is the non-affine part of the coordinate change used to define 
$R^{k+1}F$ and $\Lambda_k$ is the dilation by $s_k<-1$.
Now, for $k<n$, let
\begin{equation}\label{Phi}
\Phi^n_k=\phi^k_v\circ \phi^{k+1}_v\circ \cdots \circ \phi^{n-1}_v:
\Domain(F_{n})\to A_{n-k}(F_k),
\end{equation}
where
$$
A_k(F)=\Phi^k_0(\Domain(F_k))\cap B.
$$
Notice, that each $A_k\subset B$ is of full height and 
$\Phi^k_0$ conjugates $R^kF$ to $F^{2^k}|A_k$.
Furthermore,
$
A_{k+1}\subset A_k
$.

\bigskip

Let $n\ge 1$ and $F\in \HH^n_\Omega(\overline{\epsilon})$. The domain of $R^nF$
is
$$
\Omega_n=\Omega^h_n\times \Omega^v_n,
$$
where $[-1,1]\subset\Omega^h_n$. Furthermore, 
$R^nF(x,y)=(f_n(x)-\epsilon_n(x,y), x)$. 

\begin{lem}\label{sizedomain} Given $\Omega$ and $\overline{\epsilon}>0$ small enough, there exist $r>1$ and $C>0$ such that for every 
$F\in \HH^n_\Omega(\overline{\epsilon})$
$$
\diam(\Omega_n)\le C\cdot r^n.
$$
\end{lem}

\begin{proof} Let $k< n$. The maps $f_k|\Omega^h_k$ stay within a compact family and 
$|\epsilon_k|=O(\overline{\epsilon}^{2^k})$. This is explained in \S 4 of 
\cite{CLM}. Hence, the coordinate changes $\Lambda_k\circ H_k$ used to define $R^{k+1}F$ as a renormalization of $R^kF$ has a uniform bound on its derivative. The Lemma follows.
\end{proof}

\begin{rem}\label{asympsizedomain} For an infinitely renormalizable map $F\in
 \II_\Omega(\overline{\epsilon})$ the diameters of $\Omega_n$ grow exponentially. In particular
$$
\diam(\Omega^v_n)\asymp \lambda^n,
$$
where $\lambda=1/\sigma=2.6\dots$ and $\sigma$ the scaling factor of the unimodal renormalization fixed point. Let $\Phi^{k+1}_k:\Omega_{k+1}\to \Omega_k$ be the diffeomorphism which conjugates $R^{k+1}F$ to $(R^kF)^2|A(R^kF)$. The inverse of this diffeomorphism was constructed in \S 3.5 of  
\cite{CLM}. In fact,
$$
(\Phi^{k+1}_k)^{-1}=\Lambda_k\circ H_k,
$$
where $H_k(x,y)=(f_k(x)-\epsilon_k(x,y), y)$ and $\Lambda_k$ is a dilation. 
The scaling factor $s_k$ of $\Lambda_k$ converges exponentially fast: $s_k\to -\lambda$. 
This is shown in  Lemma 7.4 of \cite{CLM}. In particular, 
$$
\diam(\Omega^v_{k+1})=|s_k|\cdot \diam(\Omega^v_{k}).
$$
\end{rem}

Let
$$
B_{v^n}=\Phi^n_0(B).
$$
Notice, for $k<n$
$$
B_{v^{k+1}}\subset B_{v^k}.
$$
An infinitely renormalizable H\'enon-like  map has an invariant Cantor set:
$$
\OO_F=\bigcap_{n\ge 1} \bigcup_{i=0}^{2^n-1} F^i(B_{v^n}).
$$ 
Its geometry was discussed in part I. The dynamics on this Cantor set
 is conjugate to an adding machine. Its unique invariant measure is denoted 
by $\mu$. 
The {\it average  Jacobian}
$$
b_F=\exp\int \log \Jac F d\mu
$$ 
with respect to $\mu$ is an important parameter that
essentially influences the geometry of $\OO_F$, see \cite{CLM}.

The critical point (and critical value) of a unimodal map plays a crucial role 
in its dynamics. The counterpart of the critical value for H\'enon-like maps is the 
{\it tip}
$$
\{\tau_F\}=\bigcap_{n\ge 1} B_{v^n}.
$$

The convergence $R^nF\to F_*$ is 
exponential, $F\in \II_\Omega(\overline{\epsilon})$. 
Theorem 7.9 of \cite{CLM} gives a precise asymptotical form of the convergence. Namely,
\begin{equation}\label{univ}
R^nF(x,y)=(f_n(x)-b_F^{2^n}a(x)y(1+O(\rho^n)),x).
\end{equation}
 The analytic function  $a(x)$ is universal, independent of $F$, and positive, and $\rho<1$. The unimodal part converges exponentially 
fast: $f_n\to f_*$.

\bigskip

We will use the following general notions and notations throughout the text.

A ball in a metric space of radius $r>0$ and centered at $x$ is denoted
 by $B_r(x)$.
 The diameter of a set is denoted by $\text{diam}(\cdot)$.
Let  $\pi_1:X\times Y\to X$ and $\pi_2:X\times Y\to Y$ be the projections to resp. the first and second factor.

 The graph of a function $\phi:X\to Y$ is denoted by $\text{graph}(\phi)$.
The domain of a map $F$ is denoted by $\Domain(F)$.                     
The image is denoted by $\Image(F)=F( \Domain(F))$.

The tangent space at a point $x\in W$ of a smooth  curve
$W\subset \mathbb{R}^2$ is denoted by $T_xW$. If two submanifolds $M_1$ and $M_2$ (of $M$) are tangent at some point we write
$$
M_1\tangent M_2.
$$
If we want to specify a point of tangency $x\in M_1\cap M_2$ we write
$$
M_1\tangent_x M_2.
$$

The forward
orbit of a point or set is denoted by $\text{Orb}(\cdot)=\bigcup_{k\ge 0} f^k(\cdot)$. If a point has also a complete backward orbit then the backward and forward orbit together is denoted by $\Orb_\Bbb{Z}(\cdot)$.
The limit set of a point $x\in \Domain(F)$ is denoted by $\omega(x)$. 
If a point $x\in \Domain(F)$ has an infinite backward orbit then the limit set of this backward orbit 
is denoted by $\alpha(x)$.                                  
The cycle of a periodic point $\beta$, $\gamma$ etc. will be called $\bbe$, 
$\bg$ etc. 
The set of periodic points of a map $F$ is denoted by $\PP_F$.
A point
$x\in B$ is a {\it wandering } point for $F:B\to B$ if there is a neighborhood $x\in U$ such that $F^n(U)\cap U=\emptyset$, for $n\ge 1$. 
Denote the set of non-wandering points of $F$ by $\Omega_F$. Given two points $z,z'\in W^{u/s}(x)\subset B$ in the same connected component of $W^{u/s}(x)$, then the arc in $W^{u/s}(x)$ which connects $z$ with $z'$ is denoted by 
$[z,z']^{u/s}$. When  end points of such an arc are deleted we will denoted 
the 
remaining arc by $(z,z']^{u/s}$, $(z,z')^{u/s}$, etc.

$Q_1\asymp Q_2$ means that $C^{-1}\le Q_1/Q_2\le C$, where $C>0$ is 
an absolute constant or depending on, say $F$.

\bigskip

For the reader's convenience, more special notations are collected in the Nomenclature.

\section{Local stable manifolds}\label{secstabman}

\begin{lem}\label{pullb} Let  $U, U',V'\subset \Omega_h$ with  
$\overline{U'}\subset V'$.  Assume, $U'\subset \Omega^v$. There exists $C>0$ such that the following holds.
If $F\in \HH_\Omega(\overline{\epsilon})$, 
$F(x,y)=(f(x)-\epsilon(x,y),x)$, and  
$
f:V'\to f(V')
$
is univalent with 
$$
f(U')\supset \overline{U}
$$
then  for every $A>0$ there exists
$\overline{\epsilon}>0$ such that the following holds. If 
$$
\phi: \Omega^v\to U
$$ 
with
$$
|D\phi|\le A\cdot \overline{\epsilon}
$$
then the  preimage
$F^{-1}(\graph(\phi))\cap (U'\times \Omega^v)$
is the graph of some $\psi:\Omega^v\to U'$ with
$$
|D\psi|\le C\cdot \overline{\epsilon}.
$$
\end{lem}

\begin{rem} The domains $\Omega^h, \Omega^v$ are neighborhoods of $[-1,1]$. In the applications of Lemma \ref{pullb} the domains $U$ and $U'$ will be small neighborhoods of points in $[-1,1]$. Although formally we have that $U'\subset \Omega^h$ we can assume in the applications that 
$U'\subset \Omega^v$.
\end{rem}

\begin{proof} First we will show that for any given $y\in \Omega^v$ there exists a unique $x\in U'$ and $y'\in \Omega^v$ such that
$$
F(x,y)=(\phi(y'),y')\in \graph(\phi).
$$
Finding such an $x\in U'$ is equivalent to solving
\begin{equation}\label{equa}
\phi(x)=f(x)-\epsilon(x,y)\equiv \phi_y(x). 
\end{equation}
This equation is consistent because, $U'\subset \Omega^v=\Domain(\phi)$.
Now, $\phi$ is a strong contraction when $\overline{\epsilon}$ is small, 
$|D\phi|\le A\cdot \overline{\epsilon}$. The map $f|U'$ is univalent. So, 
for $\overline{\epsilon}$ small enough, the map
$$
\phi_y: U'\to \phi_y(U')\supset \overline{U}
$$
is univalent and 
$$
\phi_y^{-1}\circ \phi: U' \to U'
$$                                                          
is a well defined contraction. We used again that 
$U'\subset \Omega^v=\Domain(\phi)$. The unique fixed point of this map is the point 
$x\in U'$ which solves the equation (\ref{equa}). We proved that the set $F^{-1}(\graph(\phi))\cap (U'\times \Omega^v)$ is the graph of some $\psi:\Omega^v\to U'$.

Left is to estimate the derivative of $\psi$. Differentiate 
$\phi(x)=f(x)-\epsilon(x,y)$ with respect to $y$. This gives the following expression 
$$
D\psi(y)=-\frac{\frac{\partial \epsilon}{\partial y}(x,y)}
               {D\phi(x)-Df(x)+\frac{\partial \epsilon}{\partial x}(x,y)}.
$$
There is a lower bound on $|Df(x)|\ge D>0$, $x\in U'$. Furthermore, the partial derivatives of $\epsilon$ are of the order $\overline{\epsilon}$. So, for 
$\overline{\epsilon}$ small enough, we get 
$$
|D\psi(y)|\le C \cdot \overline{\epsilon}.
$$  
\end{proof}

In the sequel of this section we will fix a $F\in \HH^n_\Omega(\overline{\epsilon})$, $n\ge 1$. The domain of $R^nF$ is
$$
\Omega_n=\Omega^h_n\times \Omega^v_n,
$$
where $[-1,1]\subset\Omega^h_n$. Furthermore, 
$R^nF(x,y)=(f_n(x)-\epsilon_n(x,y), x)$. 

Let $\hat{\beta}_n\in \Omega_{n-1}$ be the saddle point of $R^{n-1}F$
 which is of flip type, it has two negative eigenvalues. The connected component of its stable manifold which contains $\hat{\beta}_n$ is denoted by
$
W^s_{\loc}(\hat{\beta}_n).
$
This set is called the {\it local stable manifold} of $\hat{\beta}_n$.

The point $\hat{\beta_n}\in \Omega_n$ corresponds to a periodic point of the original map. Namely, $\beta_n=\Phi^n_0(\hat{\beta_n})\in \Omega_0$. Objects with a hat are in the domain of a renormalization. The corresponding object in the domain of the original map will have no hat.

\begin{lem}\label{locstabelWtilde} For $\overline{\epsilon}>0$ small enough, the local stable manifold
of the point $\hat{\beta}_n$ is the  graph of a function 
$\hat{\psi}_n:\Omega^v_{n-1}\to \Omega^h_{n-1}$ with
$$
|D\hat{\psi}_n|=O(\overline{\epsilon}^{2^{n-1}}).
$$
\end{lem}

\begin{proof} The map $f_{n-1}$ lies in a compact family which is determined
 by 
$\Omega$. This implies that for some $\delta>0$ and $D>0$ we have the 
following. Let $U_n'=U_n=B_{\delta}(\pi_1(\hat{\beta}_n))$ and 
$V_n=B_{2\delta}(\pi_1(\hat{\beta}_n))$. Then
$$
|Df_{n-1}(x)|\ge D>1,
$$
for $x\in U'_n$.

Consider the family of graphs of  the following  functions: 
$$
\mathcal{G}_K=\{ \phi:\Omega^v_{n-1}\to \Omega^h_{n-1}| \phi(\pi_2(\hat{\beta}_n))=
\pi_1(\hat{\beta}_n), |D\phi|\le K\cdot \overline{\epsilon}^{2^{n-1}}\}.
$$
Notice, for $\phi\in \mathcal{G}_K$ we have 
$$
\phi(\Omega^v_{n-1})\subset U_n.
$$
This follows from Lemma \ref{sizedomain}. Namely,
$$
\diam(\phi(\Omega^v_{n-1}))\le K\cdot \overline{\epsilon}^{2^{n-1}}
\cdot \diam(\Omega^v_{n-1})\le  K\cdot \overline{\epsilon}^{2^{n-1}}
\cdot C\cdot r^{n}<\delta.
$$
We can apply 
Lemma \ref{pullb} which  says that, for $\overline{\epsilon}$ small enough, the connected component of $(R^{n-1}F)^{-1}(\graph(\phi))$ containing $\hat{\beta}_n$ is the graph of some function $\psi$. It also says that if we take $K>0$ large enough we have $\psi\in \mathcal{G}_K$. This observation defines the {\it graph transform}
$
\TT:\mathcal{G}_K\to \mathcal{G}_K
$
with
$$
\TT:\phi\mapsto \psi.
$$
The special form of H\'enon-like maps allows us to define the graph transform for 
(global) graphs of $\phi:\Omega^v_{n-1} \to \Omega^h_{n-1}$. Because 
$|Df_{n-1}|\ge D>1$, $f_{n-1}$ is expanding, and $|\epsilon_{n-1}|\le 
\overline{\epsilon}^{2^{n-1}}$ we can use the usual technique to show that 
this graph transform contracts the $C^0$ distance on $\mathcal{G}_K$. The unique fixed point is $W^s_{\text{loc}}(\hat{\beta}_n)\in \mathcal{G}_K$. In particular, it is the graph of a function $\hat{\psi}_n\in \mathcal{G}_K$:
$$
|D\hat{\psi}_n|\le K\cdot \overline{\epsilon}^{2^{n-1}}.
$$
\end{proof}

The map $R^{n-1}F$ is renormalizable. It has two fixed points:    
$\hat{\beta}_n=\beta_1(R^{n-1}F)$,
 which is of flip type, and $\beta_0(R^{n-1}F)$ which has two positive 
eigenvalues.
Let $\hat{p}^n_0\in W^u(\beta_0(R^{n-1}F))$ be such that the curve
 $[\beta_0(R^{n-1}F),\hat{p}^n_0]\subset W^u(\beta_0(R^{n-1}F))$ intersects 
$W^s_{\loc}(\hat{\beta}_n)$ only in $\hat{p}^n_0$.
 Let
$$
\hat{p}^n_i=(R^{n-1}F)^i(\hat{p}^n_0), i\in \mathbb{Z}.
$$
The {\it extended local stable manifold} of $R^{n-1}F$ consists of four curves contained in $W^s(\hat{\beta}_n)$,
$$
\hat{M}^n=[\hat{p}^n_0,\hat{\beta}_n]^s\cup 
(R^{n-1}F)^{-1}([\hat{p}^n_0,\hat{\beta}_n]^s)\cup
(R^{n-1}F)^{-2}([\hat{p}^n_0,\hat{\beta}_n]^s),
$$
where $[\hat{p}^n_0,\hat{\beta}_n]^s\subset W^s(\beta_n)$ is the curve which connects $\hat{p}^n_0$ with $\hat{\beta}_n$.
Lemma \ref{pullb} and Lemma \ref{locstabelWtilde} imply: 

\begin{lem}\label{Wslocexttilde} If  $\overline{\eps}>0$ is small enough then 
$\hat{M}^n\cap (\Omega_{n-1}\cap \mathbb{R}^2)$ consists of four curves,
\begin{enumerate}
\item $\hat{M}^n_{-2}\ni \hat{p}^n_{-2}$,
\item $\hat{M}^n_{-1}\ni \hat{p}^n_{-1}$,
\item $\hat{M}^n_0=W^s_{\loc}(\hat{\beta}_{n})\ni \hat{\beta}_n$,
\item $\hat{M}^n_1$, $W^u(\beta_0(R^{n-1}F))\cap \hat{M}^n_1=\emptyset$ .
\end{enumerate}
These curves are contained in graphs of functions.
These functions, 
denoted by $\hat{M}^n_i:\Omega^v_{n-1}\rightarrow \Omega^h_{n-1}$, have the property that
$\graph(\hat{M}^n_i)\subset W^s(\hat{\beta}_n)$ and 
$$
|D\hat{M}^n_i|=O(\overline{\epsilon}^{2^{n-1}}), \text{   } i=-2,-1,0,1.
$$
\end{lem}

\begin{rem} For maps $F\in \II_\Omega(\overline{\epsilon})$ the bounds on the derivatives in the previous Lemma can be replaced by $O(b_F^{2^{n-1}})$, where $b_F$ is the average Jacobian.
\end{rem}

The map $\Phi^n_0: \Omega_n\to \Omega$ is the coordinate change which 
conjugates $R^nF$ to $F^{2^{n}}|A_{n}$, see Equation (\ref{Phi}). 
 Let
$$
\beta_n=\Phi^{n-1}_0(\hat{\beta}_n)\in \Omega,
$$
$$
p^n_i=\Phi^{n-1}_0(\hat{p}^n_i), i\in \mathbb{Z},
$$
$$
M^n_i=\Phi^{n-1}_0(\hat{M}^n_i), i=-2,-1,0,1,
$$
and
$$
M^n=\Phi^{n-1}_0(\hat{M}^n).
$$

Define the domain $D_1=D_1(F)\subset B$ to be the closed disc bounded by two arcs
$\partial^s=\partial^s(F)\subset W^s(\beta_1(F))$ and
$\partial^u=\partial^u(F)\subset W^u(\beta_0(F))$ whose boundary points are
 $p^n_0$ and  $p^n_1$. Let
$$
D_n=\Phi_0^{n-1}(D_1(R^{n-1}F)).
$$
Notice,
$$
F^2(D_1)\subset D_1.
$$
So
$$
F^{2^n}(D_n)\subset D_n .
$$
The map $F^{2^n}|D_n$ is called the 
{\it preferred $n^{th}$-prerenormalization}.
Finally, 
$$
\partial D_n=\partial_n^s\cup \partial^u_n,
$$
where
$$
\partial^{u,s}_n=\Phi^{n-1}_0(\partial^{u,s}(R^{n-1}F)).
$$
Observe,  
$$
\{\tau_F\}=\bigcap_{n\ge 0} D_n,
$$
which holds because $D_n\subset B_{v^{n}}$.

\begin{figure}[htbp]
\begin{center}
\psfrag{Bn-1}[c][c] [0.7] [0] {\Large $\beta_{n-1}$}
\psfrag{Bn}[c][c] [0.7] [0] {\Large $\beta_n$} 

\psfrag{F}[c][c] [0.7] [0]{\Large $F$} 
\psfrag{B}[c][c] [0.7] [0]{\Large $B$} 
\psfrag{p0}[c][c] [0.7] [0] {\Large $p^n_0$}
\psfrag{p1}[c][c] [0.7] [0] {\Large $p^n_1$} 
\psfrag{p2}[c][c] [0.7] [0] {\Large $p^n_2$} 
\psfrag{p-1}[c][c] [0.7] [0] {\Large $p^n_{-1}$} 
\psfrag{p-2}[c][c] [0.7] [0] {\Large $p^n_{-2}$} 

\psfrag{Wsn-1}[c][c][0.7] [0] {\Large $W^s(\beta_{n-1})$} 
\psfrag{wsn}[c][c][0.7] [0] {\Large $W^s_\loc(\beta_{n})$} 
\psfrag{Wun-1}[c][c][0.7] [0] {\Large $W^u(\beta_{n-1})$}

\psfrag{M-2}[c][c] [0.7] [0] {\Large $M^n_{-2}$} 
\psfrag{M-1}[c][c] [0.7] [0] {\Large $M^n_{-1}$} 
\psfrag{M1}[c][c] [0.7] [0]  {\Large$M^n_{1}$} 

\psfrag{1}[c][c] [0.7] [0] {\Large $1$} 
\psfrag{sigma}[c][c] [0.7] [0] {\Large $\sigma^{2n}$} 

\pichere{0.9}{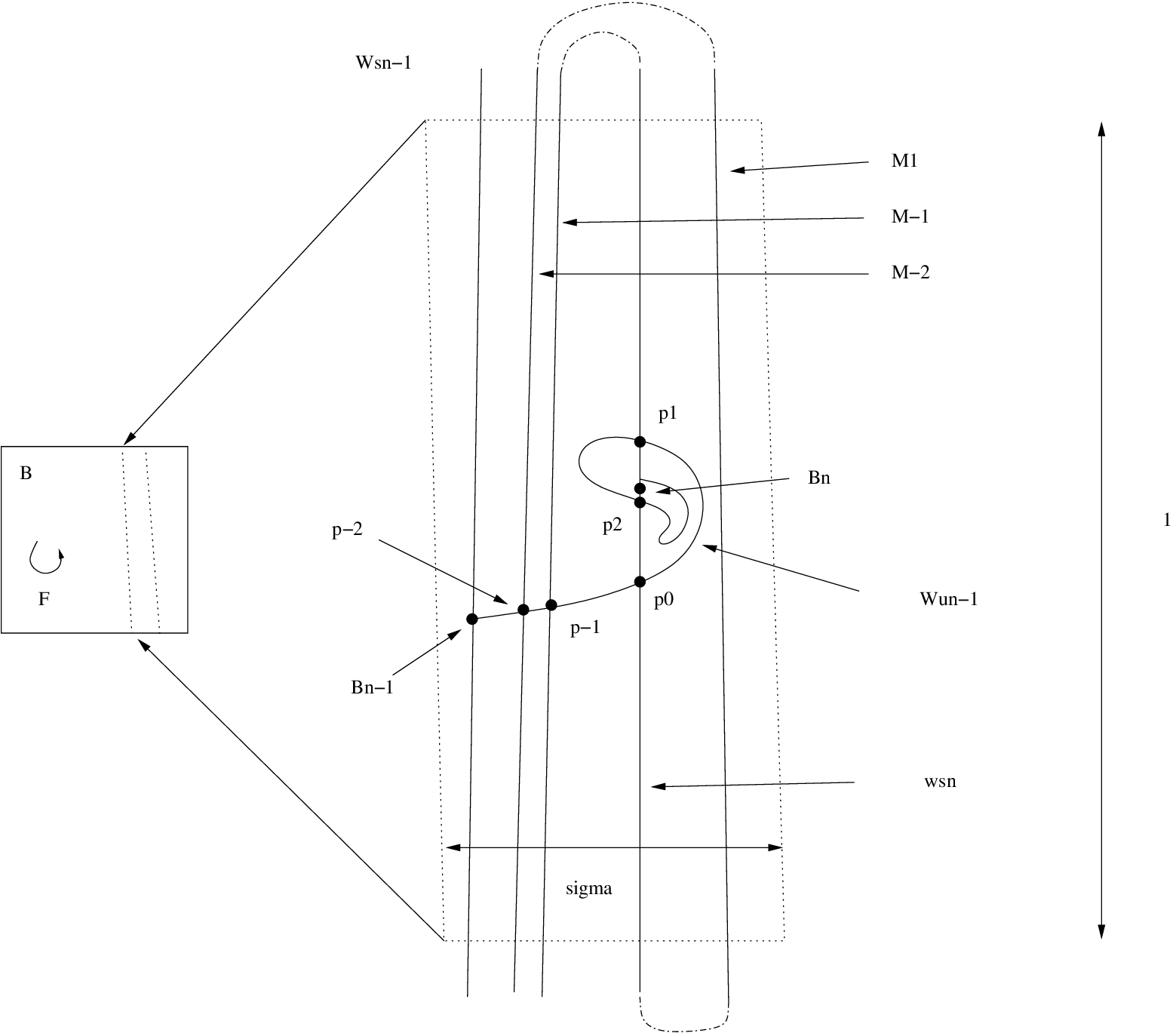} \caption{Extended local
stable manifold} \label{extendedlocstabmani}
\end{center}
\end{figure}

\begin{rem} The box $A_n$, illustrated in 
Figure \ref{extendedlocstabmani}, 
has unit height and exponential small horizontal width: each horizontal slice 
is proportional to $\sigma^{2n}$.
\end{rem}

\begin{rem}\label{disclaim} The upper index of the points $p^n_i$ and $\hat{p}^n_i$ will be omitted when no ambiguity is possible. See for example Figure \ref{saddlereg}.
\end{rem}

\begin{lem}\label{projection} For every $C>0$ there exists $K>0$, independent of $N\ge 1$,  such that the following holds. Let $\hat{\phi}:\Omega^v_n\to \Omega^h_n$ with
$$
|D\hat{\phi}|\le C\cdot \overline{\epsilon}^{2^n}.
$$
Then $\Phi^n_0(\graph(\hat{\phi}))$ is the graph of some 
$\phi:\Omega^v\to \Omega^h$ with
$$
|D\phi|\le K\cdot \overline{\epsilon}.
$$
\end{lem}

\begin{proof} In \S 2, the preliminaries, we introduced the coordinate change which conjugates $R^nF$ with $F^{2^n}|A_n$. Recall that   
$$
\Phi^n_0=\Phi^1_0\circ \Phi^2_1\circ \cdots \circ \Phi^n_{n-1},
$$
where each $\Phi^{k+1}_k$ conjugates $R^{k+1}F$ to the restriction 
$(R^kF)^2|A(R^kF)$.
These conjugations were constructed in such a way that they map horizontal
 lines into horizontal lines. Each map
$$
\pi_2\circ \Phi^{k+1}_k:\Omega_{k+1}\to \Omega^v_k
$$
is onto. Moreover, $\pi_2\circ \phi^{k+1}_k(x,y)$ is independent of $x$ and affine in $y$.
This implies that  $\Phi^n_k(\graph(\hat{\phi}))$ is the
 graph of some $\phi_k:\Omega^v\to \Omega^h$. In particular,  
$\Phi^n_0(\graph(\hat{\phi}))$ is the
 graph of some $\phi:\Omega^v\to \Omega^h$.

Observe that the graph of $\phi_k$ stays away from $x=c_k$, $c_k$ is the critical point of $f_k$. In particular, there exists $D>0$ such that
$$
|Df_k(x)|\ge D
$$
for every $(x,y)\in \graph(\phi_k)$.
                    
The coordinate change $\Phi^{k+1}_k$ is a composition of two maps, see  the preliminaries in \S 2, 
$$
\Lambda_k\circ H_k\equiv (\Phi^{k+1}_k)^{-1},
$$
where                                    
$$
H_k(x,y)=(f_k(x)-\epsilon_k(x,y), y).
$$
We will estimate $|D\phi_k|$ inductively. Let $K_k>0$ be minimal such that  $|D\phi_{k}|\le K_{k}\cdot 
\overline{\epsilon}^{2^{k}}$. In particular, $K_n\le C$. Assume
 $|D\phi_{k+1}|\le K_{k+1}\cdot \overline{\epsilon}^{2^{k+1}}$. 
Choose a point $(x,y)\in \graph(\phi_k)$ and let 
$\Phi^{k+1}_k(x',y')=(x,y)$ with $(x',y')\in \graph(\phi_{k+1})$. 
Take a tangent vector $(D\phi_k(y)z,z)$ to the graph of $\phi_k$. Then
$$
D\Lambda_k\circ DH_k(x,y) (D\phi_k(y)z,z)=(D\phi_{k+1}(y')z',z').
$$
$D\Lambda_k$ is conformal which implies  that for some $s\in \mathbb{R}$ we have  $z'=s\cdot z$. Hence,
$$
(Df_k(x)-\frac{\partial \epsilon_k(x,y)}{\partial x})\cdot D\phi_k(y)-
\frac{\partial \epsilon_k(x,y)}{\partial y}=D\phi_{k+1}(y').
$$ 
 Using $|Df_k(x)|\ge D$ and the above equation we get, for $\overline{\epsilon}$ small enough, constants $A_0, A_1>0$, independent of $N$, such that
$$
|D\phi_k|\le A_0\cdot |D\phi_{k+1}| +A_1\cdot \overline{\epsilon}^{2^{k}}.
$$
Hence,
$$
K_k\le A_0\cdot \overline{\epsilon}^{2^{k}}\cdot K_{k+1}+A_1.
$$
This implies that there is a uniform bound $K\ge K_k$, $k\ge 1$.
\end{proof}

\begin{prop}\label{Wslocext} If $F\in \HH^n_\Omega(\overline{\eps})$,
with $\overline{\eps}>0$ small enough, then $M^n\cap B$
consists of four curves,
\begin{enumerate}
\item $M^n_{-2}\ni p^n_{-2}$,
\item $M^n_{-1}\ni p^n_{-1}$,
\item $M^n_0=W^s_{\loc}(\beta_{n})\ni \beta_n$,
\item $M^n_1$, $M^n_1\cap W^u(\beta_{n-1})=\emptyset$.
\end{enumerate}
These curves are contained in graphs of functions.
These functions, 
denoted by $M^n_i:\Omega^v \rightarrow \Omega^h$, have the property that
$\graph(M^n_i)\subset W^s(\beta_n)$ and satisfy
$$
|DM^n_i|=O(\overline{\epsilon}), \text{   } i=-2,-1,0,1.
$$
Letting  $z_n$ be the intersection point of $M^n_1$ with the horizontal line
through $\tau_F$, we have:  $\pi_1(z_n)>\pi_1(\tau_F)$ and
$$
|z_n-\tau_F|\asymp \sigma^{2n}.
$$
\end{prop}

\begin{proof} The first part of the Proposition follows by applying 
Lemma \ref{Wslocexttilde} and Lemma \ref{projection}. The second, which
 describes the distance from $\tau_F$ to $z_n$, is an immediate consequence of
 the results in \S 7.2 of \cite{CLM}.
\end{proof}

Let $W_n\subset B$ be the real domain bounded by $M^n_{0}$ and $M^n_1$. The 
domain $W_n$ is a topologically defined variation of $A_n$ (the restriction
 $F^{2^n}|A_n$ is conjugate to the  $n^{th}$-renormalization of $F$, 
see \S 2).  Note,
$$
\beta_n\in W_{n}
$$
and
$$
\beta_n'\equiv F^{2^{n-1}}(\beta_n)\in W_{n-1}.
$$
The connected component of the stable manifold $W^s(\beta'_n)\cap B$ which contains $\beta'_n$ is called the {\it local stable manifold} of $\beta'_n$, denoted by $W^s_{\loc}(\beta'_n)$.

\begin{prop}\label{Wsloc'} For $\overline{\epsilon}>0$ small enough the local stable manifold $W^s_{\loc}(\beta'_n)$ is the graph of a function $\phi:[-1,1]
\to [-1,1]$ and
$$
|D\phi|=O(\overline{\epsilon}).
$$
\end{prop}

\begin{proof} The maximal $k\le n$ for which $R^kF$ has a periodic point corresponding to $\beta'_n$ is $k=n-2$. Namely,
$$
\beta'_n=\Phi^{n-2}_0(R^{n-2}F(\Phi^{n-1}_{n-2}(\hat{\beta}_n))).
$$
Let 
$$
G_1=W^s_{\loc}(\hat{\beta}_n)\subset \Domain(F_{n-1}),
$$
$$
G_2=W^s_{\loc}(\Phi^{n-1}_{n-2}(\hat{\beta}_n))\subset \Domain(F_{n-2}),
$$
and
$$
G_3=W^s_{\loc}(F_{n-2}(\Phi^{n-1}_{n-2}(\hat{\beta}_n)))\subset 
\Domain(F_{n-2}).
$$
Observe,
$$
W^s_{\loc}(\hat{\beta}_n)=\Phi^{n-2}_{0}(G_3).
$$
Lemma \ref{Wslocexttilde}(3) says that $G_1$ is the graph of a function with small derivative. Then Lemma \ref{projection} implies that also $G_2$ is the graph of a function with small derivative. Use Lemma \ref{pullb} to show that $G_3$
is the graph of a function with small derivative. Finally,
we get that $W^s_{\loc}(\beta'_n)$ is the graph of a function with small angle because of Lemma \ref{projection}.
\end{proof}

Consider an infinitely renormalizable $F\in \II_\Omega(\overline{\epsilon})$, with $\overline{\epsilon}>0$ small enough.  Observe, the tip of $F$ satisfies
$$
\tau_F\in W_{n},
$$
for all $n\ge 1$. Also, notice that $W_n\subset W_{n-1}$. Let 
$$
W^s_{\loc}(\tau_F)=\bigcap_{n\ge 1} W_n.
$$
This set is called the {\it local stable manifold of the tip}.

\begin{prop}\label{Wstip} For $\overline{\epsilon}>0$ small enough, the local 
stable manifold $W^s_{\loc}(\tau_F)$ is the graph of an analytic function 
$\phi:[-1,1] \to [-1,1]$ and
$$
|D\phi|=O(\overline{\epsilon}).
$$
\end{prop}

\begin{proof} Let $\phi_n:\Omega^v\to \Omega^h$ be the function whose graph is $M^n_0=W^s_{\loc}(\beta_n)$. 
This graph contains the left boundary $M^n_0$ of 
$W_n$. According to Proposition \ref{Wslocext} we have a uniform bound on 
$\phi_n:\Omega^v\to \Omega^h$. 
Normality implies $\phi_n\to \phi$. The analyticity of $\phi$ follows.  
The real slice of the graph of $\phi$ is  $W^s_{\loc}(\tau_F)$. 
\end{proof}

The characteristic exponents of the invariant measure on $\OO_F$ 
are $0$ and 
$\ln b_F$, see Theorem 6.3 of \cite{CLM}.
 The next Proposition states that the local stable
 manifold of the tip is indeed part of its stable manifold. 
However, we do not know the actual value of the stable exponent at the tip.

\begin{prop}\label{stablesettip}  For $\overline{\epsilon}>0$ small enough,
$$
\diam(F^t(W^s_{\loc}(\tau_F)))=O((\sqrt{b_F})^t),
$$ 
$t\ge 0$. Also, if $v\in T_{\tau_F}W^s_{\loc}(\tau_F)$ is a non-zero 
tangent vector then
$$
\limsup_{t\to \infty} \frac{1}{t}\ln|DF^t(\tau_F)v|\le \frac12 \ln b_F.
$$
\end{prop}

\begin{proof} Choose $t\ge 0$ and let $n\ge 0$ be maximal such that 
$$
2^n\le t.
$$
For $t\ge 0$ large we have a precise description of $R^nF$, see Theorem 7.9 
of \cite{CLM} or equation (\ref{univ}). In particular, by using this result, 
we get
$$
\diam(R^nF(W^s_{\loc}(\tau_{R^nF})))=O(b_F^{2^n}\cdot \diam(\Omega^v_n)).
$$
Remark  \ref{asympsizedomain} gives
$$
\diam(\Omega^v_n))\asymp \frac{1}{\sigma^n}.
$$
 
The renormalization microscope described in \S 5 of \cite{CLM} can be used to construct $F^t(W^s_{\loc}(\tau_F))$. Recall,
$$
\phi^k_c=\Phi^{k+1}_k,
$$
and 
$$
\phi^k_v=R^kF\circ\Phi^{k+1}_k,
$$
with $k\ge 0$. The renormalization microscope was made such that there exists a sequence $\omega_k\in \{v,c\}$, $k=1,2,\dots, n$, such that
$$
F^t(W^s_{\loc}(\tau_F))=\phi^1_{\omega_1}\circ \phi^2_{\omega_2}\circ \dots\circ
\phi^n_{\omega_n}(R^nF(W^s_{\loc}(\tau_{R^nF}))).
$$
Lemma 5.1 of  \cite{CLM} says that each $\phi^k_\omega$ is a contraction. In fact,
$$
|D(\phi^1_{\omega_1}\circ \phi^2_{\omega_2}\circ \dots\circ
\phi^n_{\omega_n})|\le C\sigma^n.
$$
 Hence, 
$$
\diam(F^t(W^s_{\loc}(\tau_F)))
 =   O((\sqrt{b_F})^t),
$$
where we used 
that $2^n>\frac12 t$. The proof of the infinitesimal version is the 
same.
\end{proof}

\section{Laminar structure of the attractor}\label{Lamstruc}

For a map $F: B\ra \R^2$, the set $\bigcap_{k\ge 0} F^k(B)$ is called the
 {\it global attracting set}  of $F$. It is the maximal backward invariant  
subset of $B$. For a discussion on the concept of attractor see \cite{Mi1} 
and \cite{Mi2}.

For an infinitely renormalizable H\'enon-like map $F$, let
$$
\AAA_F=\OO_F \cup \bigcup_{n\ge 0} W^u(\bbe_n).
$$
For a  map $F\in \HH^n_\Omega(\overline{\eps})$,
with $\overline{\eps}>0$ small enough, let $B_0\subset B$ be 
the connected component of $B\setminus W^s_\loc(\beta_0)$ which 
contains $\beta_1$. The set $\overline{B}_0$ consists of {\it non-escaping} 
points.

\begin{rem} The set $\AAA_F$ is backward invariant. It is also essentially  
forward invariant. Notice, 
$$
F(\AAA_F\cap \overline{B}_0)=\AAA_F\cap \overline{B}_0.
$$
However, $\AAA_F\cap (B\setminus\overline{B}_0)$ is a piece of the unstable manifold of $\beta_0$ which is mapped strictly over itself with some points outside of  $\AAA_F\cap (B\setminus\overline{B}_0)$. 
\end{rem}

The following result shows that, in fact,  this set $\AAA_F$ is the global attracting set:

\begin{thm}\label{attrac}
Given  an infinitely renormalizable H\'enon-like map $F\in \II_\Omega(\overline{\eps})$
with $\overline{\eps}>0$ small enough,
we have: 
$$
\AAA_F=\bigcap_{k\ge 0} F^k(B)=\overline{W^u(\beta_0)}.
$$
Furthermore, for every point $x\in B_0$ either $x\in W^s(\bbe_n)$ for some $n\ge 0$ 
or $ \omega(x)=\OO_F.$
The non-wandering set of $F$  is
$
\Omega_F=\PP_F \cup \OO_F.
$
\end{thm}

The second part of Theorem \ref{attrac}, 
concerning the limit sets of points and the non-wandering set, was already 
proved in  \cite{GST}.
The proof of this Theorem needs some preparation. 
For a  map $F\in \HH^n_\Omega(\overline{\eps})$,
with $\overline{\eps}>0$ small enough, define the {\it $n^{th}$-trapping region} of $F$ as
$$
\Trap_n=\Orb(D_n).
$$
Note that
\begin{equation}\label{intertrap}
\OO_F \cup \bigcup_{k\ge n} \bbe_n \subset \inter(\Trap_n).
\end{equation}

\begin{lem}\label{trapping} Let $\overline{\eps}>0$ small enough and   
$F\in \HH_\Omega(\overline{\eps})$ be a renormalizable map.  
For every $x\in B_0$, there exists $k\ge 1$ such that
$$
F^k(x)\in D_1\subset \Trap_1.
$$
Let $U\supset \AAA_F$ be a neighborhood. 
Then there exists $k_0\ge 1$ such that for $ k\ge k_0$
$$
F^k(\overline{B_0})\subset U.
$$
\end{lem}

\begin{proof} Let $\overline{\eps}>0$ be small enough such that Proposition
~\ref{Wslocext} applies. Then we 
can divide the domain $B_0$ by cutting it using the 
curves $\graph(M^1_i)$, $i=-2,-1,0,1$, 
see Figure ~\ref{extendedlocstabmani}. Let 
$$
Z_1\cup Z_2\cup Z_3 \cup Z_4 \cup Z_5 =
B_0\setminus \bigcup_{i=-2}^1\graph(M^1_i),
$$
counting the connected components from left to right. 
In particular $D_1\subset Z_4$. The curve in $W^u(\beta_0)$ which connects 
$p^1_{-1}$ with $p^1_0$ is denoted by $[p^1_{-1},p^1_0]^u$. Let 
$$
Z^+_3\cup Z^-_3=Z_3\setminus [p^1_{-1},p^1_0]^u,
$$
be the partition by $[p^1_{-1},p^1_0]^u\subset W^u(\beta_0)$ of 
$Z_3$ in the connected components.
One easily checks the following properties
\begin{enumerate}
\item $F(Z^+_3)\subset D_1$,
\item $F(Z^-_3)\subset Z_4$,
\item $F(Z_4)\subset Z^+_3$,
\item $F(Z_2)\subset Z_3$,
\item $F(Z_5)\subset Z_1\cup Z_2$,
\item for every $x\in Z_1$ there exists $k\ge 1$ such that $F^k(x)\in Z_2$.
\end{enumerate}
The Lemma follows.
\end{proof}

Observe, $\Trap_1\cap W^s(\beta_0)=\emptyset$. This implies the following.

\begin{cor}\label{nohomcli} Let $F\in \HH^n_\Omega(\overline{\eps})$,
with $\overline{\eps}>0$ small enough, then there are no homoclinic orbits connected to $\bbe_k$,
$$
W^s(\bbe_k)\cap W^u(\bbe_k)=\emptyset,
$$
 $k\le n$.
\end{cor}

Let $\Gamma_j$, $j\ge 1$, 
and $\Gamma$ be smooth curves in the plane. 
We say that the $\Gamma_j$ converge to $\Gamma$, $\Gamma_j\to \Gamma$,
if there are smooth parametrisations of these curves such that the corresponding parametrised curves 
converge in the $C^1$-topology.

\begin{lem}\label{lambdalem} Let $F\in \HH^n_\Omega(\overline{\eps})$,
with $\overline{\eps}>0$ small enough, and $\Gamma\subset W^u(\beta_n)$.
 Then there are arcs 
$\Gamma_j\subset W^u(\beta_0)$ and $t_j\to \infty$ such that 
$F^{t_j}(\Gamma_j)\to \Gamma$.
\end{lem}

\begin{proof} 
Note first that it is sufficient to prove the assertion for some arc 
$\Gamma\subset W^u(\beta_n)$ containing $\beta_n$ in its interior
(since $\cup_{k\geq 0} F^k(\Gamma)=W^u(\beta_n)$). 

The proof goes by induction. For $n=1$ the Lemma can be proved as 
follows. As before, let $p_0=p_0^1$ be the first intersection of $W^u(\beta_0)$
with $W^s_\loc(\beta_1)$. The two manifolds intersect transversally. 
If $\Gamma\subset W^u(\beta_1)$ is a curve containing $\beta_1$
then  the $\lambda$-Lemma 
(see Chapter 2 Lemma 7.1 of \cite{dMP}) 
allows us to choose arcs
$$ \Gamma_1\supset\Gamma_2\supset \Gamma_3\supset\dots\ni \{p_0\}$$ and times 
$t_j\to \infty$ such that $F^{t_j}(\Gamma_j)\to \Gamma$.

Assume the Lemma holds for $n-1$. 
Take an arc $\Gamma\subset W^u(\beta_{n})$ containing $\beta_{n}$ in its interior, 
say $\Gamma=\Psi^{n-1}_0(\hat{\Gamma})$ with
$\hat{\Gamma}\subset W^u(\beta_1(R^{n-1}F))$. 
For  $\overline{\eps}>0$ small enough, 
all the renormalizations $R^k F$, $k\le n-1$, belong to the class $\HH_{\Omega'}(\overline{\eps})$
with some $\Om'\subset \Om$. 
In particular, we can apply the base of induction to 
$\beta_0(R^{n-1}F)$ and  $\beta_1(R^{n-1}F)$.  
This gives a sequence of curves 
$\hat{\Gamma}_j\subset W^u(\beta_0(R^{n-1}F))$ and 
$\hat{t}_j\to \infty$ such that 
$$
(R^nF)^{\hat{t}_j}(\hat{\Gamma}_j)\to \hat{\Gamma}\subset W^u(\beta_1(R^nF)).
$$
Now, the induction assumption allows us to approximate the curves 
$\Gamma_j=\Psi_0^n(\hat{\Gamma}_j)\subset W^u(\beta_n)$ by curves from 
$W^u(\beta_0)$ and the Lemma follows.
\end{proof}

\noindent
{\it Proof of Theorem ~\ref{attrac}}. First we will prove that every point converges 
to a periodic point or to the Cantor set. 
Let $x\in B_0$ be a point that does not converge to any periodic orbit. 
                  
According to Lemma ~\ref{trapping} there exists $k_1\ge 1$ such that
$$
F^{k_1}(x)\in D_1(F).
$$
Notice that
$$
D_1(F)\subset B^1_v(F)\subset \Image \phi^1_v,
$$
where the map $\phi^1_v$ is defined in (\ref{phi}).
Now
$$
x_1=(\phi^1_v)^{-1}(F^{k_1}(x))\in B_0(RF),
$$
because the orbit of $x$ does not converge to the periodic orbit of 
$\beta_1(F)$. 
Again, using Lemma ~\ref{trapping}, there is $k_2\ge 1$ such that
$$
F^{k_2}(x_1)\in D_1(RF).
$$
In particular,
$$
\Orb(F^{k_2}(x_1))\subset \Orb(D_1(RF)).
$$
So,
$$
F^{k_1+2k_2}(x)\in D_2\subset\Trap_2.
$$
Note again that $D_1(RF)\subset \Image \phi^2_v$. 
Because $R^nF\in \HH_{\Om'}(\bar\eps)$,             
we are allowed to repeatedly apply Lemma ~\ref{trapping}. 
Hence for every $n\ge 1$ there exists $k\ge 1$ such that
$$
F^k(x)\in \Trap_n.
$$
Thus
$$
\omega(x)=\OO_F.
$$
Obviously,  $\OO_F\cup \PP_F\subset \Omega_F$. Take a point $x\in B$ that does 
not converge to any periodic orbit and is not in the Cantor set $\OO_F$. 
The argument above gives for every $n\ge 1$ a neighborhood $U$ of  $x$  and 
$k_0\ge 1$ such that for $k\ge k_0$
$$
F^k(U)\subset \Trap_n.
$$
For $n\ge 1$ large enough we have $x\notin \Trap_n$. 
Thus, the point is  wandering. 

Let us now consider  a  non-periodic point $x\in W^s(\bbe_n)$. 
According to Lemma ~\ref{trapping} 
there are disjoint neighborhoods $U$ of $x$ and $V\supset \AAA_F^n$, $n\ge 1$,
 such that for $k\ge k_0$
$F^k(U)\subset V$. Thus, the point is wandering. 
This completes the proof of
$$
\Omega_F=\PP_F \cup \OO_F.
$$
Since $\AAA_F$  is backward invariant,
$$
\AAA_F\subset \bigcap_{k\ge 0} F^k(B).
$$
The opposite inclusion is obtained as follows. 
Choose a point $x\in \cap_{k\geq 0}F^k (B)$.
 If $x\in \OO_F$ we have $x\in \AAA_F$.
Assume $x\notin \OO_F$. 
For every $j\ge 0$
we have $F^{-j}(x)$ exists. Let $\alpha(x)\subset B$ consists of all limits of the negative orbit of $x$.
This is a closed forward and backward invariant  
 set. Choose $n\ge 1$ large enough such that
$$
x\notin \Trap_n.
$$
Because, $F(\Trap_n)\subset \Trap_n$ we have for every $j\ge 0$
$$
F^{-j}(x)\notin \Trap_n.
$$
Observe, $\OO_F\subset \inter(\Trap_n)$. So
$$
\alpha(x)\cap \OO_F= \emptyset.
$$ 
Now, the orbit of every point not in $\PP_F$ converges to $\OO_F$. 
Hence,
$$
\alpha(x)\subset \PP_F.
$$
This in turn implies that  $x\in W^u(\bbe_n)$ for some $n\ge 1$. 
This completes the proof of
$$
\AAA_F=\bigcap_{k\ge 0} F^k(B).
$$
The closure of the unstable manifold of $\beta_0$ is backward invariant. Hence,
$$
\overline{W^u(\beta_0)}\subset \bigcap_{k\ge 0} F^k(B)=\AAA_F.
$$
The opposite inclusion is obtained as follows. The stable and unstable 
manifolds are analytic. This implies that there are only countably many 
heteroclinic points. In particular, there are points in $W^u(\beta_0)$ which 
do not converge to any periodic orbit. These points converge to the Cantor set. Hence,
$$
\OO_F\subset \overline{W^u(\beta_0)}.
$$
Lemma ~\ref{lambdalem} implies that for every $n\ge 1$ 
$$
W^u(\beta_n)\subset \overline{W^u(\beta_0)}.
$$
Hence,
$$
\AAA_F=\overline{W^u(\beta_0)}.
$$
\qed

The proof of the following Lemma is the same as the part 
of the proof of Theorem \ref{attrac}
dealing with the non-wandering set and will be omitted.

\begin{lem} \label{OmegaN} Let 
$F\in \II^n_\Omega(\overline{\eps}) $, with 
$\overline{\eps}>0$ small enough. Then
$$
\Omega_F=\PP_F.
$$
\end{lem}

\begin{defn}\label{deflam}
Let 
$F\in \HH^n_\Omega(\overline{\eps})$, with $\overline{\eps}>0$ small.
A point 
$$
z\in W^u(\beta_{k'})\subset \bigcup_{k\le l\le n} W^u(\bbe_l)
$$
 is {\it laminar} if for
any sequence  $z_m\in W^u(\beta_{k_m})$ with $z_m\to z$ and $k\le k_m\le n$
the following holds
$$
T_{z_m}W^u(\beta_{k_m}) \to T_{z}W^u(\beta_{k}).
$$
The attractor of $F\in \HH^n_\Omega(\overline{\eps})$
is called {\it laminar} if every point in
$$
\bigcup_{n\ge 0} W^u(\bbe_n)=\AAA(F)\setminus \OO (F)
$$ 
is laminar.  
\end{defn}

\begin{rem} If $\AAA_F$ is laminar then every point of $\AAA_F\setminus \OO_F$ has a neighborhood which is a $C^{1}$-diffeomorphic image of $(-1,1)\times Q$, where $Q\subset [-1,1]$ is a countable set. It has a local product 
structure. The set  $\AAA_F\setminus \OO_F$ is a {\it match-box} manifold, see \cite{AM}.

Indeed, it requires work
 to show that the neighborhood can be linearized by a $C^{1+\alpha}$-coordinate change. It relies on the linearizability of saddle points, see \cite{H}. More on H\"older laminations can be taken from \cite{PSW} and references therein. 

The transverse sections can be described as follows. Choose a point $x\in W^u(\bbe_n)\subset \AAA_F\setminus \OO_F$ then 
$$
Q=\bigcup_{k=0}^{n} Q_k,
$$
where
\begin{enumerate}
\item $Q_k\cap Q_l=\emptyset$ when $k\ne l$,
\item $Q_k$ is countable and discrete, $k<n$,
\item $Q_n=\{x\}$,
\item for $k<n$
$$
\overline{Q_k}= \bigcup_{l\ge k} Q_l.
$$
\end{enumerate}
Moreover, $Q_k$ accumulates at $Q_{k+1}$ with an asymptotic rate. 
The rate equals $\mu_{k+1}\in (-1,0)$ which is the stable multiplier of 
$\beta_{k+1}$. More precisely, for each $q\in Q_{k+1}$ there exists a neighborhood $U\ni q$ in $Q$ such that $U\cap Q_k=\{q_i\}_{i\ge 1}$ with
$$
\lim_{i\to \infty}\frac{q-q_i}{\mu^i_{k+1}}=C\ne 0.
$$
Note that the set $Q_k$ accumulates from both sides at $q\in Q_{k+1}$.
\end{rem}

\begin{thm}\label{lam} The attractor of an infinitely renormalizable H\'enon
map $F\in \II_\Omega(\overline{\eps})$, with $\overline{\eps}>0$ small enough,
is laminar if there are no heteroclinic tangencies.
\end{thm}

The proof of this Theorem needs some preparation. Let $n\ge 1$ and
$q_1, q_2, q_3\in W^u(\beta_{n-1})$ be the first three intersections, coming from $\beta_{n-1}$ along $W^u(\beta_{n-1})$,
with $W^s_{\loc}(\beta_{n+1})$,  and let $q'_2, q'_3\in W^u(\beta_{n-1})$
be the second and third intersection with
$W^s_{\loc}(\beta'_{n+1})$. Now define the {\it saddle-region $T_n$
of $\beta_n$} the domain containing $\beta_n$ which is bounded by
the following four arcs, $[q_2, q'_2]^u\subset W^u(\beta_{n-1})$,
$[q_3,q'_3]^u\subset W^u(\beta_{n-1})$, $[q_2,q_3]^s\subset
W^s(\beta_{n+1})$ and $[q'_2, q'_3]^s\subset W^s(\beta'_{n+1})$, see
Figure ~\ref{saddlereg}.

\begin{figure}[htbp]
\begin{center}
\psfrag{Bn-1}[c][c] [0.7] [0] {\Large $\beta_{n-1}$}
\psfrag{Bn}[c][c] [0.7] [0] {\Large $\beta_n$}
\psfrag{Bn+1}[c][c] [0.7] [0] {\Large $\beta_{n+1}$}
\psfrag{B'n+1}[c][c] [0.7] [0] {\Large $\beta'_{n+1}$}
\psfrag{Kn-1}[c][c] [0.7] [0] {\Large $K^u_{n-1}$}
\psfrag{p0}[c][c] [0.7] [0] {\Large $p_0$}
\psfrag{p1}[c][c] [0.7] [0] {\Large $p_1$}
\psfrag{p2}[c][c] [0.7] [0] {\Large $p_2$}
\psfrag{q1}[c][c] [0.7] [0] {\Large $q_1$}
\psfrag{q2}[c][c] [0.7] [0] {\Large $q_2$}
\psfrag{q3}[c][c] [0.7] [0] {\Large $q_3$}
\psfrag{q'1}[c][c] [0.7] [0] {\Large $q'_1$}
\psfrag{q'2}[c][c] [0.7] [0] {\Large $q'_2$}
\psfrag{q'3}[c][c] [0.7] [0] {\Large $q'_3$}
\psfrag{Wsn}[c][c] [0.7] [0] {\Large $W^s_\loc(\beta_n)$}
\psfrag{Wsn-1}[c][c] [0.7] [0] {\Large $W^s_\loc(\beta_{n-1})$}
\psfrag{W'sn+1}[c][c] [0.7] [0] {\Large $W^s_\loc(\beta'_{n+1})$}
\psfrag{Wsn+1}[c][c] [0.7] [0] {\Large $W^s_\loc(\beta_{n+1})$}
\psfrag{Wun-1}[c][c] [0.7] [0] {\Large $W^u(\beta_{n-1})$}
\psfrag{Tn}[c][c] [0.7] [0] {\Large $T_{n}$}
\psfrag{Un}[c][c] [0.7] [0] {\Large $U_{n}$}
\psfrag{Sn}[c][c] [0.7] [0] {\Large $S_{n}$}

\pichere{0.8}{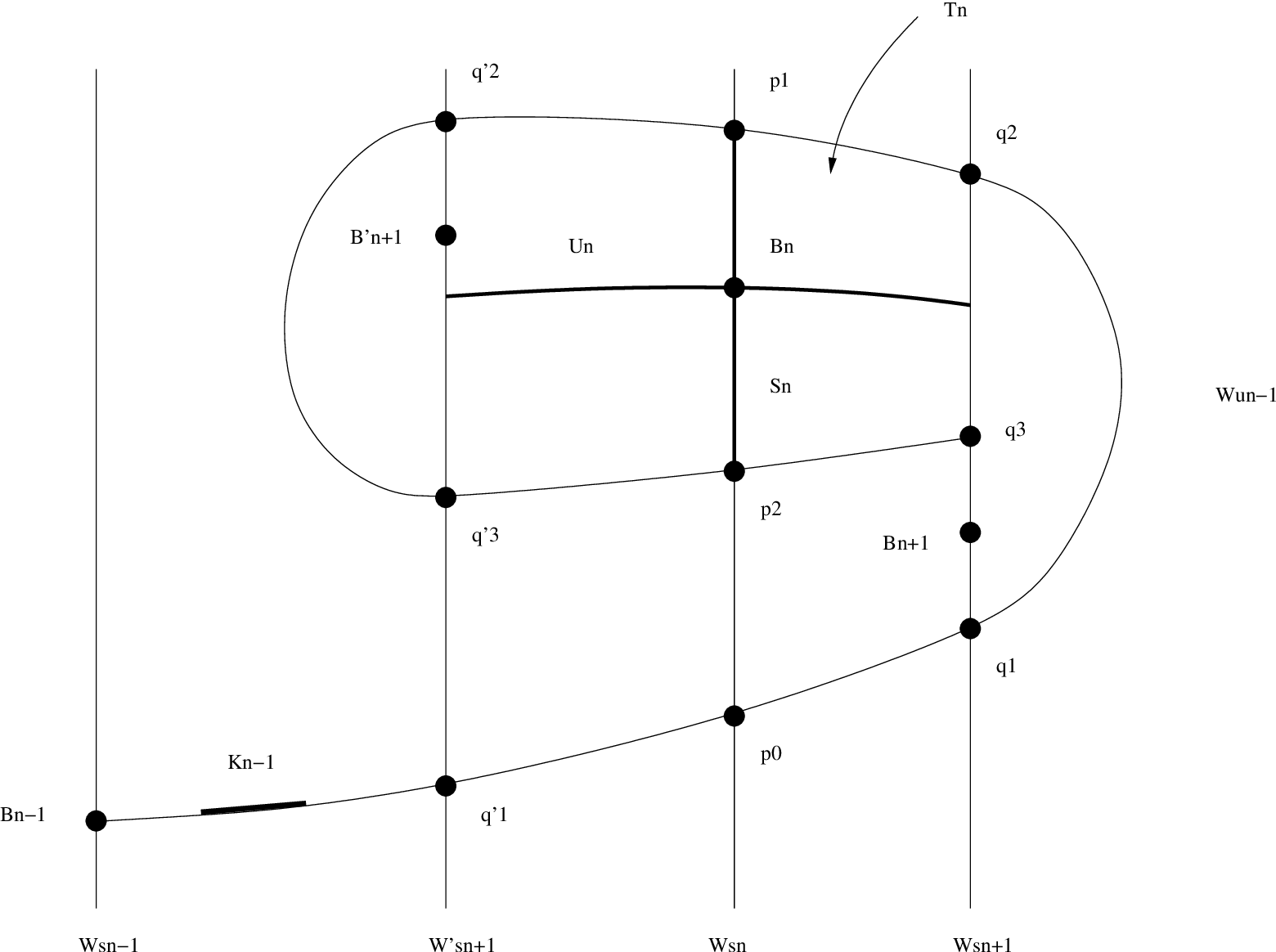}
\caption{Saddle-region of $\beta_n$}
\label{saddlereg}
\end{center}
\end{figure}

Furthermore, let $p_0,p_1\in W^u(\beta_{n-1})$ be the first two
intersections with $W^s_{\loc}(\beta_{n})$. The curve
$[p_0,p_1]^u\subset W^u(\beta_{n-1})$ is a fundamental domain for
$F$ restricted to $W^u(\beta_{n-1})$. Consider the
following fundamental domain
$$
K^u_{n-1}=F^{-2^n}([p_0,p_1]^u).
$$
The connected components of $W^u(\beta_n)\cap T_n$ and
$W^s(\beta_n)\cap T_n$ ( which contain $\beta_n$) are denoted by
$U_n, S_n\subset T_n$. Notice, $S_n\cap \Trap_{n+1}=\emptyset$.

\begin{lem} \label{1D}
$W^u(\beta_n)$ and $W^s(\beta_n)$ are (non-compact) one-dimensional 
embedded manifolds, for each $n\ge 0$.
\end{lem}

\begin{proof} For each $n\ge 1$ we have that the fundamental domain of
$W^u(\beta_{n-1})$ satisfies $F^{2^{n}}(K^u_{n-1})\subset \Trap_n$. Because
$K^u_{n-1} \cap \Trap_n = \emptyset$ we have
$$
\overline{\bigcup_{i\ne 0} F^i(K^u_{n-1} )\setminus K^u_{n-1}}\cap \inter 
K^u_{n-1} = \emptyset.
$$
This implies that $W^u(\beta_{n-1} )$ does not accumulate at itself, it 
is a one-dimensional manifold. 

The proof for the stable manifold is similar, we have to show that 
$W^s(\beta_n)$ does not accumulate at itself. Suppose it does. Then 
 for some $s\in  \inter S_n$ there exists a sequence 
$W^s(\beta_n)\setminus S_n\ni s_k\to s$. We may assume that each 
$s_k\in T_n$. Apply Lemma \ref{trapping} to $R^nF$ and we get 
$$
\beta_n\in \omega(s_k)\subset \Trap_{n+1}.
$$
Contradiction.
\end{proof}

\begin{rem} The unstable manifolds are connected. However, the stable manifold
 of a periodic points consists of countably many closed curves.
\end{rem}

\begin{lem}\label{manifinsaddle}
Let $F\in \HH^{n}_\Omega(\overline{\eps})$, with 
$\overline{\eps}>0$ small enough. Then 
$$
W^s(\beta_n)\cap T_n=S_n.
$$
\end{lem}

\begin{proof} Let $Z_3\subset B$ be the open domain bounded by $M^n_{-1}$ and 
$W^s_\loc(\beta_n)$ and $Z_4\subset B$ be the open domain bounded by $M^n_1$ and $W^s_\loc(\beta_n)$, see Figure \ref{extendedlocstabmani}. Recall,
$$
F^{2^n}(Z_3)\subset Z_4
$$
and 
$$
F^{2^n}(Z_4)\subset Z_3.
$$
Hence, no point in $Z_3\cup Z_4$ will ever enter $W^s_\loc(\beta_n)$. This 
means
$$
W^s(\beta_n)\cap (Z_3\cup Z_4)=\emptyset.
$$
Finally, observe that $T_n\subset Z_3\cup W^s_\loc \cup Z_4$.
 The Lemma follows.
\end{proof}

\begin{lem}\label{1Dclos} Let $F\in \HH^{n+1}_\Omega(\overline{\eps})$, with 
$\overline{\eps}>0$ small enough. Then 
$$
\overline{W^u(\beta_n)}\setminus W^u(\beta_n)=
\overline{W^u(\beta_{n+1})}\cup\overline{W^u(\beta'_{n+1})} . 
$$ 
\end{lem} 

\begin{proof} Applying Theorem ~\ref{attrac} to $RF$, we obtain:    
\begin{equation}\label{Abo}
\begin{aligned}
\AAA_F\setminus W^u(\beta_0)&= 
\{(\AAA_F\cap D_1)\cup (\AAA_F\cap F(D_1))\}\setminus W^u(\beta_0)\\
&= (\overline{W^u(\beta_1)}\cap D_1)
\cup F(\overline{W^u(\beta_1)}\cap D_1)\\
&=\overline{W^u(\beta_1)}.
\end{aligned}
\end{equation}
Figure ~\ref{saddlereg} might be useful in the following argument.
Apply Theorem ~\ref{attrac} and equation ~(\ref{Abo}) to $R^nF$ and we obtain
$$
(\overline{W^u(\beta_n)}\setminus W^u(\beta_n))\cap D_n=
\overline{W^u(\beta_{n+1})}.
$$
Observe that
$$
\overline{W^u(\beta_n)}\subset D_n \cup F^{2^{n-1}}(D_n).
$$
Hence,
$$
\begin{aligned}
&\overline{W^u(\beta_n)}\setminus W^u(\beta_n)=\\
&((\overline{W^u(\beta_n)}\setminus W^u(\beta_n))\cap D_n)\cup 
 ((\overline{W^u(\beta_n)}\setminus W^u(\beta_n))\cap F^{2^{n-1}}(D_n))=\\
&((\overline{W^u(\beta_n)}\setminus W^u(\beta_n))\cap D_n)\cup 
 F^{2^{n-1}}((\overline{W^u(\beta_n)}\setminus W^u(\beta_n)\cap D_n)=\\
&\overline{W^u(\beta_{n+1})}\cup \overline{W^u(\beta'_{n+1})}.
\end{aligned}
$$
\end{proof}

Define, for $k<n$,
$$
E_{k,n}=\{x\in K_k| \text{   } \exists t>0 \text{  }\forall j<t
\text{   } F^j(x)\notin T_n \text{ and } F^t(x)\in T_n\}.
$$
The time $t>0$ in the above definition is called the {\it time of entry }   
of $x\in E_{k,n}$ into $T_n$.

\begin{defn}\label{Indkn} Let $k<n$.
We say that $F$ satisfies the transversality condition
$\mathcal{T}_{k,n}$ if the following holds. Let $z_j\in
E_{k,n}$, $j\ge 0$, be a sequence such that
$$
F^{t_j}(z_j)\rightarrow s\in S_n,
$$
where $t_j>0$ is the time of entry of $z_j$ into $T_n$, then
$$
DF^{t_j}(z_j)(T_{z_j}W^u(\beta_k))\nrightarrow T_sW^s(\beta_n).
$$
\end{defn}

\begin{defn}\label{Kkn} A $(k,n)$-heteroclinic tangency, $k<n$, for an $n$-times renormalizable
H\'enon map is a 
tangency between $W^u(\beta_k)$ and
$W^s(\beta_n)$. If there is such a tangency we write
$$
W^u(\beta_k) \tangent\: 
W^s(\beta_n).
$$ 
Let $\mathcal{K}_{k,n}(\overline{\eps})\subset \HH^{n}_\Omega(\overline{\eps})$
consists of the $n$-times renormalizable maps which have a $(k,n)-$heteroclinic
 tangency and
$$
\UU\KK_{k,n}(\overline{\eps})=
\bigcup_{k\le k'<n'\le  n} \KK_{k',n'}(\overline{\eps}).
$$
\end{defn}

\begin{prop} \label{propTN} Let $F\in \HH^n_\Omega(\overline{\eps})$,  
$\overline{\eps}>0$ small enough. Let $k<n$ and suppose that $F$ 
satisfies
$$
F \notin \UU\KK_{k,n}(\overline{\eps}).
$$
Then 
$
\mathcal{T}_{k,n}
$
holds.
\end{prop}

\begin{proof}
Fix  $k<n$
and
a sequence $z_j\in E_{k, n}$, $j\ge 0$, with  $z_j\rightarrow z$  and
$$
F^{t_j}(z_j)\rightarrow s\in S_{n},
$$
where $t_j>0$ in the time of  entry of $z_j$ into $T_{n}$.

If the $t_j$'s are bounded, say constant $t_j=t$,
the absence of heteroclinic tangencies, $F\notin \KK_{k,n}(\overline{\eps})$, implies that
$$
DF^{t}(z)(T_{z}W^u(\beta_k))\ne T_{s}S_{n}.
$$
Hence,
$$
DF^{t_j}(z_j)(T_{z_j}W^u(\beta_k))\nrightarrow T_sS_{n}.
$$

Secondly, we will consider the case when the times $t_j$ of entry are 
unbounded.

\begin{clm} \label{orbz} There exists $k<m_1\le n$ such that
$
z\in E_{k, m_1}\cap W^s(\bbe_{m_1}).
$
\end{clm}

\begin{proof} Theorem ~\ref{attrac} describes the limit behavior of the orbit
of $z$. Assume,
$$
\omega(z)\subset \OO(F) \cup \bigcup_{j>n} \bbe_j.
$$
Then for some $t>0$ we have $F^i(z)\in \inter(\Trap_{n+1})$ whenever $i\ge t$. This means that the orbit of
$z_j$, $j>0$ large enough,  will also enter this trapping region after $t$
steps. For $j>0$ large enough, $t_j>t$. This contradicts
$$
F^{t_j}(z_j)\to s\notin \Trap_{n+1}.
$$
\end{proof}

\begin{figure}[htbp]
\begin{center}
\psfrag{Tm1}[c][c] [0.7] [0] {\Large $T_{m_1}$}
\psfrag{Tm2}[c][c] [0.7] [0] {\Large $T_{m_2}$}
\psfrag{Tm3}[c][c] [0.7] [0] {\Large $T_{m_3}$}
\psfrag{z}[c][c] [0.7] [0] {\Large $z$}
\psfrag{zj}[c][c] [0.7] [0] {\Large $z_j$}
\psfrag{s1}[c][c] [0.7] [0] {\Large $s_1$}
\psfrag{u1}[c][c] [0.7] [0] {\Large $u_1$}
\psfrag{s2}[c][c] [0.7] [0] {\Large $s_2$}
\psfrag{u2}[c][c] [0.7] [0] {\Large $u_2$}
\psfrag{s3}[c][c] [0.7] [0] {\Large $s_3$}
\psfrag{u3}[c][c] [0.7] [0] {\Large $u_3$}
\psfrag{Fr1}[c][c] [0.7] [0] {\Large $F^{r_1}$}
\psfrag{Fr2}[c][c] [0.7] [0] {\Large $F^{r_2}$}
\psfrag{Fr3}[c][c] [0.7] [0] {\Large $F^{r_3}$}
\pichere{0.8}{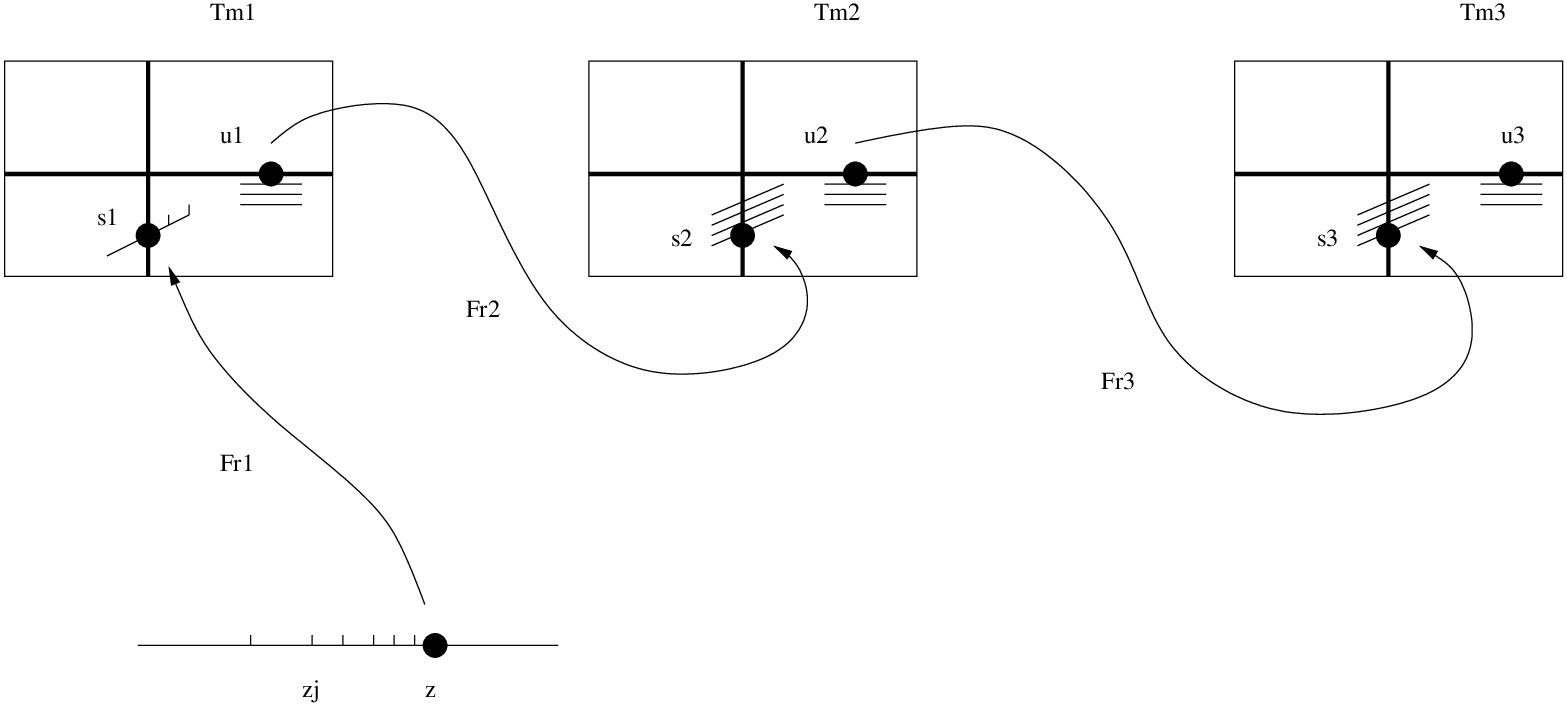}
\caption{}
\label{prflam}
\end{center}
\end{figure}

Denote the time of entry of $z$ into $S_{m_1}\subset T_{m_1}$ by $r_1>0$ and
let $F^{r_1}(z)=s_1$. We will call $r_1$ the first {\it transient time}.
 For $j>0$ large enough, $z_j\in E_{k,m_1}$ with
corresponding entry time $t^1_j=r_1$. Note that $m_1<n$. Otherwise, the sequence consisting of $t_j=t_j^1=r_1$
would be bounded. 
The absence of heteroclinic tangencies, $F\notin \KK_{k,m_1}(\overline{\eps})$,  implies that
$$
DF^{r_1}(z)(T_{z}W^u(\beta_k))\ne T_{s_1}S_{m_1}.
$$
Hence,
\begin{equation}\label{zjs}
DF^{t^1_j}(z_j)(T_{z_j}W^u(\beta_k))\nrightarrow T_{s_1}S_{m_1}.
\end{equation}

Let $e^1_j>0$ be maximal such that when  $r_1\le i\le  e_j^1$ we have 
$$
F^i(z_j)\in T_{m_1} .
$$
The moment $e^1_j$ is called the {\it time of exit} of $z_j$ from $T_{m_1}$.
We may assume that $F^{e^1_j}(z_j)\rightarrow u_1\in U_{m_1}$.
Then ~(\ref{zjs}) implies, use the $\lambda$-Lemma from \cite{dMP},
\begin{equation}\label{zju}
DF^{e^1_j}(z_j)(T_{z_j}W^u(\beta_k))\rightarrow T_{u_1}U_{m_1}.
\end{equation}

Now, we can repeat the proof of Claim ~\ref{orbz} and obtain $m_1<m_2\le n$ and $r_2>0$, the second transient time,  such that
$$
F^{r_2}(u_1)=s_2\in S_{m_2}.
$$
For $j>0$ large enough we have $z_j\in E_{k, m_2}$. Denote the time of entry 
of $z_j$ into $T_{m_2}$ by $t^2_j>0$ then $t^2_j=e^1_j+r_2$.
The absence of heteroclinic tangencies, $F\notin \KK_{m_1,m_2}(\overline{\eps})$, implies that
$$
DF^{r_2}(u_1)(T_{u_1}W^u(\beta_{m_1}))\ne T_{s_2}S_{m_2}.
$$
Hence, ~(\ref{zju}) implies
\begin{equation}\label{zjs2}
DF^{t^2_j}(z_j)(T_{z_j}W^u(\beta_k))\nrightarrow T_{s_2}S_{m_2}.
\end{equation}
Let $e^2_j>0$ be maximal such that when  $t^2_j\le i \le  e_j^2$ we have
$$
F^i(z_j)\in T_{m_2} .
$$
We may assume that $F^{e^2_j}(z_j)\rightarrow u_2\in U_{m_2}$.
Then
\begin{equation}
DF^{e^2_j}(z_j)(T_{z_j}W^u(\beta_k))\rightarrow T_{u_2}U_{m_2}.
\end{equation}
If $m_2=n$, statement  ~(\ref{zjs2}) proves the transversality
property. In the case when $m_2<n$ we can repeat this
construction, and we get a sequence $m_1< m_2<m_3<\dots< m_g$
together with points $s_l\in S_{m_l}$, $u_l\in U_{m_l}$ and times of
entry and exit  $t^l_j>0$ and  $e^l_j>0$ for $z_j\in
E_{k, m_l}$ and the
 corresponding
asymptotic expressions ~(\ref{zjs}) and  ~(\ref{zju}).

The sequence $m_l$ is strictly increasing. Hence, $m_g=n$ and $t_j=t^g_j$
for some $g\ge 1$.
Now, statement ~(\ref{zjs}) corresponding to $T_{m_g}$,
$$
DF^{t_j}(z_j)(T_{z_j}W^u(\beta_k))\nrightarrow T_{s}S_{n}
$$
finishes the proof of  the Proposition.
\end{proof}

\begin{prop}\label{proplam}
Let 
 $F\in \HH^n_\Omega(\overline{\eps})$, 
with $\overline{\eps}>0$ small enough, and $k<n$. Assume
\begin{equation}\label{condK}
F\notin \UU\KK_{k+1,n}(\overline{\eps}).
\end{equation}
Then 
$$
\bigcup_{k\le j \le n} W^u(\bbe_j)
$$
is laminar.
\end{prop}

\begin{proof}
 Choose $k\le j\le n$. To prove that every 
point
in $W^u(\beta_j)$ is laminar it suffices to prove that every point $z\in U_j$
is laminar. According to Lemma \ref{1Dclos} $W^u(\beta_j)$ is not accumulated by $W^u(\beta_m)$ with $m>j$.
 From Lemma ~\ref{1D} we have that $W^u(\beta_j)$ is a 
one-dimensional embedded manifold. Hence, the only non-trivial
 accumulation is from $W^u(\beta_l)$ with $k\le l<j$. Assume that $z\in U_j$ is not 
a laminar point. 
Let $k\le l<j$ and $z_m\in E_{l,j}$ be a sequence with
$$
F^{e_m}(z_m)\to z
$$
but $DF^{e_m}(z_m)(T_{z_m}W^u(\beta_l))$ stays away from $T_{z}U_{j}$. Let
 $t_m<e_m$ be such that
$$
F^{t_m}(z_m)\to s\in S_j.
$$
Proposition \ref{propTN} states that $\mathcal{T}_{l,j}$ holds. 
Hence, for a subsequence, 
$
DF^{t_m}(z_m)(T_{z_m}W^u(\beta_l))
$
stays away from $T_{s}S_{j}$. Then again the $\lambda$-Lemma implies that for 
this subsequence
$$
DF^{e_m}(z_m)(T_{z_m}W^u(\beta_l))\rightarrow T_{z}U_{j}.
$$
Contradiction.
\end{proof}

\noindent
{\it Proof of Theorem ~\ref{lam}.} Suppose, $x\in W^u(\beta_n)$ is a non-laminar point of $\AAA_F$. According to Lemma \ref{1Dclos} this implies that this point is actually a non-laminar point of 
$$
\bigcup_{j\le n} W^u(\bbe_j).
$$
This contradicts Proposition \ref{proplam}.
\qed

\section{Conjugations}

The boundary $\partial B$ of the domain of a H\'enon-like map $F:B\to B$ does not have 
dynamical meaning. 
Even if we restrict the map to $\overline{B_0}$, only an arc of $\di B_0$,
namely  $W^s_\loc(\beta_0)$, is dynamically meaningful. 
The fact that the boundary is rather arbitrary entails that the notion of topological equivalence 
defined by conjugations $h:B\to h(B)=\tl B$ is too restrictive. 
For this reason, below we slightly relax this notion.

A {\it relative neighborhood} $U\subset \bar B_0$ 
 is an open set in the intrinsic topology 
of $B_0$.
A {\it conjugation} between two H\'enon-like maps
 $F, \tilde{F}\in \II_\Omega(\overline{\eps})$ is a homeomorphism 
$h:U\to h(U)=\tilde{U}$ such that
\begin{enumerate}
\item $U\supset \AAA_F$ and $\tilde{U}\supset \AAA_{\tilde{F}} $ are relative 
 neighborhoods in the corresponding boxes;
\item $U$ and   $\tilde{U}$ 
   are  forward invariant under the corresponding dynamics; 
\item $h\circ F=\tilde{F}\circ h$.
\end{enumerate}

\begin{thm}\label{tiptop} 
Let $h:U\to \tilde{U}$ be a  conjugation between two 
infinitely renormalizable H\'enon-like maps $F, 
\tilde{F}\in \II_\Omega(\overline{\eps})$,
with $\overline{\eps}>0$ small enough. Then
$$
h(\Orb_\Bbb{Z}(\tau_F))=\Orb_\Bbb{Z}(\tau_{\tilde{F}}).
$$
\end{thm}

  The dynamics of $F|\OO_F$, the adding machine, is homogeneous,
in the sense that the group of automorphisms acts transitively on $\OO_F$
(here an {\it automorphism} is  a homeomorphism commuting with $F$). 
The situation is different when this Cantor set is 
embedded as the attractor of a H\'enon-like map and the automorphism 
has an extension to a conjugation. 
Then, as the above Theorem shows, any automorphism has to preserve the orbit of the tip. 
This easily implies that the automorphism group is reduced 
to the cyclic group $\Bbb{Z}$ of the iterates of $F|\OO_F$.  

\bigskip

The proof of Theorem \ref{tiptop} needs some preparation.
A map $F\in \II_\Omega(\overline{\eps})$ has exactly two fixed points:
 $\beta_0$ and $\beta_1$. The first has positive eigenvalues and the second 
one is of flip type, it has negative eigenvalues. 
The topological difference between the fixed points imply that every conjugation between two maps satisfies
$$
h(\beta_i)=\tilde{\beta}_i,
$$
 for $i=0,1$. 
If a H\'enon-like map $F$ is renormalizable then
there is only one heteroclinic orbit coming from $\beta_0$ and going to $\beta_1$: 
$$
W^u(\beta_0)\cap W^s(\beta_1)=\{p^1_i\}_{i\in \mathbb{Z}},
$$
with $F(p^1_i)=p^1_{i+1}$.
(The topological definition of renormalizable H\'enon-like maps is discussed in \S 3.4 of \cite{CLM}). Observe,
$$
p^1_0\in W^u(\beta_0)\subset U=\Domain(h).
$$
Hence, every conjugation will satisfy
\begin{equation}\label{match}
h(p^1_0)=\tilde{p}^1_m,
\end{equation}
for some $m\in \mathbb{Z}$. 

In \S \ref{secstabman} the domain $D_1$ was introduced, 
the domain of the first pre-renormalization $F^2|D_1$. 
This  topological disc is bounded by two curves $\partial^s\subset W^s(\beta_1)$ 
and $\partial^u\subset W^u(\beta_0)$ whose endpoints are $p^1_0$ and $p^1_1$.
The forward images of $D_1$, $\partial^s$, and $\partial^u$ are denoted respectively by
$
D_1^l=F^l(D_1)
$, $\delta^s_l=F^l(\partial^s)$, and $\delta^u_l=F^l(\partial^u)$, 
$l\ge 0$. The map $F^l:D_1\to D^l_1$ is a diffeomorphism.

For $l\geq 0$,
the curve $\delta^s_l\subset W^s(\beta_1)$ connects $p^1_l$ with $p^1_{l+1}$.
On the other hand,
for $l<0$, and  $\overline{\eps}>0$ small enough,  
there is no arc in $W^s(\beta_1)$ which connects $p^1_l$ with $p^1_{l+1}$. 
The connected components of $W^s(\beta_1)$ that contain the points $p^1_l$, $l<0$, are pairwise disjoint. 
This is observed in 
\S\ref{secstabman}, see Figure \ref{extendedlocstabmani}, 
and will be useful in what follows.

\begin{lem}\label{newh} Let $h:U\to \tilde{U}$ be a conjugation between
$F, \tilde{F}\in \II_\Omega(\overline{\eps})$,
$
\tilde{F}\circ h=h\circ F.
$
There exists $k,l\ge 0$ and  a conjugation 
$$
h':V\to \tilde{V},
$$
given by
$$
h'=\tilde{F}^{-l}\circ h \circ F^k
$$ 
such that
\begin{enumerate}
\item $D_1\subset V$, $\tilde{D}_1\subset \tilde{V}$,
and
$$
h'(D_1)=\tilde{D}_1,
$$
\item 
for every $x\in V$ 
$$
h'(\Orb(F^l(x)))=h(\Orb(F^k(x))).
$$
\end{enumerate}
\end{lem}

\begin{proof} Lemma \ref{trapping} gives a $k\ge 1$ such that 
$F^{k}(\overline{B_0})\subset U$. Define a conjugation
$$
h_1: B_0 \to h_1(B_0) \subset \tilde{U}\subset B_0,
$$
by
$
h_1=h\circ F^k.
$ 
Observe that the maps $h$ and $h_1$ act in the same way on the space of orbits. 
In particular, equation (\ref{match}) gives some $l\in \Bbb{Z}$ such that
$$
h_1(p^1_0)=\tilde{p}^1_l.
$$
The curve $\partial^s\subset W^s(\beta_1)$ connects $p^1_0$ with $p^1_1$. Hence, the points $\tilde{p}^1_l$ and $\tilde{p}^1_{l+1}$ are connected by a curve in the stable manifold of $\tilde{\beta}_1$. So, $l\ge 0$. 
The domain of $h_1$, $\Domain(h_1)=B$, contains $D_1$.
Actually, $h_1$ matches the boundaries $\partial^{u,s}$ of $D_1$ with the boundaries  $\tilde{\delta}^{u,s}$ of 
$\tilde{D}^l_1$. Hence,
$$
h_1(D_1)=\tilde{D}^l_1.
$$
As was noticed previously,  the map 
$$
\tilde{F}^{-l}: \tilde{D}^l_1\cap \AAA_{\tilde{F}} \to \tilde{D}_1\cap \AAA_{\tilde{F}}
$$
is a well-defined homeomorphism because $l\ge 0$. 
Choose  a relatively open $\tilde{F}$-forward  invariant set  $\tilde{V}'\subset B_0$ satisfying
$$
\tilde{D}^l_1\cup \AAA_{\tilde{F}}\subset \tilde{V}'\subset h_1(B_0),
$$
and small enough such that $\tilde{F}^{-l}|\tilde{V}'$ is a well-defined
 diffeomorphism. Let 
$$
V=h_1^{-1}(\tilde{V}'),
$$
$$
\tilde{V}=\tilde{F}^{-l}(\tilde{V}'),
$$
and let $h': V\to \tilde{V}$ be defined by
$$
h'=\tilde{F}^{-l}\circ h_1= \tilde{F}^{-l}\circ h \circ F^k.
$$
By  construction this conjugation satisfies
$
h'(D_1)=\tilde{D}_1.
$
\end{proof}

\comm{
If a conjugation matches $p^1_0$ to $\tilde{p}^1_0$, $h(p^1_0)=\tilde{p}^1_0$,
 then
$$
h(\partial^{u,s})=\tilde{\partial}^{u,s}.
$$
Hence,
$$
h(D_1)=\tilde{D}_1.
$$

Consider a situation where the domain and image of a conjugation
$h:U\to \tilde{U}$ contain the local stable manifold of the fixed point of flip type: $W^s_{\loc}(\beta_1)\subset U$ and $W^s_{\loc}(\tilde{\beta}_1)\subset \tilde{U}$. These local manifolds split the domain and image of $h$ into two components. Hence,
$$
h(W^s_{\loc}(\beta_1))=W^s_{\loc}(\tilde{\beta}_1).
$$
The point $p^1_0$ is the first intersection point between $W^u(\beta_0)$ and
 $W^s_{\loc}(\beta_1))$. So,
$$
h(p^1_0)=\tilde{p}^1_0,
$$
and $h(D_1)=\tilde{D}_1$. 
This property of the conjugations holds when $\Domain(h)=\Image(h)=B$.

}

According to the previous Lemma, we can replace any conjugation by another one 
which matches the first pre-renormalization domains $D_1$ and $\tilde{D}_1$, 
and coincides with the original conjugation on the space of orbits. 
The following Proposition will complete the proof of Theorem~\ref{tiptop}.

\begin{prop} \label{conj} Let $h:U\to \tilde{U}$ be a conjugation between two 
infinitely renormalizable H\'enon-like maps $F, 
\tilde{F}\in \II_\Omega(\overline{\eps})$,
with $\overline{\eps}>0$ small enough. 
If $h(D_1)=\tilde{D}_1$    
then
$$
h(D_n)=\tilde{D}_n,
$$
for all $n\ge 1$. In particular, $h(\tau_F)=\tau_{\tilde{F}}$.
\end{prop}

\begin{figure}[htbp]
\begin{center}
\psfrag{Bn-1}[c][c] [0.7] [0] {\Large $\beta_{n-1}$}
\psfrag{Bn}[c][c] [0.7] [0] {\Large $\beta_n$}
\psfrag{Bn+1}[c][c] [0.7] [0] {\Large $\beta_{n+1}$}
\psfrag{gam}[c][c] [0.7] [0] {\Large $\gamma$}
\psfrag{x0}[c][c] [0.7] [0] {\Large $x_0$}
\psfrag{x1}[c][c] [0.7] [0] {\Large $x_1$}
\psfrag{Dn+1}[c][c] [0.7] [0] {\Large $D_{n+1}$}
\psfrag{Dn}[c][c] [0.7] [0] {\Large $D_{n}$}
\psfrag{En}[c][c] [0.7] [0] {\Large $E_n$}
\psfrag{y0}[c][c] [0.7] [0] {\Large $y_0$}
\psfrag{y1}[c][c] [0.7] [0] {\Large $y_1$}
\psfrag{dsn}[c][c] [0.7] [0] {\Large $\partial^s_n$}
\psfrag{dun}[c][c] [0.7] [0] {\Large $\partial^u_n$}
\pichere{0.75}{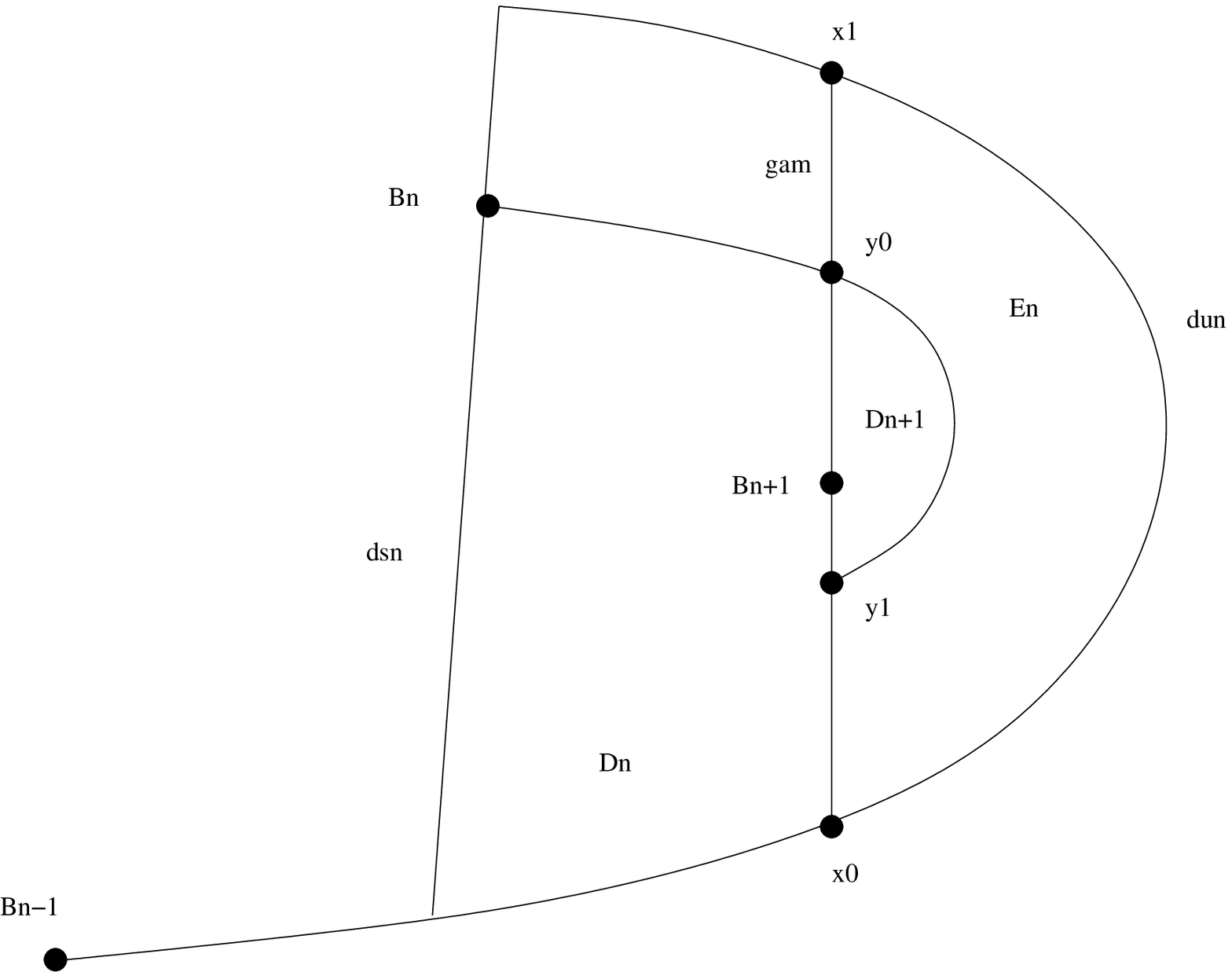}
\caption{Domains of pre-renormalizations}
\label{Domains of pre-renormalizations}
\end{center}
\end{figure}

\begin{proof} 
First notice that for all $n\ge 0$ we have $D_n\subset D_1\subset \Domain(h)$ and $\tilde{D}_n\subset \tilde{D}_1\subset \Image(h)$. So, $h(D_n)$ is well defined.

The proof will be by induction. Assume that $h(D_k)=\tilde{D}_k$, for $k\le n$.
There exists a unique periodic point of period $2^{k+1}$ in $D_k$. Namely,
$\beta_{k+1}\in \inter(D_k)$, $k\le n$. In particular, 
$h(\beta_k)=\tilde{\beta}_k$, with $k\le n+1$.

Proposition \ref{Wslocext} gives that both components of
$W^s_{\loc}(\beta_{n+1})\setminus\{\beta_{n+1}\}$ intersect $\partial^u_n$.
 Let
$x_0, x_1\in  W^s_{\loc}(\beta_{n+1})\cap \partial^u_n$
be the boundary points of the connected component of
$W^s(\beta_{n+1})\cap D_n$ containing $\beta_{n+1}$.  Say $x_0$ is the first 
and $x_1$ is the second  intersection of $W^u(\beta_{n-1})$ with the connected component of $W^s(\beta_{n+1})\cap D_n$ which contains $\beta_{n+1}$.
These points are topologically defined. As was noticed before, 
$h(\beta_k)=\tilde{\beta}_k$,
 with $k=n-1, n+1$. Hence, for $i=0,1$,
$$
h(x_i)=\tilde{x}_i.
$$

\noindent
Let $\gamma\subset W^s_{\loc}(\beta_{n+1})$ bounded by $x_0$ and $x_1$. Then
$$
h(\gamma)=\tilde{\gamma}.
$$
Define $E_n\subset D_n$ to be the connected component of $D_n\setminus \gamma$
which does not contain $\beta_n$. Then
$$
h(E_n)=\tilde{E}_n.
$$
Observe,
$$
D_{n+1}\subset E_n.
$$
Let $y_0, y_1\in  W^u(\beta_{n})\cap \gamma$ be the first and second
intersections of $W^u(\beta_n)$ and $\gamma$. Then for $i=0,1$
$$
h(y_i)=\tilde{y}_i.
$$
 Notice, that the arc between $y_0$ and $y_1$ in
$ W^u(\beta_{n})$ equals $\partial^u_{n+1}$. Furthermore, the arc between
$y_0$ and $y_1$ in $ W^s(\beta_{n+1})$ equals $\partial^s_{n+1}$. Hence, the
boundary of $D_{n+1}$ is matched to the boundary of $\tilde{D}_{n+1}$. This finishes the induction step,
$
h(D_{n+1})=\tilde{D}_{n+1}.
$
\end{proof}


\begin{rem} Without loss of generality we will only consider conjugations 
which match the tips of the maps under consideration.
\end{rem}

\section{Heteroclinic tangencies}\label{hettan}

If there is a heteroclinic tangency between $W^u(\beta_k)$ and
$W^s(\beta_n)$ then $\beta_n\in \AAA_F$ is not a laminar point.
Under this circumstances there will be non-periodic
points which are non-laminar. Let $\CC_F\subset \AAA_F$ consists
of the non-laminar points. Note, $\OO_F\subset \CC_F$.

Any 
map $F\in \HH^n_\Omega(\overline{\eps})$, with $\overline{\eps}>0$ small 
enough,
has a unique periodic orbit  of period $2^{k-1}$, it is the orbit of
$\beta_{k}$. Let $\lambda_{k}\in (-1,0]$ and $\mu_{k}< -1$ be the
stable and unstable multiplier.
The  attractor at the {\it $n^{th}$-scale}  is
$$
\AAA_F^n=\Orb(\Psi_0^n(\AAA_{R^nF}))\subset \AAA_F, \quad n\ge 0.    
$$

\begin{thm}\label{CF} If the infinitely renormalizable H\'enon
map $F\in \II_\Omega(\overline{\eps})$, with $\overline{\eps}>0$ small enough,
has an $(k,n)$-heteroclinic tangency and
$$
\frac{\ln |\lambda_k|}{\ln |\mu_n|}\notin \mathbb{Q}
$$
then
$$
 \AAA_F^n \subset \mathcal{C}_F.
$$
\end{thm}

\begin{proof}  We can choose a $C^{1}$-coordinate system for $T_n$ such
that $U_n$ and $S_n$ are part of the $x$-axis and $y$-axis resp and that
$F^{2^n}$ becomes linear with exponents $\lambda_n$ and $\mu_n$, see 
\cite{H}. Consider
the fundamental domain $[1,\mu_n^2]^u\subset U_n$.
 Let $z\in S_n$ be a $(k,n)$-heteroclinic tangency. 

Observe, that
there are components of $W_j\subset W^u(\beta_{k-1})\cap T_k$, $j\ge 1$,
 which accumulate from both sides and in $C^3$ sense on $U_k$ and they are dynamically related. Namely,
$$
F^{-2^k}(W_{j+1})\subset W_j.
$$ 
This laminar structure of $W^u(\beta_{k-1})$ around $U_k$ will also be visible in a neighborhood of $z\in S_n$. Because of the
tangency at $z\in S_n$ there will be a sequence of points
$e_j\in W^u(\beta_{k-1})$ with vertical tangent accumulating at $z\in S_n$.

\begin{figure}[htbp]
\begin{center}
\psfrag{Tn}[c][c] [0.7] [0] {\Large $T_n$}
\psfrag{Tk}[c][c] [0.7] [0] {\Large $T_k$}
\psfrag{Wuk}[c][c] [0.7] [0] {\Large $W^u(\beta_k)$}
\psfrag{Wsk}[c][c] [0.7] [0] {\Large $W^s(\beta_k)$}
\psfrag{Un}[c][c] [0.7] [0] {\Large $U_n$}
\psfrag{Sn}[c][c] [0.7] [0] {\Large $S_n$}
\psfrag{ej}[c][c] [0.7] [0] {\Large $e_j$}
\psfrag{1}[c][c] [0.7] [0] {\Large $1$}
\psfrag{mun}[c][c] [0.7] [0] {\Large $\mu^2_n$}

\pichere{0.8}{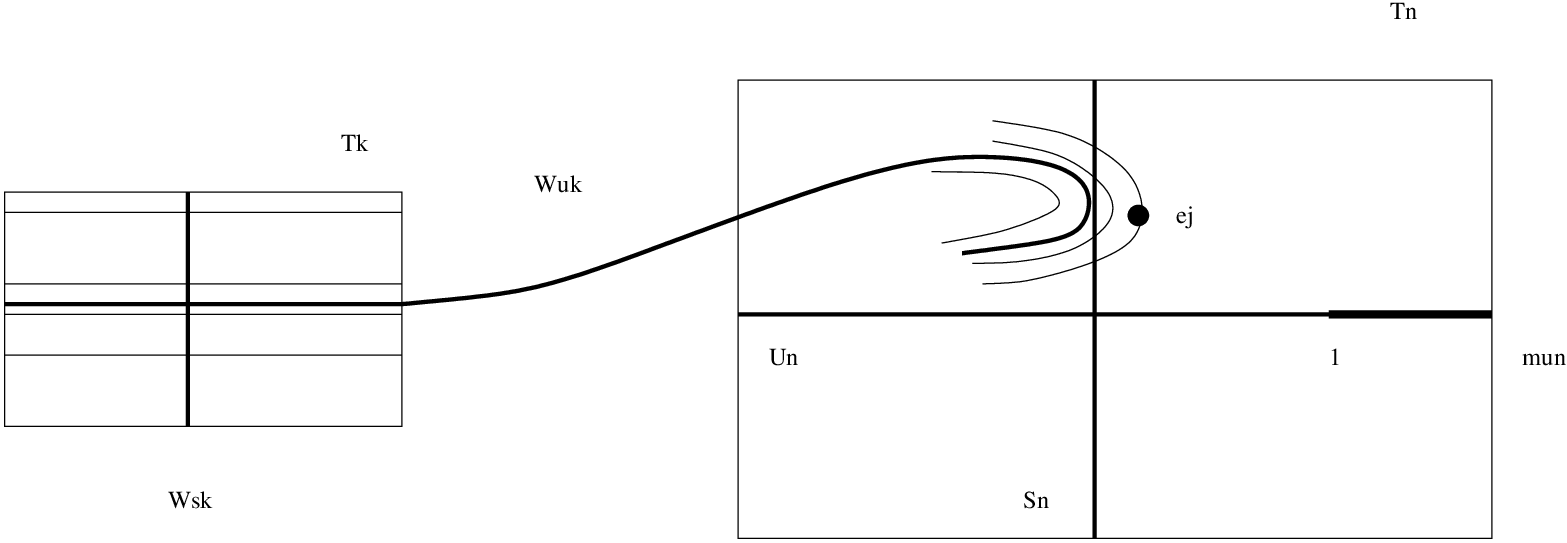}
\caption{}
\label{prfCF}
\end{center}
\end{figure}

 Let
$x_j$ be the $x$-coordinate of $e_j$. Then there is some $\rho<1$ and $C>0$ 
such that
$$
x_j=C\cdot (1+O(\rho^j))\cdot \lambda_k^j.
$$
Notice that accumulation points of $\text{Orb}(\{e_j | j>0\})$ on $U_n$
are non-laminar points.

For $j>0$ even, let $s_j>0$ be the (even) moment when
$F^{s_j}(e_j)$ is above $[1,\mu_n^2]^u$, 
$$
x_j\cdot \mu_n^{s_j}\in [1,\mu_n^2]^u.
$$
Let $A=\ln C/\ln \mu_n$ then
$$
s_j+j\frac{\ln |\lambda_k|}{\ln |\mu_n|} +O(\rho^j)+A=h_j\in [0,2].
$$
Hence,
$$
\frac{h_j}{2}=\frac{s_j}{2}+\frac12 \cdot 
(j \cdot \frac{\ln |\lambda_k|}{\ln |\mu_n|}+A)+O(\rho^j)\in [0,1].
$$
Because $s_j$ is even we have
$$
h_j=2\{\frac12 \cdot j \cdot \frac{\ln |\lambda_k|}{\ln |\mu_n|}+A\}+O(\rho^j),
$$
where $\{.\}$ stand for the fractional part.
The sequence $h_j$ is dense in $[0,2]$ because
$\frac{\ln |\lambda_k|}{\ln |\mu_n|}$
is irrational.
Now, let $\hat{x}_j$ be the projection of $F^{s_j}(e_j)$ on $U_n$. Then
$$
\hat{x}_j=\mu_n^{h_j}.      
$$
We proved that $\mathcal{C}_F$ contains a fundamental domain of $W^u(\beta_n)$. Namely,
$$
[1,\mu_n^2]^u\subset \mathcal{C}_F.
$$
 The set
$\mathcal{C}_F$ is closed and invariant. Apply 
Theorem ~\ref{attrac} and the proof is finished.
\end{proof}

\begin{cor}\label{existmessCF} For $\overline{\eps}>0$ small enough,
for every $k<n$ there exists a dense $G_\delta$ of infinitely renormalizable maps $F\in
\mathcal{K}_{k,n}(\overline{\eps})\cap \II_\Omega(\overline{\eps})$ such that
$$
\AAA_F^n \subset \mathcal{C}_F.
$$
\end{cor}             

\section{Location of the tip}\label{sectip}

In this section we will give quantitative information on the location
of the tip.
 Let  $F\in
\II_\Omega(\overline{\epsilon})$, $\overline{\epsilon}>0$ small
enough.  The domain
bounded by $W^s_\loc(\beta_2)$ and
$W^s_\loc(\beta'_2)$ is called the {\it
extended saddle region} and denoted by $X\subset B$.

Consider the following family $\mathcal{W}$ of curves. A curve
$\gamma\subset X$ is in $\mathcal{W}$ if it is the graph of a
$C^2$-function, also denoted by $\gamma$, with
$$
\gamma: [x'_\gamma,x_\gamma] \rightarrow \Bbb{R},
$$
such that
$$
(x'_\gamma,\gamma(x'_\gamma))\in  W^s_\loc(\beta'_2)
$$
and
$$
(x_\gamma,\gamma(x_\gamma))\in  W^s_\loc(\beta_2).
$$
We will use the notation $\gamma_\infty=U_1\in \mathcal{W}$, 
(see Figure \ref{saddlereg} to recall the definition of $U_1\subset B$). 
The distance $\dist(\gamma,\gamma_\infty)$ is the $C^2$-norm between
the corresponding functions measured on the domain of $\gamma$. 
Note that we can always extend
$\gamma_\infty$ within $W^u_\loc(\beta_1)$ such that the
corresponding domain of this extended  function  contains the
domain $[x'_\gamma,x_\gamma]$ of any function $\gamma\in \mathcal{W}$.

Let $\Gamma=F(\gamma)$, with $\gamma\in
\mathcal{W}$. In particular, $\Gamma_\infty=F(\gamma_\infty)\supset
 \gamma_\infty$.
Note that each curve $\Gamma$ can be described as
a graph over the $y$-axis,
$$
\Gamma: [y'_\gamma,y_\gamma] \rightarrow \Bbb{R}.
$$
This is a consequence of the fact that  H\'enon-like maps 
map vertical lines to horizontal lines, $y'=x$.  Moreover,
$$
(\Gamma(y_\gamma), y_\gamma)\in  W^s_\loc(\beta'_2)
$$
and
$$
(\Gamma(y'_\gamma), y'_\gamma)\in  W^s_\loc(\beta_2).
$$
Indeed, we will consider the curves $\Gamma=F(\gamma)$ as graphs over the 
$y-$axis. Note that we can always extend
$\Gamma_\infty$ within $W^u_\loc(\beta_1)$ such that the
corresponding domain of the extension  contains the
domain $[y'_\gamma,y_\gamma]$ of any function $\Gamma$.  The distance
$\dist(\Gamma,\Gamma_\infty)$ is the $C^2$-norm between the
corresponding functions measured on the domain of $\Gamma$.

The map $F$ acts on $\mathcal{W}$ as a graph transform.
Namely, the curve  $W^s_\loc(\beta_2)$ divide each
$\Gamma$ into two components. One of
which, denoted by $\TT_F(\gamma)$, is in $\mathcal{W}$,
$$
\TT_F:\gamma \mapsto F(\gamma)\cap X.
$$

\begin{figure}[htbp]
\begin{center}
\psfrag{B0}[c][c] [0.7] [0] {\Large $\beta_0$}
\psfrag{B1}[c][c] [0.7] [0] {\Large $\beta_1$}
\psfrag{gam1}[c][c] [0.7] [0] {\Large $\gamma_1$}
\psfrag{gam2}[c][c] [0.7] [0] {\Large $\gamma_2$}
\psfrag{gam3}[c][c] [0.7] [0] {\Large $\gamma_3$}
\psfrag{gam4}[c][c] [0.7] [0] {\Large $\gamma_4$}
\psfrag{gam5}[c][c] [0.7] [0] {\Large $\gamma_5$}

\psfrag{Gam2}[c][c] [0.7] [0] {\Large $\Gamma_2$}
\psfrag{Gam3}[c][c] [0.7] [0] {\Large $\Gamma_3$}
\psfrag{Gam4}[c][c] [0.7] [0] {\Large $\Gamma_4$}
\psfrag{Gam5}[c][c] [0.7] [0] {\Large $\Gamma_5$}
\psfrag{Gaminf}[c][c] [0.7] [0] {\Large $\Gamma_\infty$}

\psfrag{Zn}[c][c] [0.7] [0] {\Large $Z_n$}

\psfrag{Ws0}[c][c] [0.7] [0] {\Large $W^s_\loc(\beta_0)$}
\psfrag{Ws1}[c][c] [0.7] [0] {\Large $W^s_\loc(\beta_{1})$}
\psfrag{Ws2}[c][c] [0.7] [0] {\Large $W^s_\loc(\beta_2)$}
\psfrag{Ws2'}[c][c] [0.7] [0] {\Large $W^s_\loc(\beta'_2)$}

\pichere{0.8}{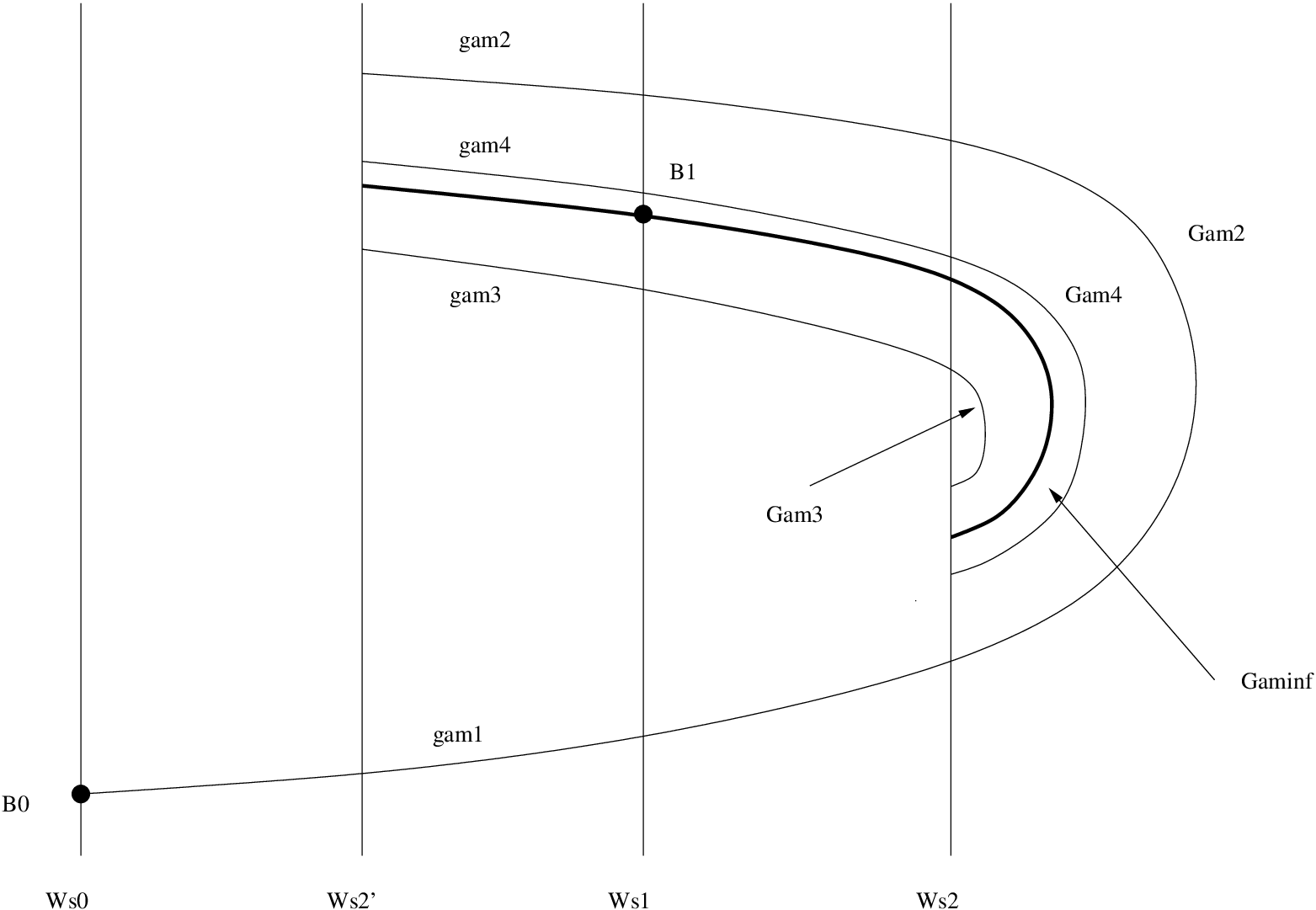}
\caption{}
\label{hatFn}
\end{center}
\end{figure}

\begin{prop}\label{quatlam} There exists $C>0$, such that for all 
 $F\in
\II_\Omega(\overline{\epsilon})$, $\overline{\epsilon}>0$ small
enough, the following holds:
\begin{equation}\label{Gamma}
\dist(\Gamma,\Gamma_\infty)\le C\cdot b_F \cdot
\dist(\gamma,\gamma_\infty)
\end{equation}
and
\begin{equation}\label{Fgamma}
\dist(\TT_{F}(\gamma),\gamma_\infty)\le C\cdot b_F \cdot
\dist(\gamma,\gamma_\infty).
\end{equation}
If $|\gamma(x)-\gamma_\infty(x)|\in  [\delta_1,\delta_2]$ with $x\in
[x'_\gamma, x_\gamma]$ then
\begin{equation}\label{Fgammax}
|\TT_{F}(\gamma)(x)-\gamma_\infty(x)|\in [\frac{1}{C} \cdot b_F \cdot
\delta_1,C\cdot b_F \cdot\delta_2],
\end{equation}
for $x\in [x'_{\TT_{F}(\gamma)},x_{\TT_{F}(\gamma)}]$ and
\begin{equation}\label{Gammay}
|\Gamma(y)-\Gamma_\infty(y)|\in [\frac{1}{C} \cdot b_F \cdot
\delta_1,C\cdot b_F \cdot\delta_2],
\end{equation}
for $y\in [y'_\gamma,y_\gamma]$.
\end{prop}

\begin{proof} Let $F\in \II_\Omega(\overline{\epsilon})$, 
say $F(x,y)=(f(x)-\epsilon(x,y),x)$. 
Inspired by  Theorem 7.9 from [CLM], we use the following representation
$$
\epsilon(x,y)=b_F a(x) y (1+m(x,y)).
$$
The correction term $m$ is uniformly bounded for $F\in \II_\Omega(\overline{\epsilon})$.
The specific form of a H\'enon-like map, $F(x,y)=(f(x)-\epsilon(x,y),x)$, 
 implies
$$
\begin{aligned}
\Gamma (y)-\Gamma_\infty(y)=&
-\epsilon(y,\gamma(y))-\epsilon(y,\gamma_\infty(y))\\
= & - b_F a(y)\cdot ( \gamma(y)-\gamma_\infty(y))+\\
& - b_F a(y)\cdot 
(\gamma(y) -\gamma_\infty(y)) \cdot m(y,\gamma(y))+\\
& - b_F a(y)\cdot 
\gamma_\infty(y) \cdot ( m(y,\gamma(y))-m(y,\gamma_\infty(y) ).\\
\end{aligned}
$$
The uniform bounds on $m$ and its derivatives give immediately the 
Properties (\ref{Gamma}) and (\ref{Gammay}).

\comm{

= &- b_F a(y)\cdot ( \gamma(y)-\gamma_\infty(y))+\\
& - b_F a(y)\cdot 
(\gamma(y)\cdot  m(y,\gamma(y))-\gamma_\infty(y)\cdot  m(y,\gamma_\infty(y)))\\

}

Let $\gamma_0: x\mapsto 0$.
The expression $F(x,y)=(f(x)-\epsilon(x,y), x)$ implies that
$\Gamma_0$ is the graph of the unimodal map $f$. Let
 $Y\subset \Domain(\Gamma_0)$ be such that the graph of $\Gamma_0|Y\subset X$, the extended saddle region. The graph of $f$ has its maximum outside the extended saddle region. This implies that
$$
|D\Gamma_0|Y|\ge \delta>0.
$$
Also
$$
|D\Gamma_\infty|Y|\ge \frac12 \delta>0,
$$
which holds because of 
(\ref{Gamma}):
$$
\dist(\Gamma_0,\Gamma_\infty)=O(b_F).
$$
Now, use in the above expression for the difference of $\Gamma (y)$ and 
$\Gamma_\infty(y)$ the fact that the derivative of $\Gamma_\infty|Y$ is away 
from zero:  
Properties (\ref{Fgamma}) and (\ref{Fgammax}) follow.
\end{proof}

\begin{rem} According to Theorem 7.9 from \cite{CLM} 
the correction term  $m_n$ of $R^nF$ decays exponentially. 
\end{rem}

Let $q'_1,q_1\in W^u(\beta_0)$ be the the first intersection, coming from $\beta_0$ along $W^u(\beta_0)$,
with $W^s_\loc(\beta'_2)$ and  $W^s_\loc(\beta_2)$, compare Figure \ref{saddlereg}.
Consider the corresponding curve $\gamma_1=[q'_1, q_1]^u\subset
W^u(\beta_0)$. Note, $\gamma_1\in \mathcal{W}$. Let
$$
\gamma_i=\TT_F^{i-1}(\gamma_1), \text{   } i=2,3,\dots
$$
The curves
$$
\Gamma_i=F(\gamma_{i-1})\supset \gamma_i, \text{   } i=2,3,\dots
$$
will be used to locate the tip.

\begin{cor}\label{Gammabounds} There exists $C>0$, independent of the particular $F\in \II_{\Omega}(\overline{\epsilon})$, such that for $i\ge 2$ 
$$
\dist(\Gamma_i,\Gamma_\infty)\le C^i\cdot  b_F^{i-1}
$$
and
$$
\frac{1}{C^i} \cdot b_F^{i-1}\le |\Gamma_i(y)-\Gamma_\infty(y)|\le C^i \cdot
 b_F^{i-1},
$$
for $y\in \Domain(\Gamma_i)$.
\end{cor}

\comm{

The intersection $\Gamma_\infty \cap W^s_\loc(\beta_2(F_n))$ consists of two
point. Denote them by by $s^1_\infty, s^2_\infty\in \Gamma_\infty$. Say,
$s^1_\infty$ is the upper one. Similarly, let
$s^1_5, s^2_5\in \Gamma_5\cap W^s_\loc(\beta_2(F_n))$. Let $Z_n$ be the domain
bounded by the curves $[s^1_\infty, s^2_\infty]\subset \Gamma_\infty$,
$[s^1_5, s^2_5]\subset  \Gamma_5$ and $W^s_\loc(\beta_2(F_n))$. Notice, that
the horizontal width of $Z_n$ is proportional to  $b_F^{2^{n+2}}$.

\begin{lem}\label{toploctip} $\tau_{F_n}\in Z_n$.
\end{lem}

\begin{proof}
\end{proof}

}

\begin{lem}\label{vertang} There exist unique points $v_i\in
\Gamma_i$, $i\ge 2$, and $v_\infty\in
\Gamma_\infty$ with a vertical tangent direction. Moreover,
$$
\frac{d\Gamma_i}{dy}\asymp -(y-v_i), \text{  } i=2,3,\dots,\infty.
$$
\end{lem}

\begin{proof} Let  $F(x,y)=(f(x)-\epsilon(x,y),x)$ and
$\gamma_0: x\mapsto 0$. Then $\Gamma_0=F(\gamma_0)$ is the graph of $f$
over the $y-$axis. Proposition ~\ref{quatlam} and Corollary ~\ref{Gammabounds}  imply that  in $C^2$ sense we have 
$$
\dist(\Gamma_0,\Gamma_i)\le\dist(\Gamma_0,\Gamma_\infty)+
\dist(\Gamma_\infty,\Gamma_i)\le C\cdot b_F +C^i\cdot b^{i-1}_F<<1,
$$
for $i=2,3\dots$.
The Lemma follows because $f$ has a non 
degenerate critical point.  
\end{proof}

Let $a$ be the intersection point of the horizontal line through
$\tau=\tau_{F}$ and $\Gamma_2$. Note, $a$ is to the right of
$\tau$. In case we are analyzing these points of the renormalization $R^kF$,
 we will use the notation $v^k_i\in \Gamma^k_i$, $i=2,3,\dots,\infty$, the points with vertical tangency.

\comm { The tip is denoted by $\tau^n$ and $a^n$ is the intersection point of the horizontal line through the tip and $\Gamma^n_2$.}

\begin{prop}\label{depth} For all $F\in
\II_\Omega(\overline{\epsilon})$, $\overline{\epsilon}>0$ small enough, the following holds.
\begin{enumerate}
\item  $|\pi_1(v_2)-\pi_1(a)|=O(b_F^{2})$,
\item  $|\pi_2(v_2)-\pi_2(a)|=O(b_F)$,
\item $|v_2-\tau|=O(b_F)$,
\item $|\tau-a|\asymp b_F$.
\end{enumerate}
\end{prop}

\begin{rem} The proof of Proposition \ref{depth} is illustrated with Figures \ref{postip1} and \ref{postip2}. The $x$-direction has been stretched dramatically.The analysis is done in a small neighborhood of $v^{k+1}_\infty$ in Figure 
\ref{postip1} and around $v_\infty$ in Figure \ref{postip2}. So, in the actual picture the curves are essentially straight vertical lines and the difference between the 
$x$-coordinates $\hat{x}_{k+1}$,  $x_{k+1}$, and  $x'_{k+1}$ would be invisible. 
\end{rem}

\begin{figure}[htbp]
\begin{center}
\psfrag{Gamma2}[c][c] [0.7] [0] {\Large $\Gamma_2^{k+1}$}
\psfrag{Gamma00}[c][c] [0.7] [0] {\Large $\Gamma_\infty^{k+1}$}
\psfrag{v00}[c][c] [0.7] [0] {\Large $v_\infty^{k+1}$}
\psfrag{vhat}[c][c] [0.7] [0] {\Large $\hat{v}_\infty^{k+1}$}
\psfrag{y}[c][c] [0.7] [0] {\Large $y_{k+1}$}
\psfrag{x}[c][c] [0.7] [0] {\Large $x_{k+1}$}
\psfrag{yhat}[c][c] [0.7] [0] {\Large $\hat{y}_{k+1}$}
\psfrag{xhat}[c][c] [0.7] [0] {\Large $\hat{x}_{k+1}$}
\psfrag{xprime}[c][c] [0.7] [0] {\Large $x'_{k+1}$}
\psfrag{Phi}[c][c] [0.7] [0] {\Large $\Phi_k^{k+1}$}
\psfrag{Gamma00'}[c][c] [0.7] [0] { $\Gamma_\infty^{k}$}
\psfrag{v00'}[c][c] [0.7] [0] { $v_\infty^{k}$}
\psfrag{Omega}[c][c] [0.7] [0] {\Large $\Omega_k$}
\pichere{0.8}{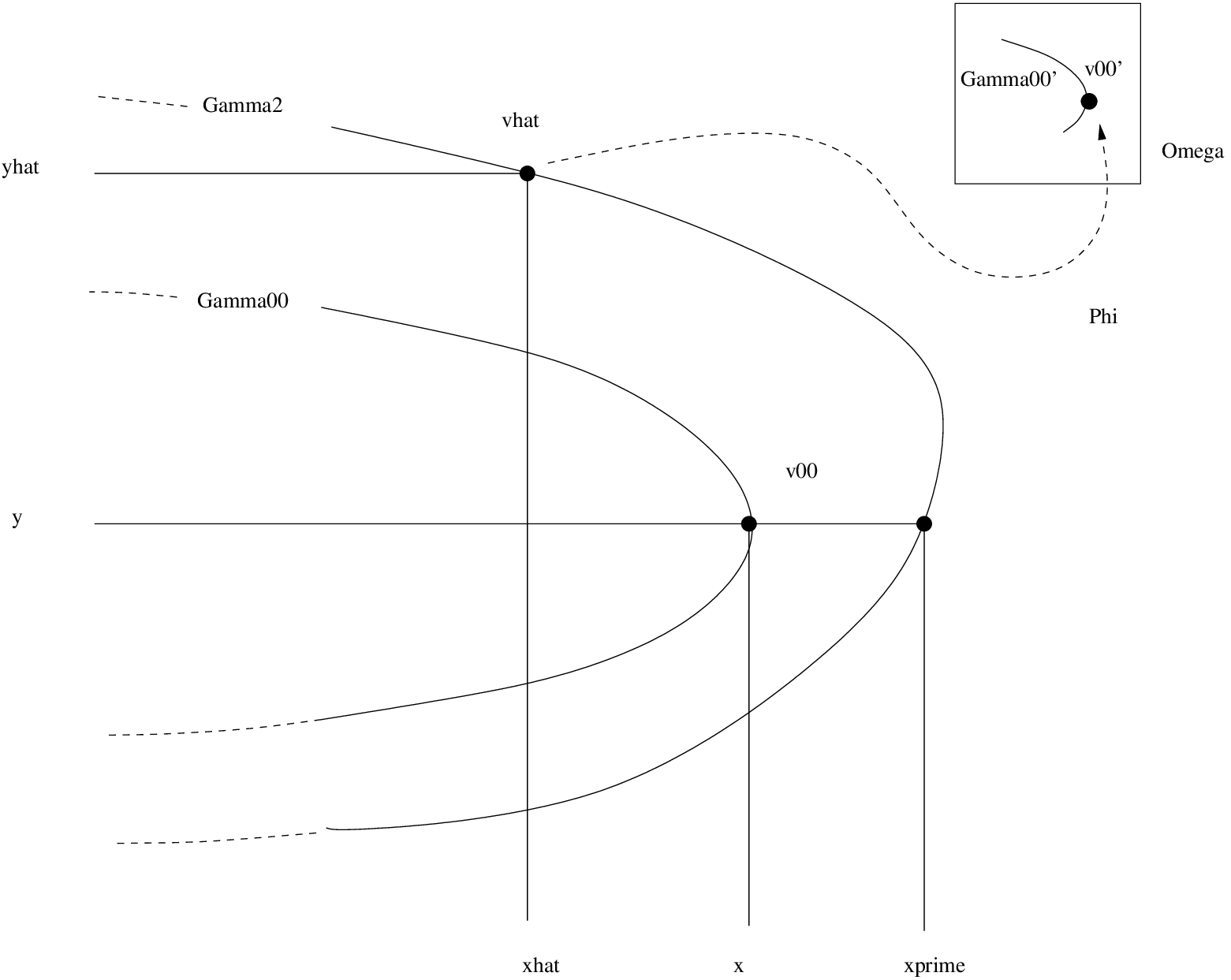}
\caption{}
\label{postip1}
\end{center}
\end{figure}

\begin{proof} Take $k\ge 0$ and consider $R^{k+1}F$. Observe,
$$
v^k_\infty\in  \Phi^{k+1}_k(\Gamma^{k+1}_2)\subset
\Phi^{k+1}_k(W^u(\beta_0(R^{k+1}F))).
$$
So, we can choose a point 
$\hat{v}^{k+1}_\infty\in \Gamma^{k+1}_2\subset \Omega_{k+1}$ 
such that
$$
\Phi^{k+1}_k(\hat{v}^{k+1}_\infty)=v^k_\infty\in D_2(R^kF).
$$
In coordinates, see Figure \ref{postip1},
$$
\hat{v}^{k+1}_\infty=(\hat{x}_{k+1}, \hat{y}_{k+1})
$$
and
$$
v^{k+1}_\infty=(x_{k+1}, y_{k+1}).
$$
Remember,
$$
\Lambda_k\circ H_k\equiv (\Phi^{k+1}_k)^{-1},
$$
with
$$
H_k(x,y)=(f_k(x)-\epsilon_k(x,y), y).
$$
The affine map $\Lambda_k$ is a dilatation.
Consider $DH_k$ and observe that this derivative maps the vertical tangent 
vector at $v^k_\infty$ to $\Gamma^k_\infty$ to an almost vertical  vector.
The angle between the image vector and the vertical is of the order 
$b_F^{2^k}$. 
Then,
because 
$\hat{v}^{k+1}_\infty=\Lambda_k\circ H_k(v^k_\infty)$ and 
$D\Gamma^{k}_\infty(y_k)=0$, we get
\begin{equation}\label{abaA}
|D\Gamma^{k+1}_2( \hat{y}_{k+1})|=O(b_F^{2^k}).
\end{equation}
Let 
$$
x'_{k+1}=\pi_1(\Gamma^{k+1}_2(y_{k+1})).
$$
Corollary \ref{Gammabounds} gives
\begin{equation}\label{xXx}
|x_{k+1}-x'_{k+1}|\asymp b_F^{2^{k+1}}
\end{equation}
and
\begin{equation}\label{B}
|D\Gamma^{k+1}_2( y_{k+1})|=O(b_F^{2^{k+1}}).
\end{equation}
The estimates (\ref{abaA}) and (\ref{B})  together with Lemma \ref{vertang} gives,
\begin{equation}\label{YyY}
|\hat{y}_{k+1}-y_{k+1}|=O(b_F^{2^k}).
\end{equation}
The derivative of $\Gamma^{k+1}_2$ is bounded, see 
Lemma \ref{vertang}. Use this and the estimates (\ref{xXx}) and (\ref{YyY}) 
to get 
$$
\begin{aligned}
|\hat{x}_{k+1}-x_{k+1}|&\le |\hat{x}_{k+1}-x'_{k+1}|+|x'_{k+1}-x_{k+1}|\\
&\le |\hat{y}_{k+1}-y_{k+1}|\cdot |D\Gamma_2^{k+1}|+ O(b_F^{2^{k+1}})=
O(b_F^{2^k}).
\end{aligned}
$$
\comm{

$$
\begin{aligned}
|x_{k+1}-\hat{x}_{k+1}|&\le |x_{k+1}-x'_{k+1}|+ 
|\hat{x}_{k+1}-x'_{k+1}|\\
&=O(b^{2^{k+1}})+O(|y_{k+1}-\hat{y}_{k+1}|^2)\\
&=O(b^{2^{k+1}}).
\end{aligned}
$$

}

\noindent
This and estimate (\ref{YyY}) implies
$$
|\hat{v}^{k+1}_\infty-v^{k+1}_\infty|=O(b_F^{2^k}).
$$
The map $\Phi^{k+1}_k$ has a uniformly bounded derivative.
So
$$
|v^{k}_\infty-\Phi^{k+1}_{k}(v^{k+1}_\infty)|=O(b_F^{2^{k}}).
$$
Lemma 5.1 of \cite{CLM} gives, for $k\ge 0$,
\begin{equation}\label{uu}
|\Phi^k_0(v^{k}_\infty)-\Phi^{k+1}_{0}(v^{k+1}_\infty)|=
O(\sigma^{k}\cdot b_F^{2^{k}}).
\end{equation}
Notice,
$$
\Phi^k_0(v^{k}_\infty)\in D_{k}.
$$
This implies 
$$
\Phi^k_0(v^{k}_\infty)\to \tau,
$$
for $k\to\infty$. The estimates (\ref{uu}) implies
\begin{equation}\label{e0}
|v_\infty -\tau|\le \sum_{k=0}^\infty 
|\Phi^k_0(v^{k}_\infty)-\Phi^{k+1}_{0}(v^{k+1}_\infty)| =O(b_F).
\end{equation}
Introduce the following notation, see Figure \ref{postip2},
$$
\tau=(x,y),
$$
$$
a=(x_a,y),
$$ 
$$
v_\infty=(u_\infty, w_\infty),
$$
$$
\Gamma_2(w_\infty)=u'_2,
$$
$$
v_2=(u_2,w_2).
$$

\begin{figure}[htbp]
\begin{center}
\psfrag{Gamma2}[c][c] [0.7] [0] {\Large $\Gamma_2$}
\psfrag{Gamma00}[c][c] [0.7] [0] {\Large $\Gamma_\infty$}
\psfrag{v00}[c][c] [0.7] [0] {\Large $v_\infty$}
\psfrag{v2}[c][c] [0.7] [0] {\Large $v_2$}
\psfrag{y}[c][c] [0.7] [0] {\Large $y$}
\psfrag{x}[c][c] [0.7] [0] {\Large $x$}
\psfrag{w00}[c][c] [0.7] [0] {\Large $w_\infty$}
\psfrag{u00}[c][c] [0.7] [0] {\Large $u_\infty$}
\psfrag{w2}[c][c] [0.7] [0] {\Large $w_2$}
\psfrag{u2}[c][c] [0.7] [0] {\Large $u_2$}

\psfrag{ua}[c][c] [0.7] [0] {\Large $x_a$}
\psfrag{u2'}[c][c] [0.7] [0] {\Large $u'_2$}

\psfrag{tau}[c][c] [0.7] [0] {\Large $\tau$}
\psfrag{a}[c][c] [0.7] [0] {\Large $a$}

\pichere{0.8}{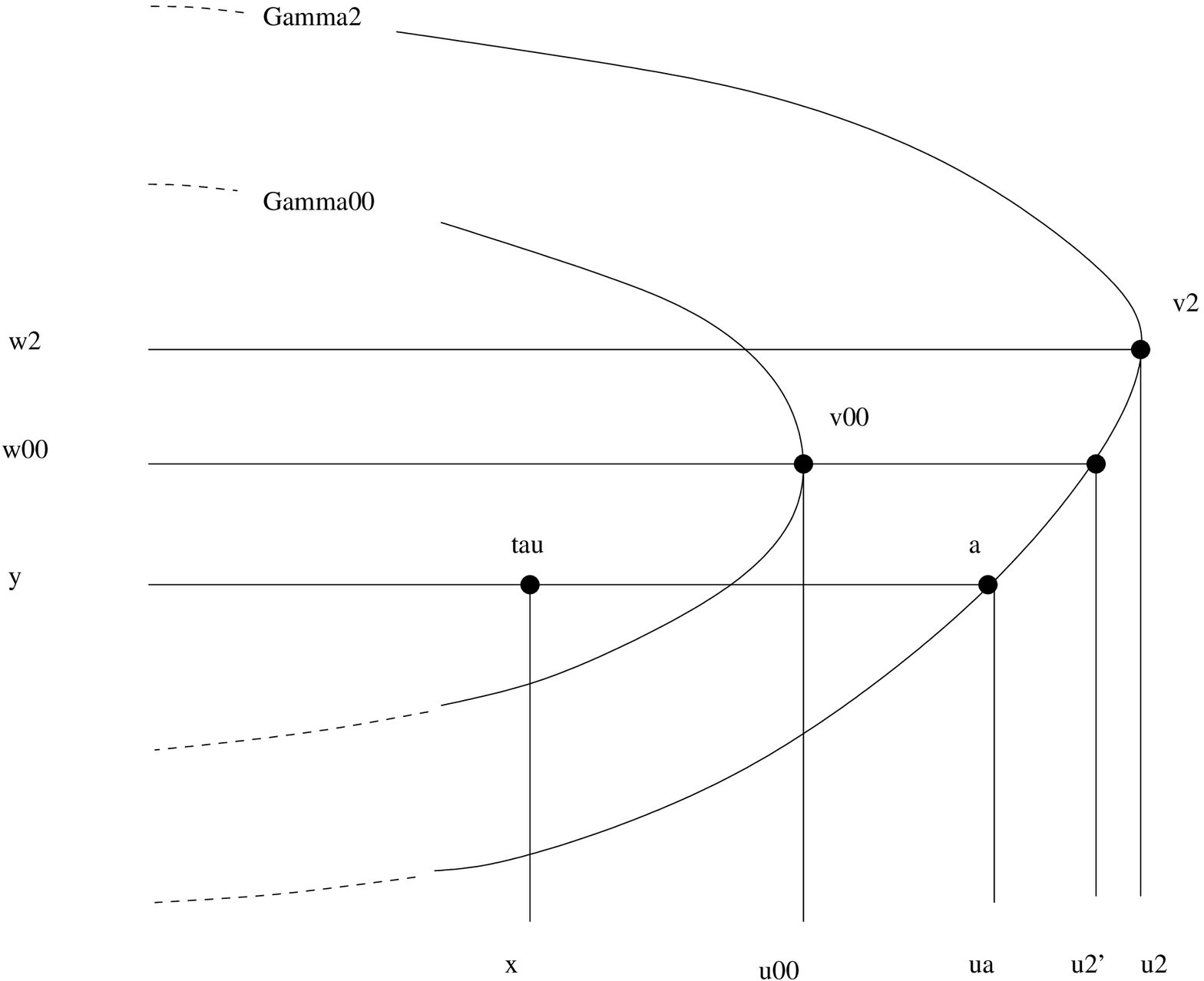}
\caption{}
\label{postip2}
\end{center}
\end{figure}

\noindent
Corollary \ref{Gammabounds} gives
\begin{equation}\label{e1}
|u_\infty-u'_2|\asymp b_F,
\end{equation}
and 
$$
|D\Gamma_2(w_\infty)= O(b_F).
$$
By definition,
$$
|D\Gamma_2(w_2)=0.
$$
So, Lemma \ref{vertang} gives
\begin{equation}\label{e2}
|w_\infty-w_2|= O(b_F).
\end{equation}
The second derivative of $\Gamma_2$ is bounded, which implies
\begin{equation}\label{e3}
|u_2-u'_2|= O(|w_\infty-w_2|^2)=O(b_F^{2}).
\end{equation}
The estimates in (\ref{e0}) and (\ref{e2})  imply
\begin{equation}\label{e4}
|y-w_2|\le |y-w_\infty|+|w_\infty-w_2|= O(b_F).
\end{equation}
We proved (2).
Use estimate (\ref{e4}) and Lemma \ref{vertang} to get
\begin{equation}\label{e5}
|x_a-u_2|=O(|y-w_2|^2)=O(b_F^{2}).
\end{equation}
We proved (1).
Next we will prove (4). First the upper bound. To do so, use the inequalities
 (\ref{e0}), (\ref{e1}), (\ref{e3}), and 
(\ref{e5}) to estimate the four term below.
\begin{equation}\label{e6}
\begin{aligned}
|\tau-a|&=|x-x_a|\\
&\le |x-u_\infty|+|u_\infty-u'_2|+|u'_2-u_2|+|u_2-x_a|\\
&=O(b_F).
\end{aligned}
\end{equation}
 The lower bound follows from inequalities
 (\ref{e1}), (\ref{e3}), and 
(\ref{e5}):
$$
\begin{aligned}
|\tau-a|&\ge |x_a-u_\infty|\\
&\ge |u'_2-u_\infty|-|u_2-u'_2|-|u_2-x_a|\\
&=O(b_F).
\end{aligned}
$$
Left is to prove (3). To do so, use  the 
inequalities (\ref{e6}), (\ref{e5}), and 
(\ref{e4}):
$$
|\tau-v_2|\le |\tau-a|+
\sqrt{|x_a-u_2|^2+|y-w_2|^2}
=O(b_F).
$$
\end{proof}

\section{The average Jacobian as topological
invariant}\label{sectopinv}

Fix $F \in \II_\Omega(\overline{\eps})$ with
$\overline{\eps}>0$ small enough.  Let
$
M^k_1\subset B
$, with  $k\ge 1$,
be as defined in \S\ref{secstabman}. 
Let $D^{\tau}\subset D_1$ be the connected 
component of $D_1\setminus W^s_{\loc}(\tau)$ which does not contain 
$\beta_1$, see Figure \ref{prfbtop}. Define,
$$
\kappa_F=\min \{k\ge 1| M_1^k\cap D^\tau \ne \emptyset\}.
$$

\begin{lem}\label{topinvkappa}
Let  $F, \tilde{F} \in \II_\Omega(\overline{\eps})$ with
$\overline{\eps}>0$ small enough. If $F$ and $\tilde{F}$ are conjugate then
$$
\kappa_{\tilde{F}}=\kappa_F.
$$
Moreover, if $h$ is a conjugation between $F$ and $\tilde{F}$ with $h(D_1)=\tilde{D}_1$ then
$$
h(D^\tau\cap M^k_1)=\tilde{D}^\tau \cap \tilde{M}^k_1.
$$
\end{lem}

\begin{proof}
The curves $M^k_1$ are not
topologically determined. However, 
 $$
 M_1^k\subset F^{-2^{k+1}}([p^k_0,p^k_2]^s),
 $$
where $[p^k_0,p^k_2]^s\subset W^s(\beta_k)$ is the curve connecting 
$p^k_0$ with $p^k_2$, see \S\ref{secstabman} and Figure 
\ref{extendedlocstabmani}. Proposition ~\ref{Wslocext} gives
$$
D^\tau \cap F^{-2^{k+1}}([p^k_0,p^k_2]^s) =  D^\tau \cap M_1^k.
$$
We may assume, see Proposition \ref{newh}, that we have a conjugation $h$ 
between $F$ and $\tilde{F}$ with
$h(D_1)=\tilde{D}_1$. The set $D^\tau \cap F^{-2^{k+1}}([p^k_0,p^k_2]^s)$ 
is topologically defined. So,
$$
\begin{aligned}
h( D^\tau \cap M_1^k)&=h(D^\tau \cap F^{-2^{k+1}}([p^k_0,p^k_2]^s))\\
& =
\tilde{D}^\tau \cap \tilde{F}^{-2^{k+1}}([\tilde{p}^k_0,\tilde{p}^k_2]^s)\\
&=
\tilde{D}^\tau \cap \tilde{M}_1^k .
\end{aligned}
$$
This means that $\kappa_F$ is a topological invariant.
\end{proof}

We will suppress the index $F$: $\kappa=\kappa_F$ and $b=b_F$, etc.
The intersection point of $M^k_1$ with the horizontal line through
$\tau$ is denoted by $z_k$. 

\begin{figure}[htbp]
\begin{center}
\psfrag{Wstau}[c][c] [0.7] [0] {\Large $W^s_\loc(\tau)$}
\psfrag{Dhat}[c][c] [0.7] [0] {\Large $D^\tau$}  
\psfrag{tau}[c][c] [0.7] [0] {\Large $\tau$} 
\psfrag{x}[c][c] [0.7][0] {\Large $C$} 
\psfrag{a}[c][c] [0.7] [0] {\Large $a$}
\psfrag{A}[c][c] [0.7] [0] {\Large $A$} 
\psfrag{L}[c][c] [0.7][0] {\Large $L$} 
\psfrag{L1}[c][c] [0.7][0] {\Large $L_1$} 
\psfrag{L2}[c][c] [0.7][0] {\Large $L_2$} 
\psfrag{M}[c][c] [0.7][0] {\Large $M^{\kappa}_1$} 
\psfrag{tau2}[c][c] [0.7] [0] {\Large $v_2$}
\psfrag{Gam2}[c][c] [0.7] [0] {\Large $\Gamma_2$}

\pichere{0.7}{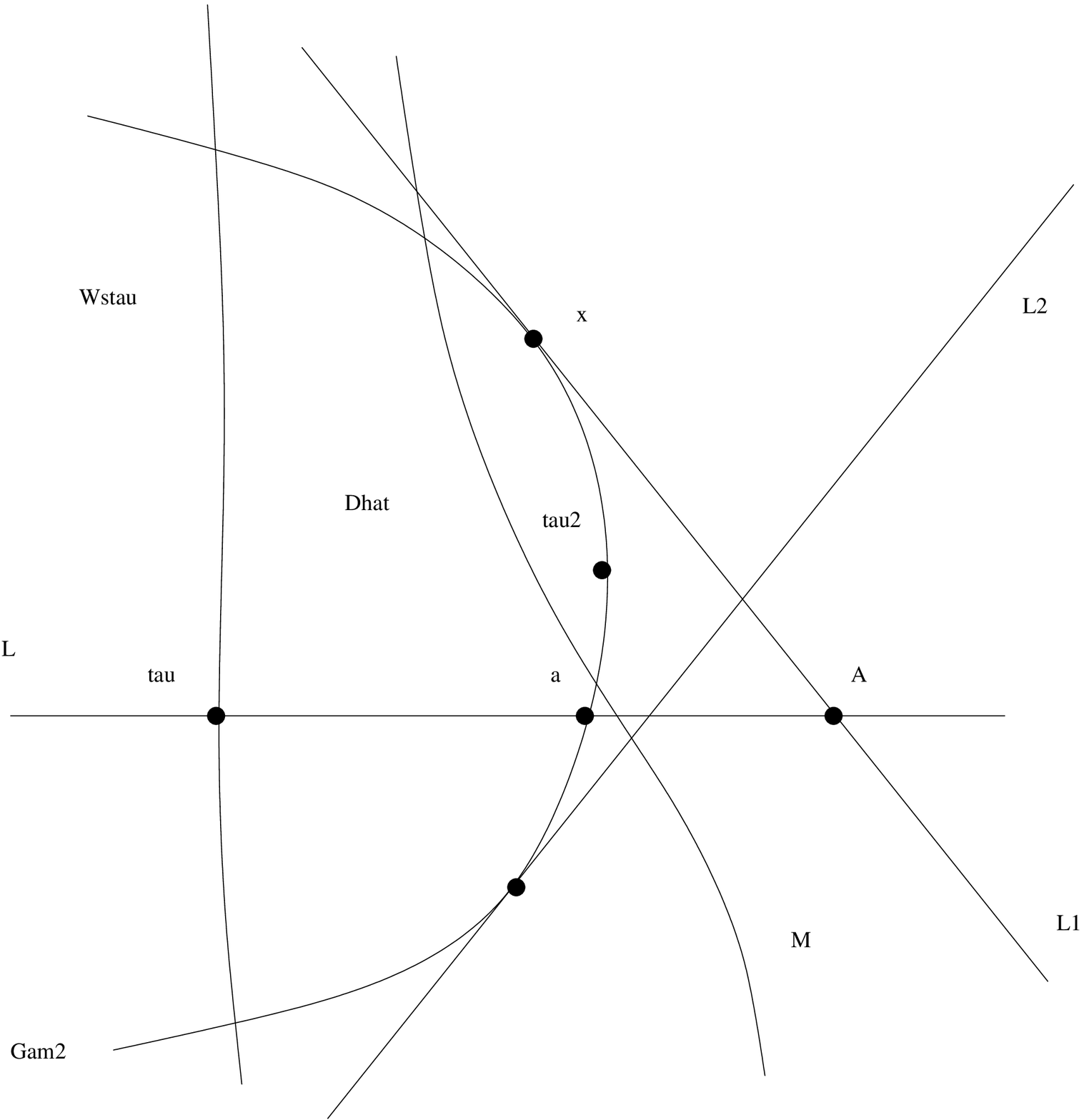} \caption{} \label{prfbtop}
\end{center}
\end{figure}

\begin{rem} The proof of Proposition \ref{ztaudist} is illustrated in
Figure \ref{prfbtop}. Again the horizontal direction has been stretched 
dramatically.
\end{rem}

\begin{prop}\label{ztaudist}
$$
|\tau-z_{\kappa}|\asymp b.
$$
\end{prop}

\begin{proof} 
 The curves $M^k_1$ are graphs over the $y$-axis with a slope bounded by
$$
|\frac{dM^k_1}{dy}|\le K_1\cdot b.
$$
Consider the curve $\Gamma_2$, as defined in \S\ref{sectip}.
This curve has a definite curvature, see Lemma \ref{vertang}. This means that there are unique tangent lines $L_1$ and $L_2$ to $\Gamma_2$ with slopes, as a graph of the $y$-axis, equal to $-K_1\cdot b$ and $K_1\cdot b$.  The connected component of 
$B\setminus (L_1\cup L_2)$ which contains $D_\tau$ is called $D$. The lines $L_1$ and $L_2$ intersect the horizontal line $L$ trough $\tau$. Let $A\in L$ be
among these intersection points the one furthest to the right, say 
$A=L\cap L_1$. Denote the tangent points of $L_1$ by $C$.
Finally, let 
$(A,\infty)\subset L$ be the connected component of $L\setminus \{A\}$ which 
does not contain $\tau$.

Observe, every curve intersecting $D$ and $(A,\infty)$ will have somewhere a tangent line with absolute value of the slope,
 with respect to the $y$-axis, larger than $K_1\cdot b$. Hence,
\begin{equation}\label{A}
M^k_1\cap D^\tau \ne \emptyset \Rightarrow M^k_1\cap [\tau, A]\ne
\emptyset.
\end{equation}
The intersection point of $\Gamma_2$ with the horizontal
line through $\tau$ is denoted by $a$. Then
\begin{equation}\label{a}
M^k_1\cap [\tau, a]\ne \emptyset \Rightarrow M^k_1\cap D^\tau\ne
\emptyset.
\end{equation}
The implication  (\ref{a})  and Proposition ~\ref{depth}(4)
imply
\begin{equation}\label{low}
|\tau-z_{\kappa-1}|\ge |\tau-a|\asymp b.
\end{equation}
Proposition ~\ref{Wslocext} states
\begin{equation}\label{tauzk}
|\tau-z_k|\asymp \sigma^{2k}.
\end{equation}
Then inequality (\ref{low}) and (\ref{tauzk}) give $K_2>0$ such that
\begin{equation}\label{upup}
|\tau-z_{\kappa}|\ge  K_2 \cdot b.
\end{equation}
An upper bound is given by (\ref{A}). Namely,
\begin{equation}\label{up}
|\tau-z_{\kappa}|\le |\tau-A|.
\end{equation}
Left is to estimate $|\tau-A|$. According to Lemma ~\ref{vertang}
and Proposition ~\ref{depth}(1) and (2) we get
$$
\aligned
|\pi_2(a)-\pi_2(v_2)|&=O(b),\\
|\pi_1(a)-\pi_1(v_2)|&=O(b^{2}).
\endaligned
$$
Lemma ~\ref{vertang}  implies
$$
\aligned
|\pi_2(C)-\pi_2(v_2)|&=O(b),\\
|\pi_1(C)-\pi_1(v_2)|&=O(b^2).
\endaligned
$$
Hence,
\begin{equation}\label{pppp}
\aligned
|\pi_2(C)-\pi_2(a)|&=O(b),\\
|\pi_1(C)-\pi_1(a)|&=O(b^2).
\endaligned
\end{equation}
The slope of $L$ with the vertical is $\pm K_1\cdot b$.
Hence,
\begin{equation}\label{q}
|\pi_1(C)-\pi_1(A)|=O(b^2).
\end{equation}
The estimates (\ref{pppp}) and (\ref{q}) imply,
$$
|a-A|=O(b^2).
$$
Together with Proposition ~\ref{depth}(4) we get
$$
|\tau-A|=|\tau-a|+|a-A|\asymp b.
$$
Thus together with (\ref{upup}) and (\ref{up}) we get
$$
|\tau-z_{\kappa}|\asymp b.
$$
\end{proof}

\begin{thm}\label{bF} Let $F \in \II_\Omega(\overline{\eps})$, with
$\overline{\eps}>0$ small enough. Then
$$
\boldsymbol{\kappa}_F=\lim_{n\to\infty} \frac{\kappa_{R^nF}}{2^{n}}=\frac12\frac{\ln b_F}{\ln \sigma}
$$
exists and is a topological invariant.
 In particular,  $b_F$ is a topological 
invariant. 
\end{thm}

\begin{proof} 
If  $F, \tilde{F}\in \II_\Omega(\overline{\eps})$ are conjugate 
then there is a conjugation $h$ between them with 
$$
h(D_n)=\tilde{D}_n, 
$$
$n\ge 0$. This follows from Lemma \ref{newh} and Proposition ~\ref{conj}.
 This means that the
conjugation $h$ induces a conjugation, denoted by $h_n$, between
$F_n=R^nF$ and $\tilde{F}_n=R^n\tilde{F}$. Thus
$$
h_n( F_n^{-2^{k+1}}([p^k_0,p^k_2]^s) )=
\tilde{F}_n^{-2^{k+1}}([\tilde{p}^k_0,\tilde{p}^k_2]^s)
$$
and
$$
h_n(D^\tau(F_n))=D^\tau(\tilde{F}_n).
$$
This implies
$$
\kappa_{\tilde{F}_n}=\kappa_{F_n}.
$$
The definition of $\boldsymbol{\kappa}_F$ is topological. 
Left is to identify its
value.  Proposition ~\ref{Wslocext} gives
$$
|\tau_{F_n}-z_{\kappa_{F_n}}|\asymp \sigma^{2\kappa_{F_n}}
$$
and
Proposition \ref{ztaudist} gives
$$
|\tau_{F_n}-z_{\kappa_{F_n}}|\asymp b^{2^n}.
$$
These two estimates imply
$$
\sigma^{2\kappa_{F_n}}\asymp b^{2^n}.
$$
Thus
$$
\boldsymbol{\kappa}_F=\lim_{n\to \infty}\frac{\kappa_{F_n}}{2^{n}}=\frac12\frac{\ln b_F}{\ln
\sigma}.
$$
\end{proof}

\section{The stable lamination}\label{stablam} 

Let $F\in \HH^{n}_\Omega(\overline{\eps})$, with
$\overline{\eps}>0$ small enough. 
The set $\PP^n_F$ consists of the periodic points of period at most 
$2^{n-1}$. Define
$$
\FF^s_n=\bigcup_{x\in\PP^n_F} W^s(x).
$$ 

\begin{lem}\label{Fsclosed} $\FF^s_n$ is closed.
\end{lem}

\begin{proof} Let $z_j\in \FF^s_n$ with $z_j\to z$. We may assume that for all $j\ge 1$ $z_j\in W^s(\beta_k)$ with $k\le n$. Suppose $z\notin \FF^s_n$. According to Theorem \ref{attrac} we have
$$
\omega(z)\subset \bigcup_{k>n} \bbe_k \cup \OO_F\subset \inter(\Trap_{n+1}).
$$
$\Trap_{n+1}$ is forward invariant. Hence, for $j\ge 1$ large 
enough,we have 
$$
\bbe_k=\omega(z)\subset \Trap_{n+1}.
$$
Contradiction, $k\le n$.
\end{proof}
 
\begin{rem} For $x\in \PP^n_F$
$$
\overline{\Orb(W^s(x))}\setminus \Orb(W^s(x))= \FF^s_{n-1}.
$$
The proof of the fact that the closure is contained in $\FF^s_{n-1}$ 
relies on Lemma \ref{1D} and the proof of Lemma \ref{Fsclosed}. The other 
inclusion follows from a statement similar to  Lemma \ref{lambdalem} but
 discussing stable manifolds. 
\end{rem}

A point $z\in\FF^s_n$ is {\it laminar} if for
any sequence  $z_j\in \FF^s_n$ with $z_n\to z$
$$
T_{z_j}W^s(z_j) \to T_{z}W^s(z).
$$
The set $\FF^s_n$ is {\it laminar} if all its points 
are laminar.

\begin{thm} \label{lams}Let $F\in \HH^{n}_\Omega(\overline{\eps})$, with
$\overline{\eps}>0$ small enough. 
If 
$$
F\notin \bigcup_{k'<n'\le n } \KK_{k',n'}
$$ then 
$\FF^s_n$ is laminar.
\end{thm}

The proof of this Theorem is similar to the proof of 
Theorem ~\ref{lam} with some modifications, see the proof of Claim ~\ref{orbzs}. For completeness we include the 
proof. 
Using the notation of \S\ref{Lamstruc} we will choose the interval
$$
K^s_n=[p_0,p_2]^s\subset W^s_\loc(\beta_n),
$$
as a fundamental domain in $W^s(\beta_n)$, see Figure ~\ref{saddlereg}. 
For $k<n$ define,
$$
E^s_{n,k}=\{x\in K^s_n| \text{   } \exists t>0 \text{  }\forall j<t
\text{   } F^{-j}(x)\notin T_k \text{ and } F^{-t}(x)\in T_k\}.
$$
The time $t>0$ in the above definition is called the {\it time of entry}
 of $x\in E^s_{n,k}$ into $T_k$.

\begin{defn}\label{Indkns} Let $k<n$.
We say that $F$ satisfies the transversality condition
$\mathcal{T}^s_{n,k}$ if the following holds. Let $z_j\in
E^s_{n,k}$, $j\ge 0$, be a sequence such that
$$
F^{-t_j}(z_j)\rightarrow u\in U_k,
$$
where $t_j>0$ is the time of entry of $z_j$ into $T_k$, then
$$
DF^{-t_j}(z_j)(T_{z_j}W^s(\beta_n))\nrightarrow T_uW^u(\beta_k).
$$
\end{defn}

\begin{prop} \label{propTsN}
Let $F\in \HH^{n}_\Omega(\overline{\eps})$,  
$\overline{\eps}>0$ small enough. Let $k<n$ and  
$$
F \notin \bigcup_{k\le k'<n'\le  n} \KK_{k',n'}(\overline{\eps})
$$ 
then 
$
\mathcal{T}^s_{n,k}
$
holds.
\end{prop}

\begin{proof}
Choose a sequence $z_j\in E^s_{n,k}$, $j\ge 0$, with  $z_j\rightarrow z$  and
$$
F^{-t_j}(z_j)\rightarrow u\in \inter(U_{n}),
$$
where $t_j>0$ is the time of entry of $z_j$ into $T_{k}$.
If the $t_j$ are bounded, say constant $t_j=t$,
the absence of heteroclinic tangencies, 
$F\notin \KK_{k,n}(\overline{\eps})$, implies that
$$
DF^{-t}(z)(T_{z}W^s(\beta_k))\ne T_{u}U_{k}.
$$
Hence,
$$
DF^{-t_j}(z_j)(T_{z_j}W^s(\beta_k))\nrightarrow T_uU_{k}.
$$
Let us continue with the case when $t_j\to \infty$.

\begin{clm} \label{orbzs} There exists $n>m_1\ge k$ such that
$
z\in E^s_{n,m_1}\cap W^u(\bbe_{m_1}).
$
\end{clm}

\begin{proof} Let $t\ge 0$ be an arbitrary moment of time. Assume
$$
F^{-t}(z)\notin \Trap_{k}.
$$
This means that for $j\ge 1$ large enough,
$$
F^{-t}(z_j)\notin \Trap_{k},
$$
because $\Trap_k$ is closed.
The invariance $F(\Trap_{k})\subset \Trap_{k}$ implies that for all $j\ge 1$
\begin{equation}\label{trapk}
F^{-t_j}(z_j)\notin \Trap_{k}.
\end{equation}
The construction of $\Trap_k$ implies that
\begin{equation}\label{invtrap}
F^{2^k}(\Trap_k)\subset \inter(\Trap_k).
\end{equation}
From (\ref{trapk}) and (\ref{invtrap})
we get
$$
F^{-t_j}(z_j)\to u\notin \Trap_{k}.
$$
Contradiction. 
So for each $t\ge 0$ 
$$
F^{-t}(z)\in \Trap_{k}.
$$
This means
\begin{equation}\label{alpha}
\alpha(z)\subset \Trap_k.
\end{equation}
Suppose, 
$$
\alpha(z)\cap \Trap_{n+1}\ne \emptyset.
$$
 Then (\ref{invtrap}) implies
that  
$$
\alpha(z)\cap \inter(\Trap_{n+1})\ne \emptyset.
$$ So there is a $t>0$ such that $F^{-t}(z)\in \Trap_{n+1}$. Again, (\ref{invtrap}) implies
$$
\bbe_n=\omega(z)\subset \Trap_{n+1}.
$$
Contradiction. Thus
$$
\alpha(z)\cap \Trap_{n+1}=\emptyset.
$$
According to Theorem ~\ref{attrac} every non periodic orbit outside 
$\Trap_{n+1}$ will enter $\Trap_{n+1}$. Hence, for some $m_1\le n$
$$
\alpha(z)=\bbe_{m_1}.
$$
In particular,
$$
z\in W^u(\bbe_{m_1}).
$$
There are no homoclinic orbits. Hence,  $m_1<n$.

Left is to show that $k\le m_1$. Observe, $\alpha(z)\subset \Trap_k$, see (\ref{alpha}), and $\alpha(z)=\bbe_{m_1}$. If $m_1<k$ then $\bbe_{m_1}\cap \Trap_k=\emptyset$. So $k\le m_1<n$, which finishes the proof of the Claim.
\end{proof}

\comm{
========================================
Suppose, $m_1<n$. There exists $s>0$ such that 
$$
F^{-s}(z)\in T_{m_1}.
$$
For every $j\ge 0$ there exists $s_j>0$ such that 
$$
F^{-s_j}(z_j)\in T_{m_1}.
$$
Observe, $u\notin \{z,F^{-1}(z),\dots, F^{-s}(z)\}$. This implies $s_j<t_j$.
Furthermore, 
$$
\Trap_{n}\cap T_{m_1}=\emptyset. 
$$
So,
$$
F^{-t_j}(z_j)\notin \Trap_n.
$$
This implies
$$
u\notin \inter(\Trap_n).
$$
Contradiction, $z\in W^u(\beta_{m_1})$ for some $n\le m_1< k$.
====================
}

Denote the time of  entry of $z$ into $U_{m_1}\subset T_{m_1}$ by $r_1>0$ and
let $F^{-r_1}(z)=u_1$. We will call $r_1$ the first {\it transient time}.
 For $j>0$ large enough, $z_j\in E^s_{n,m_1}$ with
corresponding time of entry $t^1_j=r_1$.
The absence of heteroclinic tangencies, 
$F\notin \KK_{m_1,n}(\overline{\eps})$,  implies that
$$
DF^{-r_1}(z)(T_{z}W^s(\beta_n))\ne T_{u_1}U_{m_1}.
$$
Hence,
\begin{equation}\label{zjss}
DF^{-t^1_j}(z_j)(T_{z_j}W^s(\beta_n))\nrightarrow T_{u_1}U_{m_1}.
\end{equation}
In the case when  $m_1=k$ we proved that the sequence $z_j$ satisfies the transversality
condition.

Consider the case when
$m_1> k$. Let $e^1_j>0$ be such that for $i=r_1,r_1+1,\dots, e_j^1$
$$
F^{-i}(z_j)\in T_{m_1}
$$
but
$$
F^{-(e^1_j+1)}(z_j)\notin T_{m_1}
$$
The moment $e^1_j$ is called the {\it time of exit} of $z_j$ from $T_{m_1}$.
We may assume that $F^{-e^1_j}(z_j)\rightarrow s_1\in S_{m_1}$.
Then ~(\ref{zjss}) implies
\begin{equation}\label{zjus}
DF^{-e^1_j}(z_j)(T_{z_j}W^s(\beta_k))\rightarrow T_{s_1}S_{m_1}.
\end{equation}

Now, we can repeat the proof of Claim ~\ref{orbzs} and obtain 
$k\le m_2<m_1<n$ and $r_2>0$, the second transient time,  such that
$$
F^{-r_2}(s_1)=u_2\in U_{m_2}.
$$
For $j>0$ large enough we have $z_j\in E^s_{n,m_2}$. Denote the time of entry 
of $z_j$ into $T_{m_2}$ by $t^2_j>0$ then $t^2_j=e^1_j+r_2$.
The absence of heteroclinic tangencies, 
$F\notin \KK_{m_2,m_1}(\overline{\eps})$, implies that
$$
DF^{-r_2}(s_1)(T_{s_1}W^s(\beta_{m_1}))\ne T_{u_2}U_{m_2}.
$$
Hence, ~(\ref{zjus}) implies
\begin{equation}\label{zjs2s}
DF^{-t^2_j}(z_j)(T_{z_j}W^s(\beta_n))\nrightarrow T_{u_2}U_{m_2}.
\end{equation}
Let $e^2_j>0$ be maximal such that for $i=t^2_j, t^2_j+1, \dots, e_j^2$
$$
F^{-i}(z_j)\in T_{m_2} .
$$
but
$$
F^{-(e^2_j+1)}(z_j)\notin T_{m_2} .
$$
We may assume that $F^{-e^2_j}(z_j)\rightarrow s_2\in S_{m_2}$.
Then
\begin{equation}
DF^{-e^2_j}(z_j)(T_{z_j}W^s(\beta_n))\rightarrow T_{s_2}S_{m_2}.
\end{equation}
If $m_2=k$, statement  ~(\ref{zjs2s}) proves the transversality
property. In the case when $m_2>k$ we can repeat this
construction, and we get a sequence $m_1> m_2>m_3>\dots> m_g$
together with points $u_l\in U_{m_l}$, $s_l\in S_{m_l}$ and
entry and exit times $t^l_j>0$ and  $e^l_j>0$ for $z_j\in
E^s_{n,m_l}$ and the
 corresponding
asymptotic expressions ~(\ref{zjss}) and  ~(\ref{zjus}).

The sequence $m_l$ is strictly decreasing. Hence, $m_g=k$ and $t_j=t^g_j$
for some $g\ge 1$.
Now, statement ~(\ref{zjss}) corresponding to $T_{m_g}$,
$$
DF^{-t_j}(z_j)(T_{z_j}W^s(\beta_k))\nrightarrow T_{u}U_{k},
$$
finishes the proof of  the Proposition.
\end{proof}

\bigskip

\noindent
{\it Proof of Theorem ~\ref{lams}.} Choose $k\le n$. To prove that every point
in $W^s(\beta_k)$ is laminar it suffices to prove that every point $z\in S_k$
is laminar. From Lemma ~\ref{1D} we have that 
$W^s(\beta_k)$ is an embedded one-dimensional manifold. Hence, the only non-trivial accumulation is from $W^u(\beta_{k'})$ with $n\ge k'>k$. Let $k'>k$ and 
$z_j\in E^s_{k',k}$ be a sequence with
$$
F^{-t_j}(z_j)\to u\in U_k,
$$
and
$$
F^{-e_j}(z_j)\to z \in S_k
$$
with $e_j>t_j$.

Proposition \ref{propTsN} states that $\mathcal{T}^s_{k',k}$ holds.
Now $\mathcal{T}^s_{k',k}$ implies that
$$
DF^{-t_j}(z_j)(T_{z_j}W^s(\beta_{k'}))\nrightarrow T_{u}U_{k}.
$$
Hence, according to the $\lambda$-Lemma,
$$
DF^{-e_j}(z_j)(T_{z_j}W^s(\beta_{k'}))\rightarrow T_{z}S_{k}.
$$
\qed

The proof of the following Theorem is similar to the proof of 
Theorem ~\ref{CF}. We will omit the proof. 
For a $F\in \HH^{n}_\Omega(\overline{\eps})$ let $\CC^s_n\subset\FF^s_n $ be the set of non-laminar points of $\FF^s_n$. The stable and unstable eigenvalues of $\beta_k$ are denoted by $\lambda_k$ and $\mu_k$, see \S \ref{hettan}.

\begin{thm}\label{CFs} If  $F\in \HH^{n}_\Omega(\overline{\eps})$, 
with $\overline{\eps}>0$ small enough,
has an $(k,n)$-heteroclinic tangency with $k<n$ and
$$
\frac{\ln |\lambda_k|}{\ln |\mu_n|}\notin \mathbb{Q}
$$
then  
$$
\FF^s_k\subset \CC^s_n.
$$
\end{thm}

\comm{
=====================
\note{where do we put this? In the question section?}

For $F\in \II_\Omega(\overline{\eps})$, with
$\overline{\eps}>0$ small enough, let
$$
\FF^s=\FF^s(F)=\bigcup_{k\ge 1} \FF^s_k(F)
$$
and
$$
\FF^s_\tau=\FF^s_\tau(F)=\bigcup_{k\in \mathbb{Z}} W^s(F^k(\tau)).
$$

\begin{thm}\label{denseWs} For $F\in \II_\Omega(\overline{\eps})$, with
$\overline{\eps}>0$ small enough,
$$
\overline{\FF^s(F)}=\text{Domain}(F).
$$
and 
$$
\overline{\FF^s_\tau(F)}=\text{Domain}(F).
$$
\end{thm}

\begin{proof}
\end{proof}

======================
}

\comm{

\section{Structural stability}

\note{ what do we do with this Theorem?/section}

\bigskip

\begin{thm} Let $F_t$, $t\in [0,1]$, be a family of H\'enon like maps in
$\mathcal{I}_\Omega(\overline{\epsilon})$, with $\overline{\epsilon}>0$ small
enough, and such that for each $t\in [0,1]$ $F_t$ does not have heteroclinic
tangencies. Then $F_0$ and $F_1$ are topologically conjugate.
\end{thm}

\begin{proof}
\end{proof}

}

\section{Morse-Smale components}\label{Denhyp} 

A map $F:B\to B$ is {\it Morse-Smale} if the non-wandering set $\Omega_F$ consists of finitely many periodic points, all hyperbolic, and the stable and unstable manifolds of the periodic points are all transversal to each other. Recall, the collection $\II^n_\Omega(\overline{\eps})$ consists of the maps which are exactly $n$-times renormalizable and has a periodic attractor of period $2^n$.
According to Lemma ~\ref{OmegaN} the non-wandering set of each  map 
$F\in \II^n_\Omega(\overline{\eps})$, 
with $\overline{\eps}>0$ small enough, consists of finitely many periodic 
points. In particular, a map $F\in \II^n_\Omega(\overline{\eps})$ is Morse-Smale if all its periodic points are hyperbolic and 
if for every 
$x,y\in \PP_F=\Omega_F$ there are only transverse intersections of  
$W^u(x)$ and $W^s(y)$.

\begin{thm}\label{hyp} Let $\overline{\eps}>0$ be small enough.
The Morse-Smale maps 
form an open and 
dense subset of any $\II^n_\Omega(\overline{\eps})$.
\end{thm}

A {\it Morse-Smale component} is a connected component of the set of  non-degenerate Morse-Smale maps in $\HH_\Omega(\overline{\eps})$. Morse-Smale maps are structurally stable, see \cite{P}.

\begin{prop}\label{comp} Let
 $F, \tilde{F}\in\II^n_\Omega(\overline{\eps})$, 
with $\overline{\eps}>0$ small enough, be in the same Morse-Smale component. 
Then $F$ and $\tilde{F}$ are conjugate.
\end{prop}

Two Morse-Smale components in $\II^n_\Omega(\overline{\eps})$ 
are of different {\it type} if the maps in the first component 
are not conjugate to the maps in the other. 

\begin{rem} In this discussion we will only consider non-degenerate H\'enon
 maps. Observe, if $F\in \II^n_\Omega(\overline{\eps})$ is a unimodal map, it can be of three different topological types depending the relative position of the attracting fixed point $p$  and the critical point $c$  of the unimodal map which describes the $n^{th}$-renormalization: $p<c$, $p=c$, and $p>c$. The non-degenerate H\'enon maps in the Morse-Smale component which contains perturbations of the unimodal maps in  $F\in \II^n_\Omega(\overline{\eps})$ are all conjugated. There is no difference in the topology of the periodic attractor anymore. 
\end{rem}

\begin{thm} \label{hypcomp} Let $\overline{\eps}>0$ be small enough.
Then for $n\ge 1$ large enough there are countably many Morse-Smale components 
of different type in $\II^n_\Omega(\overline{\eps})
$.  
\end{thm}

There are a non-locally finite collections of bifurcation curves in 
H\'enon-families. Some of these collections are constructed in the proof of 
Theorem \ref{hypcomp}, they are illustrated in Figure \ref{bif2}.

\begin{figure}[htbp]
\begin{center}
\psfrag{a}[c][c] [0.7] [0] {\Large $a$}
\psfrag{t}[c][c] [0.7] [0] {\Large $t$} 

\psfrag{0}[c][c] [0.7] [0]{\Large $0$} 
\psfrag{In}[c][c] [0.7] [0]{\Large $I_m$} 
\psfrag{Mn}[c][c] [0.7] [0] {\Large $M_m$}
\psfrag{K}[c][c] [0.7] [0] {\Large $\KK_{\cdot, \cdot}$} 
\psfrag{Kkn}[c][c] [0.7] [0] {\Large $\KK_{k,n}$} 
\psfrag{W}[c][c] [0.7] [0] {\Large $W$} 
\psfrag{jkn}[c][c] [0.7] [0] {\Large $j_{k}(F_{n-1})=j$}

\pichere{0.9}{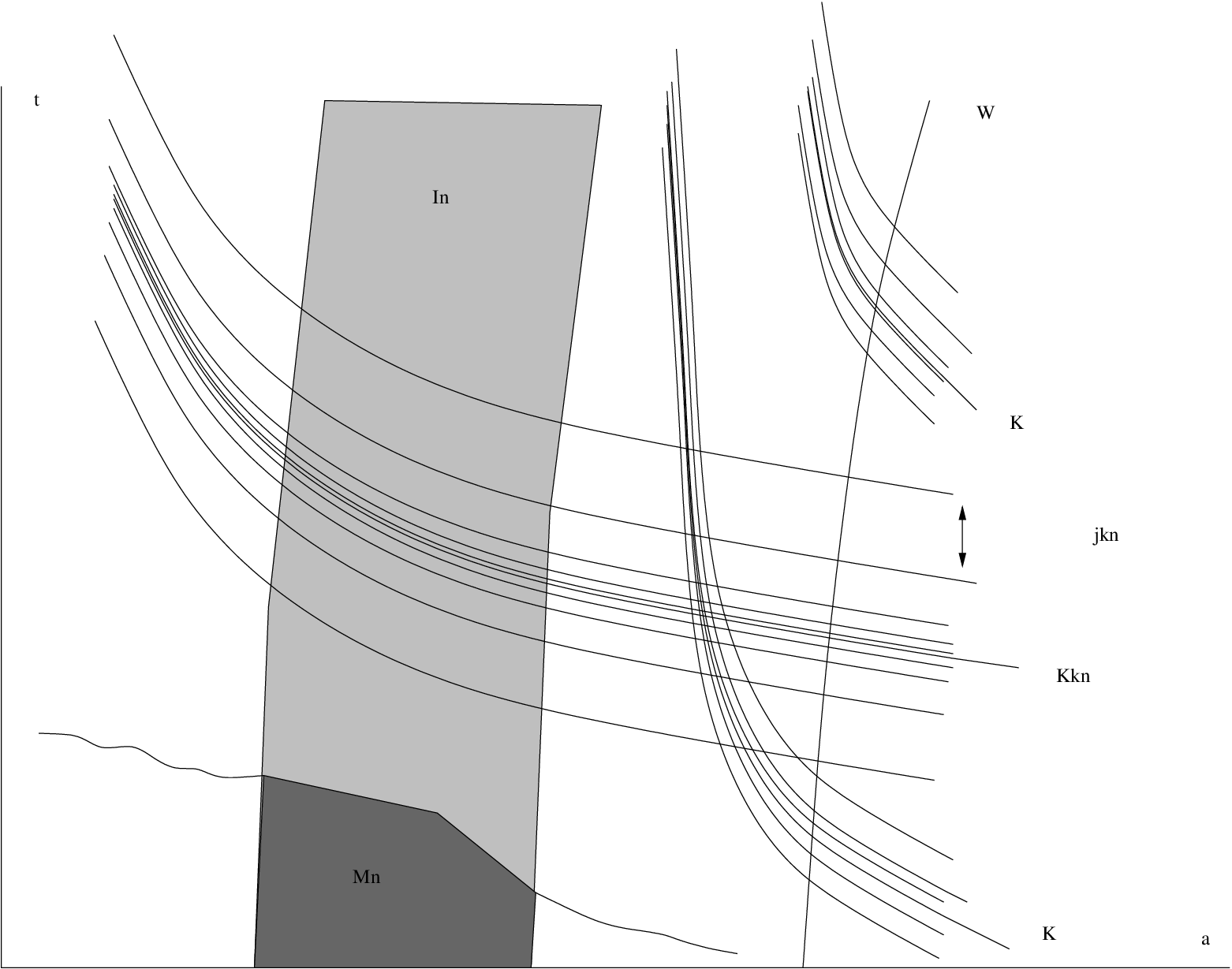} \caption{Bifurcation pattern }
 \label{bif2}
\end{center}
\end{figure}

The actual proofs of these  Theorems need some preparation. Recall,
$$
\UU\KK_{k,n}(\overline{\eps})=\bigcup_{k\le k'<n'\le n} \KK_{k',n'}(\overline{\eps}).
$$

\begin{lem}\label{Kknclosed} If $k<n$ and $\overline{\eps}>0$ small enough then $$
\overline{\KK_{k,n}(\overline{\eps})}\subset \UU\KK_{k,n}(\overline{\eps}).
$$
In particular,
$
\UU\KK_{k,n}(\overline{\eps})
$
is closed.
\end{lem}

\begin{proof} It suffices to prove that for each $k<n$ 
$$
\overline{\KK_{k,n}(\overline{\eps})}\subset \UU\KK_{k,n}(\overline{\eps}).
$$
Let $F_j\in \KK_{k,n}(\overline{\eps})$ with $F_j\to F\notin 
\UU\KK_{k,n}(\overline{\eps})$.
Let $K^{u/s}_m\subset  W^{u/s}_\loc(\beta_m)$ be a fundamental domain for $F$ restricted to $W^{u/s}(\bbe_n)$.  We may assume that 
$K^u_m\subset \inter U_m$ and $K^s_m\subset \inter S_m$.
Similarly,  let $K^{u/s}_m(j)\subset 
W^{u/s}_\loc(\beta_m(j))$ be a fundamental domain for $F_j$ 
restricted to $W^{u/s}(\bbe_n(j))$. Construct these fundamental domains 
in such a way  that
$$
K^{u/s}_m(j)\to K^{u/s}_m.
$$
 Apply Lemma 
\ref{manifinsaddle} and we see that $K^u_m(j)\subset \inter U_m(j)$ and 
$K^s_m(j)\subset \inter S_m(j)$, for $j\ge 1$ large enough.
Finally, choose $x_j\in K^{u}_k(j)$ such that 
$$
F^{t_j}(x_j)\in K^{s}_n(j)
$$
is a heteroclinic tangency and $x_j\to x\in K^u_k$.

\begin{clm}\label{attracx} $\omega(x)\subset \OO_F\cup \bigcup_{m>n} \bbe_m$
\end{clm}

\begin{proof} Suppose by contradiction 
$
\omega(x)=\bbe_{m_1},
$
say
$$
s_1=F^{r_1}(x)\in K^s_{m_1},
$$
with $m_1\le n$.
>From $F\notin \KK_{k,m_1}(\overline{\epsilon})$ we get
that 
$$
T_{s_1}F^{r_1}(K^u_k) \transverse T_{s_1}K^s_{m_1}.
$$
By definition
\begin{equation}\label{limU}
T_{x_j}K^u_k(j)\to T_xK^u_k.
\end{equation}
This implies
\begin{equation}\label{limS}
DF_j^{r_1}(x_j)(T_{x_j}K^u_k(j))\nrightarrow T_{s_1}K^s_{m_1}.
\end{equation}
Now we will  prove
\begin{equation}\label{m1}
m_1<n.
\end{equation}
To do so, assume that $m_1=n$. Observe, $F^{r_1}(x_j)\in W^s(\bbe_n(j))$ and
this point is also close to $S_n(j)$ because it is close to $S_n$. 
Lemma \ref{manifinsaddle} implies
$$
F^{r_1}_j(x_j)\in K^s_n(j).
$$ 
Hence, $t_j=r_1$. From (\ref{limS}) we get that  at $F^{r_1}_j(x_j)$ there is  
no tangency between $F^{r_1}_j(K^u_k(j))$ 
and $K^s_{n}(j)$. 
Contradiction. We proved that $m_1<n$.

Let $e^1_j>0$ be maximal such that when  $r_1\le i\le  e_j^1$ we have 
$$
F^i_j(x_j)\in T_{m_1}(j),
$$
where $T_{m_1}(j)$ is the saddle region of $\beta_{m_1}(j)$ of $F_j$, see 
Figure \ref{saddlereg}. Say
$$
F^{e^1_j}_j(x_j)\to u_1 \in K^u_{m_1}.
$$
Then
\begin{equation}\label{limU2}
DF_j^{e^1_j}(x_j)(T_{x_j}K^u_k(j))\rightarrow T_{u_1}K^u_{m_1}.
\end{equation}
Recall, $x_j\in W^s(\beta_n(j))$. Hence,
$$
\omega(u_1)\subset \bigcup_{m\le n} \bbe_m.
$$
Say,
$$
s_2=F^{r_2}(u_1)\in K^s_{m_2},
$$
with $m_1< m_2\le n$. Because, $F\notin \KK_{m_1,m_2}(\overline{\epsilon})$ 
we get
$$
T_{s_2}F^{r_2}K^u_{m_1} \transverse T_{s_2}K^s_{m_2}.
$$
This implies
\begin{equation}\label{limS2}
DF_j^{r_2+e^1_j}(x_j)(T_{x_j}K^u_k(j))\nrightarrow T_{s_2}K^s_{m_2}.
\end{equation}
 As before, we conclude
\begin{equation}\label{m2}
m_1< m_2<n.
\end{equation}
The statements (\ref{limS2}) and (\ref{limU2}) are similar to the statements
(\ref{limS}) and (\ref{limU}). We can repeat the construction in a similar
 manner as was done in the proof of Proposition \ref{propTN}. Properties 
(\ref{m1}) 
and (\ref{m2}) show that the construction can always be repeated.
 This is impossible.
\end{proof}

The first consequence of the Claim is that
$$
\AAA_F^{n+1}\ne \emptyset.
$$
Choose an open  neighborhood $U\supset \AAA_F^{n+1}$ such that
$$
\overline{F(U)}\subset U
$$
and 
\begin{equation}\label{AAAA}
\overline{U}\cap K^s_n=\emptyset.
\end{equation}
There exists $s_0>0$ such that
$F^s(x)\in U$, $s\ge s_0$. For $j\ge 1$ large enough
$$
\overline{F_j(U)}\subset U.
$$
Hence, for $j\ge 1$ large enough and $s\ge s_0$ 
$$
F_j^s(x_j)\in U.
$$
From (\ref{AAAA}) we get
$t_j\le s_0$. Say, 
$
t_j=t.
$
 Observe,
$$
F^t(x)\leftarrow F^t_j(x_j)\in K^s_n(j)\to K^s_n.
$$
So, $F^t(x)\in K^s_n$. This contradicts Claim \ref{attracx}.
\end{proof}

Consider the set $A_{k,n}\subset K^s_n$ consisting of tangencies of 
$W^u(\beta_k)$ and $W^s(\beta_n)$.

\begin{prop}\label{finitetang} Let 
 $F\in \HH^n_\Omega(\overline{\eps})$, 
with $\overline{\eps}>0$ small enough, and $k<n$ be such that
$F$ has a $(k,n)$-heteroclinic tangency,
$$
F\in\mathcal{K}_{k,n}(\overline{\eps})
$$
but
\begin{equation}\label{condK}
F\notin \UU\KK_{k+1,n}(\overline{\eps}).
\end{equation}
Then $A_{k,n}$ is a finite set.
\end{prop}

\begin{proof} 
Let 
$$
W=\bigcup_{j=k}^{n-1} W^u(\beta_j).
$$  
  Lemma \ref{1Dclos}
states that the only points in $K^s_n$ on which $W^u(\beta_k)$ can 
accumulate are points in the unstable manifolds $W^u(\beta_j)$ 
with $k<j<n$.  Hence, if 
$x\in \overline{A_{k,n}}\setminus A_{k,n}$ then
$$
x\in W^u(\beta_j)\cap W^s(\beta_n)
$$ 
for some $j>0$ with  $k< j<n$.

Condition (\ref{condK}) says that $F\notin \KK_{j,n}$. So,
the intersection at $x$ between  $W^u(\beta_j)$ and $W^s(\beta_n)$   
is transverse. The point $x$ is accumulated by heteroclinic 
tangencies. It is not a laminar point of
$$
\bigcup_{k\le j\le n} W^u(\bbe_j).
$$
This contradicts Proposition \ref{proplam}.
\end{proof}

\begin{rem} Observe that the intersection $W\cap K^s_n$, used in the previous 
proof, is closed. However, it is not finite if $k<n-1$. 
Compare with Lemma ~\ref{1Dclos} . 
\end{rem}

Choose $F\in \HH^n_\Omega(\overline{\eps})$, 
with $\overline{\eps}>0$ small enough. Fix the  fundamental domain  
$K^s_n=[p^n_0,p^n_2]^s\subset W^s_\loc(\beta_n)$ for $F$ restricted to 
$W^s(\bbe_n)$. The points $p^n_0, p^n_2$ are as 
defined in section \S \ref{Lamstruc}, see Figure ~\ref{saddlereg}. Similarly, 
we choose
$K^u_k=[p^{k+1}_{-2}, p^{k+1}_{-1}]^u$ as a fundamental domain in 
$W^u(\bbe_k)$ with $k<n$. 
Given a sequence $F_j\to F$ we will denoted the fundamental domains of $F_j$ by
$K^{u/s}_k(j)$. The invariant manifolds $W^{u/s}(\beta_k(F_j))$ of $F_j$ will be denoted by $W^{u/s}_k(j)$ and $W^{u/s}(\beta_k(F))$ of $F$ will be denoted by $W^{u/s}_k$.

\begin{lem}\label{boundedtime} Let 
 $F\in \HH^n_\Omega(\overline{\eps})$, 
with $\overline{\eps}>0$ small enough, and $k<n$ be such that
$$
F\notin \UU\KK_{k+1,n}(\overline{\eps})\cup \UU\KK_{k,n-1}(\overline{\eps}).
$$
There exists $N\ge 1$ such that the following holds. If $F_j\to F$ and there are points
\begin{equation}\label{conv1}
u(j)\in K^u_k(j)
\end{equation}
\begin{equation}\label{conv2}
z(j)=F^{t_j}_j(u(j))
\end{equation}
such that
\begin{equation}\label{conv3}
z(j)\to \hat{z}\in K^s_n
\end{equation}
and
\begin{equation}\label{conv4}
T_{z(j)}W^u_k(j)\ni v(j)\to T_{\hat{z}}K^s_n
\end{equation}
then
$$
t_j\le N.
$$
\end{lem}

\begin{proof} Fix $k\ge 1$. The proof will be given by induction in $n>k$. The definition of renormalization implies that the Lemma holds when $n=k+1$. In this case there are only two intersections of $W^u_k(j)$ with $K^s_n(j)$ and these 
intersections are transversal. There is nothing to prove.

Suppose the Lemma holds for all $k+1\le m<n$. Suppose, $F_j\to F$ and this sequence satisfies the conditions (\ref{conv1}), \dots, 
(\ref{conv4}) of the Lemma but 
$$
t_j\to \infty.
$$

\begin{clm} There is $k<m<n$ with $\hat{z}\in W^u_m$.
\end{clm}

\begin{proof}
Suppose the backward orbit of $\hat{z}$ escapes from $\Trap_k(F)$:
$$
F^{-t_0}(\hat{z})\notin \Trap_k(F),
$$
 for some $t_0>0$. Observe, $\Trap_k(F)$ is closed and $\Trap_k(F_j)$ is close 
to $\Trap_k(F)$ for $j\ge 1$ large enough. So for $j\ge 1$ large enough we 
have
$$
F^{-t_0}_j(\hat{z}(j))\notin \Trap_k(F_j).
$$
This contradicts, $\hat{z}(j)\in W_k^u(j)$.
We showed that the backward orbit of $\hat{z}$ does not escape: 
for every $t\ge 0$
$$
F^{-t}(\hat{z})\in \Trap_k(F).
$$
Hence,
$$
\hat{z}\in \bigcup_{k\le m<n} W^u_m.
$$
Suppose, $\hat{z}\in W^u_k$. Choose a neighborhood $U\supset \AAA_F^{k+1}$ such that $F(U)$ is strictly contained in $U$ and $U\cap [\beta_k, \hat{z}]^u=\emptyset$. For $j\ge $ large enough we have that $F_j$ also maps $U$ strictly inside $U$. In particular, $\AAA_{F_j^{k+1}}$ is strictly contained in $U$. Hence, 
for all $l\ge l_0$
$$
F^l_j(u(j))\in U
$$
This contradicts
$$
F_j^{t_j}(u(j))=z(j)\to \hat{z}\notin U
$$
because $t_j\to \infty$.
\end{proof}

Let $\hat{z}(j)\in W^u_m(j)\cap K^s_n(j)$ be the perturbation of $\hat{z}\in W^u_m\cap K^s_n$, see Figure \ref{piclem}. The intersection at $\hat{z}$ of $W^u_m$ with $K^s_n$ is transversal because
$F\notin \KK_{m,n}$: the perturbation $\hat{z}(j)$ is well defined.

\begin{figure}[htbp]
\begin{center}
\psfrag{Fj}[c][c] [0.7] [0] {\Large $F^{b_t}_{j_t}$}
\psfrag{Fa}[c][c] [0.7] [0] {\Large $F^{a_t}_{j_t}$}
\psfrag{Fs}[c][c] [0.7] [0] {\Large $F^{t}_{j_t}$}
\psfrag{bk}[c][c] [0.7] [0] {\Large $\beta_k(j_t)$}
\psfrag{u}[c][c] [0.7] [0] {\Large $u(j_t)$}
\psfrag{Kuk}[c][c] [0.7] [0] {\Large $K^u_k(j_t)$}
\psfrag{Ksm}[c][c] [0.7] [0] {\Large $K^s_m(j_t)$}

\psfrag{w}[c][c] [0.7] [0] {\Large $w_t$}
\psfrag{vsa}[c][c] [0.7] [0] {\Large $v_{t+a_t}(j_t)$}

\psfrag{zs}[c][c] [0.7] [0] {\Large $z_t(j_t)$}
\psfrag{vs}[c][c] [0.7] [0] {\Large $v_t(j_t)$}
\psfrag{bm}[c][c] [0.7] [0] {\Large $\beta_m(j_t)$}

\psfrag{zsh}[c][c] [0.7] [0] {\Large $\hat{z}_t(j_t)$}
\psfrag{zj}[c][c] [0.7] [0] {\Large $z(j_t)$}

\psfrag{Ksn}[c][c] [0.7] [0] {\Large $K^s_n(j_t)$}
\psfrag{zh}[c][c] [0.7] [0] {\Large $\hat{z}(j_t)$}
\psfrag{vj}[c][c] [0.7] [0] {\Large $v(j_t)$}
 
\psfrag{bn}[c][c] [0.7] [0] {\Large $\beta_n(j_t)$}
\pichere{0.9}{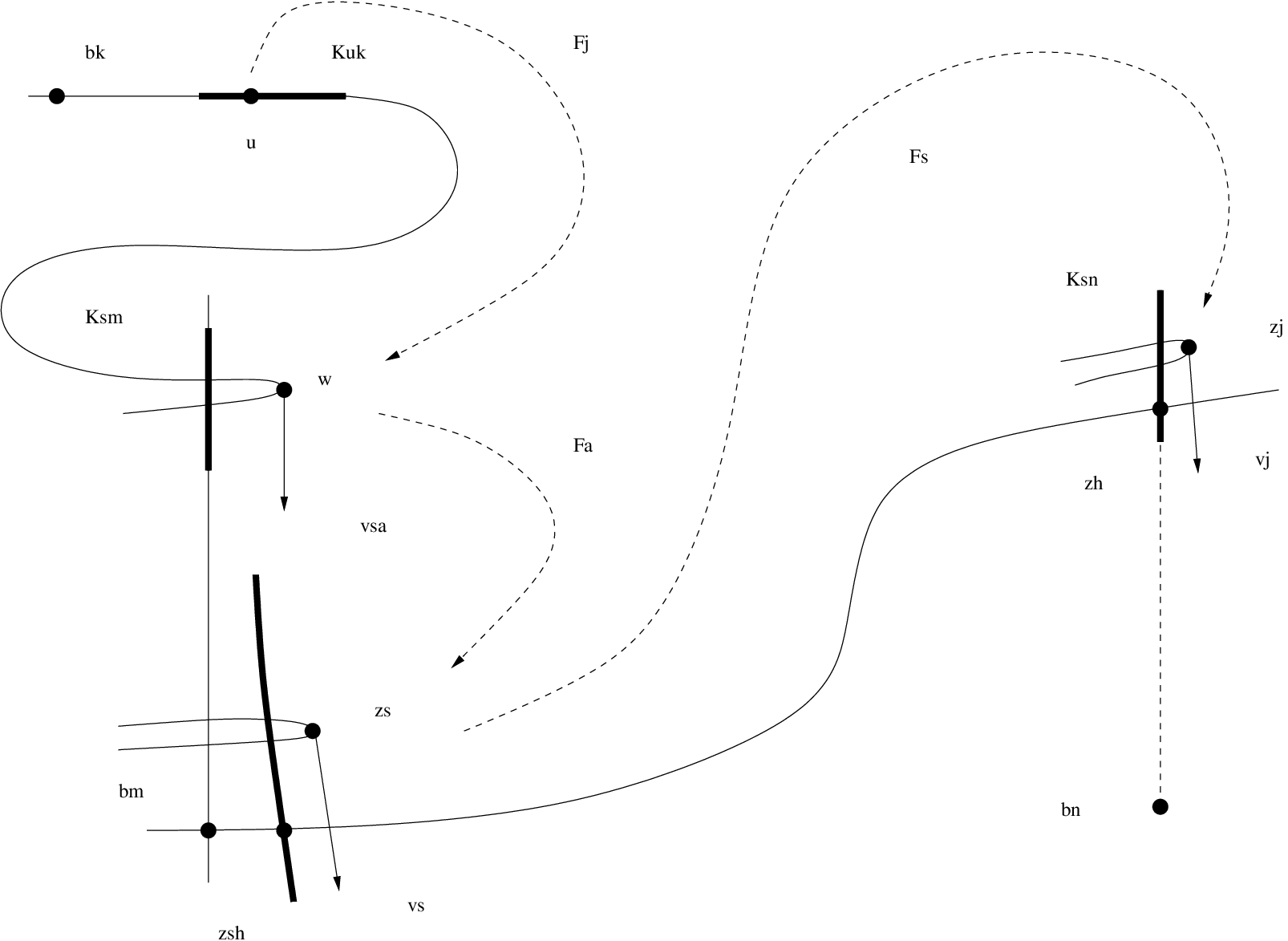}
\caption{}
\label{piclem}
\end{center}
\end{figure}

For each $t\ge 1$ define
$$
\hat{z}_t(j)=F^{-t}_j(\hat{z}(j)),
$$ 
$$
z_t(j)=F_j^{-t}(z(j)),
$$
and
$$
v_t(j)=DF^{-t}_j(v(j)).
$$
From (\ref{conv4}) we get for a given $t\ge 1$
$$
v_t(j)\to T_{\hat{z}_t(j)}W^s_n,
$$ 
when $j\to \infty$. Given $t\ge 1$ we can apply the $\lambda$-Lemma, see 
\cite{dMP},
and choose $j_t$ large enough and $a_t\ge 1$ such that
$$
w_t=z_{t+a_t}(j_t)\to w\in K^s_m,
$$
and
$$
v_{t+a_t}(j_t)\to T_wK^s_m.
$$
The time needed to go from $u(j_s)$ to $w_s$ is
$$
b_t=t_{j_t}-(t+a_t).
$$ 
The induction hypothesis, that is, the Lemma for $k<m<n$ gives a bound $N\ge 1$ such that
$$
b_t< N
$$
for all $t\ge 1$. This implies that 
$
w\in K^s_m
$
is a tangency between $W^u_k$ and $W^s_m$. This is impossible because
$$
F\notin \KK_{k,m}\subset \UU\KK_{k,n-1}.
$$
\end{proof}

Let 
 $F\in \HH^n_\Omega(\overline{\eps})$, 
with $\overline{\eps}>0$ small enough, and $k<n$. For $t\ge 1$ define the curve
$$
 C_{k,t}\equiv [\beta_k,F^t(p_0^{k+1})]^u,
$$
the intersection point $p_0^{k+1}$ of $W^u(\beta_k)$ with $W^s(\beta_{k+1})$ is defined in \S 3. See also Figure \ref{saddlereg}. The following Theorem is a 
reformulation of Lemma \ref{boundedtime}.

\begin{thm}\label{cpt}
Let 
 $F\in \HH^n_\Omega(\overline{\eps})$, 
with $\overline{\eps}>0$ small enough, and $k<n$ be such that
$$
F\notin \UU\KK_{k+1,n}(\overline{\eps})\cup \UU\KK_{k,n-1}(\overline{\eps}).
$$
There exists $N\ge 1$ and a neighborhood $\UU$ of $F\in \HH^n_\Omega(\overline{\eps})$ such that the following holds. If $\tilde{F}\in \UU$ 
and
$$
W^u(\beta_k(\tilde{F}))\tangent_x W^s(\beta_n(\tilde{F}))
$$ 
with  $x\in K^s_n(\tilde{F})$ then
$$
x\in C_{k,N}(\tilde{F}).
$$
\end{thm}

\begin{lem} \label{gentrans} Let 
 $F\in \HH^n_\Omega(\overline{\eps})$, 
with $\overline{\eps}>0$ small enough, and $k<n$, be such that
$F$ has a $(k,n)$-heteroclinic tangency, 
$$
F\in\mathcal{K}_{k,n}(\overline{\eps})
$$ 
but
\begin{equation}\label{cond}
F\notin  \UU\KK_{k+1,n}(\overline{\eps})\cup \UU\KK_{k,n-1}(\overline{\eps}).
\end{equation}
Then for every neighborhood $\UU\ni F$ there exists an open set $\VV\subset \UU$ such that
$$
\VV\cap \UU\KK_{k,n}(\overline{\eps})=\emptyset.
$$ 
\end{lem}

\begin{proof} The collections $\UU\KK_{k+1,n}$ and $\UU\KK_{k,n-1}$ are 
closed, see Lemma
 \ref{Kknclosed}. Let $\UU\supset \VV_0\ni F$ be a neighborhood with 
$$
\VV_0\cap (\UU\KK_{k+1,n}\cup \UU\KK_{k,n-1})=\emptyset.
$$
According to Theorem \ref{cpt} there is $N\ge 1$ and a neighborhood $\VV_1\subset \VV_0$ such that for all $F\in \VV_1$ all the tangencies between 
$W^u(\beta_k(\tilde{F}))$ and $W^s(\beta_n(\tilde{F}))$ in $K^s_n(\tilde{F})$ are in a bounded curve $C_{k,N}(\tilde{F})$.
The maps in  $\HH^n_\Omega(\overline{\eps})$ are analytic. This allows us to 
remove all finitely many tangencies and get an open set  $\VV\subset \VV_1$ 
such that 
$
\VV\cap \UU\KK_{k,n} =\emptyset.
$
\end{proof}

For $F\in \HH^n_\Omega(\overline{\eps})$ define
$$
\Delta_F^n=\min\{ n'-k' | F\in \KK_{k',n'}(\overline{\eps}), \: k'<n'\le n\}.
$$

\begin{cor}\label{erasekn} Let 
 $F\in \HH^n_\Omega(\overline{\eps})$, 
with $\overline{\eps}>0$ small enough, and $\Delta_F^n<\infty$.
Then for every neighborhood $\UU\ni F$ there exists an open set $\VV\subset \UU$ such that
$$
\Delta_{\tilde{F}}^n> \Delta_F^n,
$$
for all $\tilde{F}\in \VV$.
\end{cor}

\begin{proof} Let 
$$
X=\{(k',n')| k'<n'\le n \text{ and } n'-k'< \Delta_F\}
$$
and
$$
Y=\{(k',n')| k'<n'\le n, \:\:  n'-k'= \Delta_F^n, \:\: F\in \KK_{k',n'}(\overline{\eps})\}.
$$
Observe,
$$
F\notin \bigcup_{(k',n')\in X} \KK_{k',n'}(\overline{\eps})=
\bigcup_{n'-k'=\Delta_F^n-1} \UU\KK_{k',n'}(\overline{\eps})
$$
which is a finite union of closed set, see Lemma \ref{Kknclosed}. Choose, $F\in \UU_0\subset \UU$ such that
$$
\UU_0\cap  \bigcup_{(k',n')\in X} \KK_{k',n'}(\overline{\eps})=\emptyset.
$$
Now apply repeatly Lemma \ref{gentrans} to erase the points $(k',n')\in  Y$. We constructed $\VV\subset \UU_0$ such that for $\tilde{F}\in \VV$
$$
\Delta_{\tilde{F}}>\Delta_F.
$$
\end{proof}

\noindent
{\it Proof of Theorem \ref{hyp}.} Lemma \ref{Kknclosed} gives  
$\UU\KK_{0,n}(\overline{\eps})$ is 
closed. In particular, the maps without heteroclinic tangencies form an  open set in $\II^n_\Omega(\overline{\eps})$.
Let $\UU\subset \II^n_\Omega(\overline{\eps}) $ and assume that every map in $\UU$ has a tangency. We have
$$
\Delta_F^n\le n,
$$
for $F\in \UU$. There is an $F_0\in \UU$ which has a tangency and where $\Delta_{F_0}^n$ is maximal. This contradicts Corollary \ref {erasekn}.
\qed

\bigskip

The following Lemmas are a preparation for the proof of Theorem \ref{hypcomp}.

\bigskip

\begin{lem}\label{KcapI} For $\overline{\eps}>0$ small enough, there exist
$n>k\ge 1$ such that for $s\ge 1$ large enough
$$
\II^{n+s}_\Omega(\overline{\eps})\cap \mathcal{K}_{k,n}(\overline{\eps})   
\ne \emptyset.
$$ 
More precisely, there exists an 
$F\in \II^{n+s}_\Omega(\overline{\eps})
\cap \mathcal{K}_{k,n}(\overline{\eps})$
such that $\Gamma_\infty(F_{k-1})$ is tangent to $M_1^{n-k+1}(F_{k-1})$.%
\footnote{See \S \ref{sectip} for the definition of  $\Gamma_\infty(F_{k-1})$ 
          and Figure \ref{extendedlocstabmani} for $M^{n-k+1}_1(F_{k-1})$. }
\end{lem}

\begin{proof} Let $\Omega=\Omega^h\times \Omega^v$ and 
$\WW^n_\Omega(\overline{\eps})\subset\HH^n_\Omega(\overline{\eps}) $ 
consists of the maps that have a periodic point of period 
$2^{n}$ with multiplier  $-1$, $n\ge 1$. 
Choose a unimodal family  $a\mapsto f_a\in \UU_{\Omega^h}$, $a\in (-a_0,a_0)$,
which intersects transversally $W^s(f_*)$ at $a=0$.
For $a_0>0$ small enough we can consider the family 
$$
F_{a,t}(x,y)=(f_a(x)-t\cdot \epsilon(x,y),x)\in \HH_\Omega(\overline{\eps}),
$$
with $a\in (-a_0,a_0)$ and $t\in [0,1]$. 
Let
$$
W^n=\{(a,t)| F_{a,t}\in \WW^n_\Omega(\overline{\eps})\}
$$
and
$$
W=\{(a,t)| F_{a,t}\in \II_\Omega(\overline{\eps})\}.
$$
For $n\ge 1$ large enough and $\overline{\eps}>0$ small enough $W$ and $W^n$ are graphs of analytic functions $t\mapsto a(t)$ and $t\mapsto a_n(t)$. Moreover,
$a_n\to a$ (exponentially fast).  Finally, let
\begin{equation}\label{IN}
\begin{aligned}
I^n&=\{(a,t)| F_{a,t}\in \II^n_\Omega(\overline{\eps})\}\\
   &=\{(a,t)| a_{n-1}(t)< a < a_{n}(t)\}. 
\end{aligned}  
\end{equation}
These statements follows from the fact that the family $a\mapsto f_a$ 
intersects 
$W^s(f_*)$ transversally and the hyperbolicity of the H\'enon-renormalization 
operator.

Choose $t,\tilde{t}\in [0,1]$. Consider the maps 
$F=F_{a(t),t}\in\II_\Omega(\overline{\eps}) $ with average Jacobian 
$b>0$ and $\tilde{F}=F_{a(\tilde{t}),\tilde{t}} 
\in \II_\Omega(\overline{\eps})$ 
with average Jacobian $\tilde{b}>b$. We will use  a short hand notation 
for the  invariants of \S \ref{sectopinv}: let $\kappa_n=\kappa_{R^nF}$ and 
$\tilde{\kappa}_n=\kappa_{R^n\tilde{F}}$.  For $n\ge 1$ large enough we have, as a consequence of Theorem ~\ref{bF},
$$
\kappa_n> 2^n\cdot \{ \frac{|\ln b|}{|\ln \sigma^2|} -1   \}>
2^n\cdot \{ \frac{|\ln \tilde{b}|}{|\ln \sigma^2|} +1 \} > 
\tilde{\kappa}_n.
$$
Hence, for $n\ge 1$ large enough we can find $k\ge 1$ such that
$$
\kappa_{n+1}>k-1>\tilde{\kappa}_{n+1}.
$$
The map $F$ has the property that
$$
M^{k-1}_1(R^{n+1}F)\cap D^\tau(R^{n+1}F)=\emptyset.
$$
In other words, see Figure \ref{extendedlocstabmani}, that
\begin{equation}\label{F}
M^{n+1+k-1}_1(F)\cap [p_0^{n+1},p_1^{n+1}]^u=\emptyset.
\end{equation}
The map $\tilde{F}$ has the property that 
$M^{k-1}_1(R^{n+1}\tilde{F})$ has a nonempty transversal intersection with the boundary of $D^\tau(R^{n+1}\tilde{F})$. In particular,
\begin{equation}\label{Ftilde}
M^{n+1+k-1}_1(\tilde{F})\cap 
[\tilde{p}_0^{n+1},\tilde{p}_1^{n+1}]^u\ne\emptyset,
\end{equation}
and consists of transversal intersections.
The transversal intersections given in (\ref{F}) and (\ref{Ftilde}) 
persist locally. Hence, for $s\ge 1$ 
large enough the set $I^{n+k+s}$
will contain maps with a transverse intersection of type 
$$
M^{n+1+k-1}_1\cap [p_0^{n+1},p_1^{n+1}]^u\ne\emptyset,
$$
 and also maps without any intersection between $M^{n+1+k-1}_1$ and 
$[p_0^{n+1},p_1^{n+1}]^u$. The connectivity of $I^{n+k+s}$, see (\ref{IN}), 
implies that 
there are maps in 
$I^{n+k+s}$ for which there
is a tangency between $[p_0^{n+1},p_1^{n+1}]^u$ and $M^{n+k}_1$. Recall,
$$
[p_0^{n+1},p_1^{n+1}]^u\subset W^u(\beta_{n}),
$$
and 
$$
M^{n+k}_1\subset W^s(\beta_{n+k}).
$$
Hence, 
$$
\Gamma_\infty(F_{n-1})\:\tangent\: M_1^{k+1}(F_{n-1}).
$$
Moreover, for $s\ge 1$ large enough,
$\II^{n+k+s}_\Omega(\overline{\eps}) 
\cap \mathcal{K}_{n,n+k}(\overline{\eps})   
\ne \emptyset$.
\end{proof}

The following definition is illustrated by Figure \ref{prfMScomp}. 

\begin{defn}\label{jkF}
Given a map $F\in 
\HH^{k}_\Omega(\overline{\eps})
$, 
with $\overline{\eps}>0$ small enough
and $k\ge 1$. If $\Gamma_\infty(F)\cap M^k_1(F)=\emptyset$ then define
$$
j_{k}=j_k(F)=\max\{j\ge 2| \Gamma_j\cap M^k_1\ne \emptyset\}.
$$
Otherwise, if  $\Gamma_\infty\cap M^k_1\ne\emptyset$ let
$$
j_{k}=j_k(F)=\max\{j\ge 2| \Gamma_j\cap M^k_1= \emptyset\}.
$$
\end{defn}

\begin{figure}[htbp]
\begin{center}
\psfrag{wsn-1}[c][c] [0.7] [0] {\Large $W^s_{\loc}(\beta_{0})$}

\psfrag{wsn+1}[c][c] [0.7] [0] {\Large $W^s_\loc(\beta_{2})$}
\psfrag{wsn+1'}[c][c] [0.7] [0] {\Large $W^s_\loc(\beta'_{2})$}
\psfrag{wsn}[c][c] [0.7] [0] {\Large $W^s_\loc(\beta_{1})$}
\psfrag{wst}[c][c] [0.7] [0] {\Large $W^s_\loc(\tau)$}

\psfrag{betan-1}[c][c] [0.7] [0] {\Large $\beta_{0}$}
\psfrag{betan}[c][c] [0.7] [0] {\Large $\beta_{1}$}

\psfrag{Gamj}[c][c] [0.7] [0] {\Large $\Gamma_j$}
\psfrag{gaminf}[c][c] [0.7] [0] {\Large $\Gamma_\infty$}
\psfrag{gamjn}[c][c] [0.7] [0] {\Large $\Gamma_{j_{k}}$}

\psfrag{Mifn}[c][c] [0.7] [0] {\Large $M^i_1$}
\psfrag{mlFn}[c][c] [0.7] [0] {\Large $M^k_1$}

\psfrag{x0}[c][c] [0.7] [0] {\Large $x_0$}
\psfrag{x}[c][c] [0.7] [0] {\Large $x$}
\psfrag{y}[c][c] [0.7] [0] {\Large $y$}
\psfrag{z}[c][c] [0.7] [0] {\Large $z$} 
\psfrag{Fx0}[c][c] [0.7] [0] {\Large $F(x_0)$} 

\pichere{1.0}{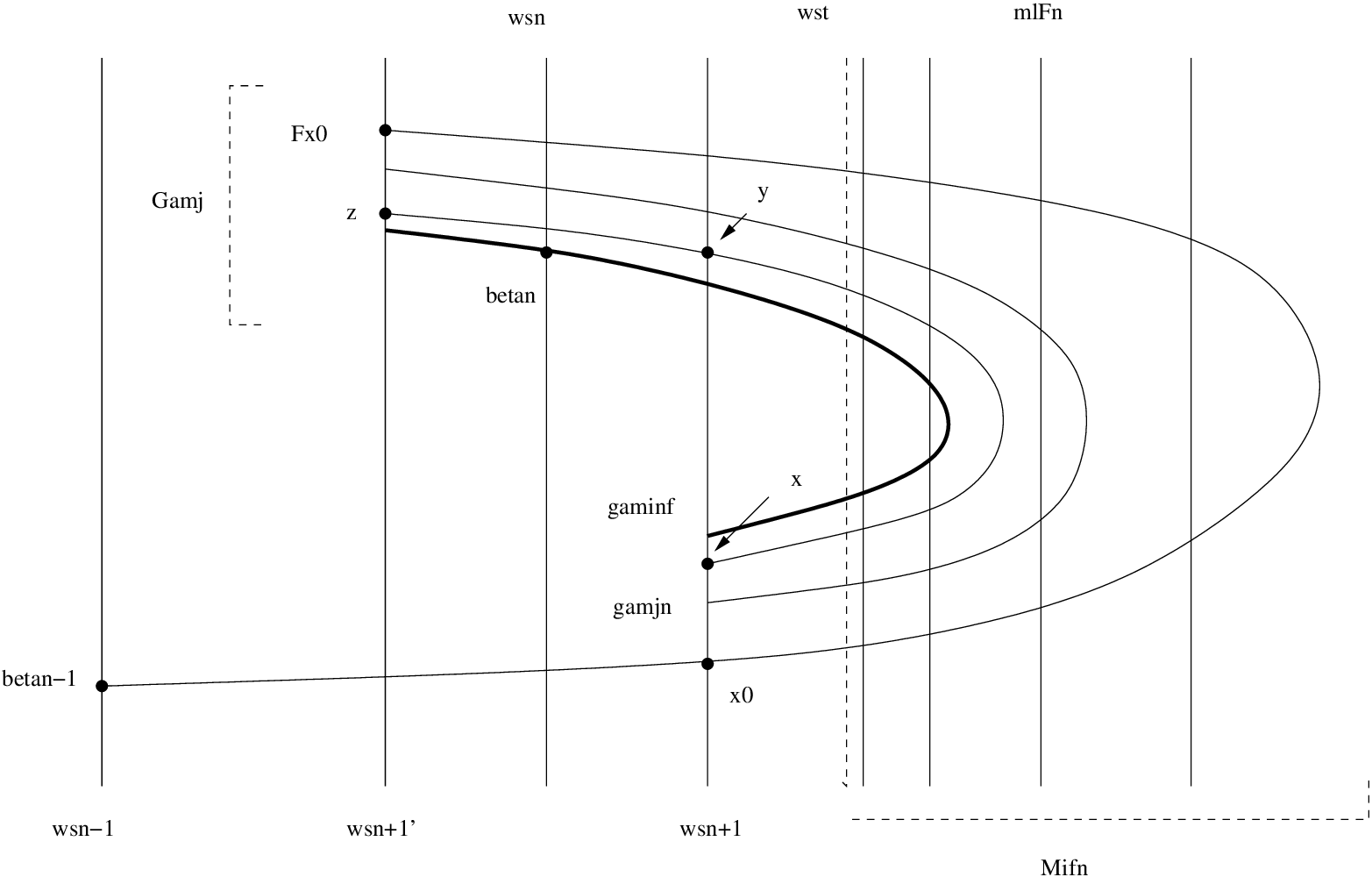}
\caption{}
\label{prfMScomp}
\end{center}
\end{figure}

\begin{lem}\label{jn} $j_{k}(F)$ is a topological invariant.
\end{lem}

\begin{proof} 
Let $x_0\in W^u(\beta_0)$ be the first intersection, coming from 
$\beta_0$ going along $W^u(\beta_0)$, with 
$W^s_{\loc}(\beta_{2})$,  see Figure \ref{prfMScomp}. 
Proposition \ref{conj} implies that this point is topologically defined. 
Notice, that
$$
\Gamma_2=[x_0,F(x_0)]^u\subset W^u(\beta_0).
$$
Thus $\Gamma_2$ is topologically defined. Now, assume that 
$$
\Gamma_j=[z, x]^u\subset W^u(\beta_0),
$$
is topologically defined, with $z\in W^s_{\loc}(\beta'_2)$ and 
$x\in W^s_{\loc}(\beta_{2})$. Then, when $\overline{\eps}>0$
is small enough, there is a unique intersection point 
$$
y\in (z,x)^u\cap W^s_{\loc}(\beta_{2}).
$$
This point is topologically defined. Now observe that
$$
\Gamma_{j+1}=F([z,y]^u).
$$
We proved that all $\Gamma_j$, $j\ge 2$, are topologically 
defined.

Recall from \S \ref{sectopinv} that $D^{\tau}\subset D_1$ is the connected 
component of $D_1\setminus W^s_{\loc}(\tau)$ which does not contain 
$\beta_1$, see Figure \ref{prfbtop} and that
 $M^k_1$ is on the right of $W^s_\loc(\tau)$.
Observe, 
$$
M^k_1\cap D_1=M^k_1\cap D^\tau.
$$
According to Lemma \ref{topinvkappa}, this intersection is a topologically 
defined set. Furthermore, notice that
$$
\Gamma_j\cap M^k_1\subset D_1.
$$
Hence, the intersections 
$$
\Gamma_j\cap M^k_1
$$
are topologically defined. We proved
 that $j_{k,n}$
is a topological invariant.
\end{proof}

 Fix a map $F\in 
\II^{n+s}_\Omega(\overline{\eps})
$, 
with $\overline{\eps}>0$ small enough and $s\ge 1$. We will need the objects 
defined in \S \ref{sectip}. In that section we used  the graphs 
$
\Gamma_j=\Gamma_j(F_n), j=2,3,\dots,
$
and 
$
\Gamma_\infty=\Gamma_\infty(F_n).
$
Apply Corollary  ~\ref{Gammabounds} to obtain a $\rho=\rho(F,n)<1$  such that 
for $j\ge 2$
\begin{equation}\label{Gammaj}
|\Gamma_j(y)-\Gamma_\infty(y)|\ge \rho^j,
\end{equation}
holds  on a neighborhood of $F$. 

\bigskip

\noindent
{\it Proof of Theorem ~\ref{hypcomp}}  Apply Lemma ~\ref{KcapI}: for every 
$\overline{\eps}>0$ small enough, there exist $k,n>0$ such that for every 
$s\ge 1$ large enough there is a map 
$F\in \II^{n+k+s}_\Omega(\overline{\eps})
\cap 
\mathcal{K}_{n,n+k}(\overline{\eps})$
such that $\Gamma_\infty(F_{n-1})$ is tangent to $M_1^k(F_{n-1})$.
Use Theorem ~\ref{hyp} to choose a sequence 
$G_m\in \II^{n+k+s}_\Omega(\overline{\eps})$ of Morse-Smale maps converging to $F$.
Now  
$$
j_{k}(R^{n-1}G_m)\to \infty,
$$
for $m\to \infty$. Otherwise, suppose that for a subsequence 
$j_{k}(R^{n-1}G_m)\le j$ 
stays bounded. Let $\Gamma^m_j$ and $\Gamma^m_\infty$ be the corresponding 
graphs  of the maps $R^{n-1}G_m$. 
Then 
$$
\min_{y}|\Gamma^m_{j+1}(y)-\Gamma^m_\infty(y)|\to 0,
$$
for $m\to \infty$.
This contradicts inequality (\ref{Gammaj}). 
We found countable many Morse-Smale  maps of different 
type accumulating at $F$.
\qed

\begin{rem} Notice that a map $F\in\II^n_\Omega(\overline{\eps})$ might have 
heteroclinic tangencies in places other then the ones discussed in the proof 
of Theorem ~\ref{hypcomp}. The proof only describes some particular
 collection of boundary curves of 
Morse-Smale components but there might be many more components of other types.
\end{rem}

\begin{rem} Observe that the results of this section also hold when we 
consider two dimensional analytic families in $\HH_\Omega(\overline{\eps})$
transverse to 
$\II_\Omega(\overline{\eps})$. 
\end{rem}

\section*{Appendix: Variation of the $\beta_0$-unstable manifold}

Let us consider a one-parameter real analytic family of H\'enon-like maps
\begin{equation}\label{family}
      F_t(x,y) = (f(x)-t\,\gamma(x,y)+O(t^2),\, x),\quad {\mathrm{where}}\ \gamma \equiv \left.{d F_t\over dt}\right|_{t=0}
\end{equation}
 is the parameter velocity of this family at $f$.

Let $\beta_t=\beta_0(t)$ be the saddle fixed point of $F_t$ with positive eigenvalues,
and let $\la_t>1$ be its repelling eigenvalue.
Below we will calculate the first variation of the unstable manifold $W^u(\beta_t)$ at $t=0$.
To this end let us linearize $F_t|\, W^u(\beta_t)$:
$$
  \Phi_t: \R\ra W^u(\beta_t),\quad \Phi_t(\la_t s) = F_t (\Phi_t(s)).
$$
Note that for $t=0$, $\Phi_0(s) = (\phi_0(s),\, \phi_0(s/\la_0))$, where
$\phi_0: \R\ra \R $ is the linearizing map of $f$ at the fixed point $\beta_0$.

Let $\phi_t$ be the first coordinate of $\Phi_t$. Then
\begin{equation}\label{eq phi}
   \phi_t(\la_t s) = f(\phi_t(s)) - t\, \gamma (\phi_t(s),\, \phi_t (s/\la_t))+ O(t^2).
\end{equation}
Let
\begin{equation}\label{expansions}
  \la_t= \la_0+\mu t+O(t^2), \quad \phi_t(s) = \phi_0(s)+\psi(s)\, t + O(t^2).
\end{equation}
Plugging it into (\ref{eq phi}) and keeping only linear terms in $t$, we obtain the following equation:
$$
  \phi_0'(\la_0 s)\mu s + \psi(\la_0 s) = f'(\phi_0(s))\,\psi(s) - \gamma(\phi_0(s), \phi_0(s/\la_0)).
$$

  Let us now look what happens when $\phi_0(s_c)=c$.%
\footnote{recall that $c$ is the critical point and $v$ is the critical value of $f$.}
 Letting $s_v=\la_0 s_c$, we have $\phi_0(s_v)=v$.
Since $v$ is the maximum of $\phi_0$,
the first terms in the  both  sides of the above equation vanish,  
and we obtain:
\begin{equation}\label{psi(v)}
  \psi(s_v) = -\gamma (c, c_{-1}),
\end{equation}
 where $c_{-1}= \phi_0(\tau_c/\la_0)\in f^{-1}(c)$ is a precritical point
(there are infinitely many  values of $s_c$; they split into two classes corresponding to the upper and lower
 critical point on the parabola,  which correspond to  the two precritical points $c_{-1}$).
It allows us to estimate the distance from the {\it turning points} of the 
unstable manifold to the critical value. A turning point of the curve $W^u(\beta_t)$ is a point with vertical tangency.

\begin{lem}\label{turning pts}
  In the above one-parameter family $F_t$ of H\'enon-like maps,
the horizontal distance from the first and the third turns of the unstable manifold $W^u(\beta_t)$ to the
turning point $(v,c)$ of the parabola $x=f(y)$ is $-\gamma(c, c_{-1})\, t+O(t^2)$,
where we should take the lower precritical point $c_{-1}$ for the first turning point,
and the upper precritical point for the second one.
\end{lem}


\begin{proof}
   According to (\ref{psi(v)}), the horizontal distance from $\phi_t(s_v)$ to $v=\phi_0(s_v)$
is $-\gamma(c, c_{-1})\, t+O(t^2)$. However, $\Phi_t(s_v)$ 
is not the turning point $(x_t, y_t)$ of the unstable manifold $W^u(\beta_t)$,
so we need to show that  $|x_t-\phi_t(s_v)|=O(t^2)$. 

Let $x_t=\phi_t(s_t)$. Let us estimate $|s_t-s_v|$. 
Since $\phi'_t(s_t)=0$, the second equation of (\ref{expansions}) yield:
$$
    0= \phi_0'(s_t)+\psi'(s_t) t +\dots
$$
Linearizing this equation in  $s$ near $s_v$, using that $\phi_0'(s_v)=0$, we obtain:
$$
   \phi_0''(s_v)(s_t-s_v) + (\psi'(s_v)+\psi''(s_v)(s_t-s_v)) t+O(t^2) =0.
$$
Hence,
$$
s_t-s_v=\frac{\psi'(s_v)}{\phi''(s_v)}\cdot t +O(t^2),
$$
 provided $\phi_0''(s_v)\not=0$. 
But the latter is actually true, which is easily checked by differentiating twice
at $s_c$ the linearization equation $\phi_0(\la_0 s) = f_0(\phi_0(s))$ 
(using that the critical point $c$ of $f_0$  is non-degenerate).  

Finally, we conclude:
$$
\begin{aligned}
x_t&=\phi_t(s_t)\\
&=\phi_0(s_t)+ \psi(s_t) t+ O(t^2)\\  
&=\phi_0(s_v) + \phi_0' (s_v)\cdot at+ \psi(s_v) t + O(t^2)\\
&=
 v+\psi(s_v) t +O(t2)= \phi_t(s_v)+O(t^2).
\end{aligned}
$$
    \end{proof}

Note that it is reasonable to assume that $\gamma(c, c_{-1})$ is positive at the upper critical point and
negative at the lower one, and has the absolute value of order 1. Then the first turning point of
$W^u(\beta_t)$ is on the right from $(v,c)$ (for $t>0$),
while the third one is on the left (as we always draw), and as $t\to 0$,
they move toward $(v,c)$ with a speed of order 1.

Let $w(F)$ be the horizontal distance between the first and the third turning points of the
unstable manifold $W^u(\beta)$ (which measures the width of the horseshoe near the tip).
Let $c^+$ and $c^-$ denote the upper and the lower precritical points of $f$ respectively.

\begin{lem}\label{width}
 For the family (\ref{family}),
$$
   \left.{ d w(F_t) \over dt} \right|_{t=0} = \int_{c^-}^{c^+} \left.{d \Jac F_t(c,y) \over dt}\right|_{t=0} \, dy
$$
\end{lem}

\begin{proof}
According to Lemma \ref{turning pts},
$$
\aligned
   \left.{ d w(F_t) \over dt} \right|_{t=0} &= \gamma(c, c^+)  -\gamma(c, c^-) =
    \int_{c^-}^{c^+} {\di \gamma(c,y) \over \di y  } dy \\
    &=  \int_{c^-}^{c^+} \left.{d \Jac F_t(c,y) \over dt}\right|_{t=0} dy.
\endaligned
$$
\end{proof}

The last formula can also be written in the following form:
For a H\'enon-like map $F= (f- \eps, x)\in \HH_\Om(\bar\eps)$,

\begin{equation}\label{variation of width}
    \de w (f) =  \int_{c^-}^{c^+} \de \Jac F (c,y)\,  dy
\end{equation}

\begin{lem}
  For $F(x,y)= (f(x) - b a(x,y), x)\in \HH_\Om(\bar\eps)$, assume
$$C^{-1}\leq |\di a/\di y|\leq C.$$
Then $ w(F) \asymp b$ for $b\leq b_0$, with the constants depending only on $\Om$, $\bar\eps$ and $C$.
\end{lem}

\begin{proof}
  Consider $b\in [0,b_0]$ as a small parameter. Then by the variational formula (\ref{variation of width})
$$
   \left.{ d w(F_b) \over db}\right|_{b=0} = \int_{c^-}^{c^+} {\di a\over \di y}(c,y) dy \asymp 1.
$$
Moreover, the second derivative $d2 w(F_b)/ db2$ is bounded on the interval $[0, b_0]$,
uniformly over $F\in \HH_\Om(\bar\eps)$ (since $w(F)$ is a $C2$-smooth, in fact analytic,  functional on this space),
and the conclusion follows by elementary calculus.
\end{proof}

The asymptotic expression for $R^nF$ as given in (\ref{univ}) implies:

\begin{prop} For $F_b\in \II_\Om(\bar\eps)$, be a family of infinitely renormalizable maps parametrized by the average Jacobian $b=b_{F_b}$. Then
$$
\lim_{b\to 0} \frac{w(R^n F_b)}{b^{2^n}}=a(c)\cdot (c^+-c^-)
$$ 
 where $b$ is the average Jacobian of $F$ and $c^\pm$ are the preimages of the critical point $c$ of $f_*$. And $x\mapsto a(x)$ the universal function given in 
(\ref{univ}).
\end{prop}

\section*{Appendix: Open Problems}\label{problems}

Let us finish with some questions related to the previous discussion.
\bigskip

\noindent
\underline{Problem I:}

The following questions are inspired by the results of \S 9 on stable laminations.
\begin{enumerate}
\item[(1)] A {\it wandering domain} is an open set in the basin of attraction of $\OO_F$. Do wandering domains exist?
\item[(2)] If a map $F\in \II_\Omega(\overline{\eps})$ does not have wandering domains then the union $\FF^s$ of all stable manifolds of periodic points are dense in the domain of $F$. Does there exist $F\in \II_\Omega(\overline{\eps})$ 
such that $\FF^s$ is not laminar even if there are no heteroclinic 
tangencies? 
\item[(3)] For $F\in \II_\Omega(\overline{\eps})$  let $\FF^s_\tau$ be the 
union of stable manifolds of the points in the orbit of the tip. Is 
$\FF^s_\tau$ dense in $\Dom(F)$? 
\end{enumerate}
\noindent

\bigskip

\noindent
\underline{Problem II:}

It is shown in \cite{CLM} that the unique invariant measure on the Cantor attractor $\OO_F$ 
has characteristic exponents $0$  and  $\ln b_F<0$.
Can the stable characteristic exponent of the tip $\tau_F$  differ 
from $\ln b_F$ (compare Proposition \ref{stablesettip})?

\bigskip

\noindent
\underline{Problem III:}

Can we still speak of rigidity of the Cantor attractor $\OO_F$?
\begin{enumerate}
\item[(1)] Are the Cantor attractors rigid within the topological conjugacy classes?

\item[(2)]  Prove or disprove that two Cantor attractors $\OO_F$ and $\OO_{\tilde{F}}$ 
are smoothly equivalent if and only if they have the same  average Jacobian.
\end{enumerate}

\bigskip

\noindent
\underline{Problem IV:}
\begin{enumerate}
\item[(1)] Can different Morse-Smale components
$$
MS_1, MS_2\subset \bigcup_{n\ge 0} \II^n_\Omega(\overline{\eps})
$$
have the same type, that is the maps in $MS_1$ are conjugate to the maps in $MS_2$?

\item[(2)] As we have shown, the Morse-Smale H\'enon maps are dense in the zero entropy region with small Jacobian. 
    Are they dense in the full zero entropy region of dissipative H\'enon maps?
    How about other real analytic families of dissipative two dimensional maps?

\item[(3)] The discussion that led to Theorem \ref{hypcomp} was based on the renormalization structure.  
However, the non-locally finiteness of the collection of Morse-Smale components might be 
a more general phenomenon.
Study the combinatorics of Morse-Smale components in other real analytic families of  dissipative two dimensional maps.

\item[(4)] Are the real Morse-Smale H\'enon maps 
from Theorem \ref{hypcomp} hyperbolic on $\Bbb{C}^2$? 
To what extent determines the topology of the 
real heteroclinic web the topology of the  
corresponding H\'enon map on $\Bbb{C}^2$? 
\end{enumerate}

\section*{Nomenclature}   
\begin{itemize}
\item[$\AAA_F$] global attractor \S \ref{Lamstruc}
\item[$\AAA^n_F$] $n^{th}-$scale attractor \S \ref{Lamstruc}
\item[$b_F$] average Jacobian, \S \ref{prelim}
\item[$B_0$] non-escaping points, \S \ref{Lamstruc}
\item[$B_{v^n}$] renormalization piece around the tip of level $n$, \S \ref{prelim}
\item[$\hat{\beta}_n$] fixed point of $R^nF$, \S \ref{secstabman}
\item[$\beta_n$] periodic point, \S \ref{secstabman}
\item[$\beta'_n$] periodic point, \S \ref{secstabman}
\item[$\mathcal{C}_F$] non-laminar points in $\AAA_F$, \S \ref{hettan}
\item[$\gamma_\infty$] curve, Figure \ref{hatFn} 
\item[$\gamma_j$] curve, Figure \ref{hatFn} 
\item[$\Gamma_\infty$] curve, Figure \ref{hatFn} 
\item[$\Gamma_j$] curve, Figure \ref{hatFn} 
\item[$D_n$] periodic domain containing the tip, \S \ref{secstabman}
\item[$D^\tau$] domain bounded by $W^s_\loc(\tau)$ and $W^u(\beta_0)$, \S \ref{sectopinv}
\item[$E_{k,n}$] heteroclinic points, \S \ref{Lamstruc}
\item[$E^s_{k,n}$] heteroclinic points, \S \ref{stablam} 
\item[$\Phi^n_0$] coordinate change, \S \ref{secstabman}
\item[$\Phi^{k+1}_k$] coordinate change, \S \ref{secstabman}
\item[$\HH^n_\Omega(\overline{\eps})$] $n$-times renormalizable maps, \S \ref{prelim}
\item[$\II_\Omega(\overline{\eps})$] infinitely  renormalizable maps, \S \ref{prelim}
\item[$\II^n_\Omega(\overline{\epsilon})$] $n$-times renormalizable maps with a periodic attractor of period $2^n$, \S \ref{prelim}
\item[$K^u_{n}$] fundamental domain in $W^u(\beta_n)$, \S \ref{Lamstruc}
\item[$K^s_{n}$] fundamental domain in $W^s(\beta_n)$, \S \ref{stablam}
\item[$\mathcal{K}_{k,n}(\overline{\eps})$] maps with heteroclinic tangencies, Definition \ref{Kkn}
\item[$\mathcal{UK}_{k,n}(\overline{\eps})$] maps in $\mathcal{K}_{k',n'}(\overline{\eps})$ with $k\le k'<n'\le n$,
 Definition \ref{Kkn}
\item[$\boldsymbol{\kappa}_F$] topological invariant, \S \ref{sectopinv}
\item[$\lambda_n$] stable eigenvalue of $\beta_n$, \S \ref{hettan}
\item[$\mu_n$] unstable eigenvalue of $\beta_n$, \S \ref{hettan}
\item[$\hat{M}^n_i$] component stable manifold of $\hat{\beta}_n$, \S \ref{secstabman}
\item[$M^n_i$] component stable manifold of $\beta_n$, \S \ref{secstabman}
\item[$\OO_F$] critical Cantor set, \S \ref{prelim}
\item[$\hat{p}^n_i$] heteroclinic point of $R^nF$, \S \ref{secstabman}
\item[$p^n_i$] heteroclinic point of $F$, \S \ref{secstabman}, also
Figure ~\ref{extendedlocstabmani}
\item[$q_i$] heteroclinic point of $R^nF$, \S \ref{Lamstruc}, also
Figure ~\ref{saddlereg}
\item[$q'_i$] heteroclinic point of $R^nF$, \S \ref{Lamstruc}, also
Figure ~\ref{saddlereg}
\item[$\sigma$] scaling factor of the unimodal fixed point, \S \ref{prelim} 
\item[$\Trap_n$] $n^{th}-$trapping region, \S \ref{Lamstruc}
\item[$\tau_F$] tip, \S \ref{prelim}
\end{itemize}


\begin{thebibliography}{*****}

\bibitem[AM]{AM} J. M. Aarts, M. Martens. Flows on one-dimensional
spaces, Fund. Math.  131 (1988). 

\bibitem[CEK]{CEK} P. Collet, J. P. Eckmann, H.~Koch. Period doubling
bifurcations for families of maps on $\R^n$.  J. Stat. Physics 25 (1980),  1-15.

\bibitem[CLM]{CLM} A. de Carvalho, M. Lyubich, M. Martens. Renormalization in the H\'enon family, I: universality but non-rigidity, 
J. Stat. Phys. 121 No. 5/6, (2005), 611-669.

\bibitem[C]{C} S. Crovisier. Birth of homoclinic intersections: a model for the central dynamics of partially hyperbolic systems. arXiv: math.DS/0605387.




\bibitem[dMP]{dMP} W. de Melo, J. Palis. Geometric Theory of Dynamical Systems, Springer 1982. 

 


\bibitem[FMP]{FMP} E.~de Faria, W.~de Melo \& A.~Pinto.
Global hyperbolicity of renormalization for $C^r$ unimodal
mappings, Ann. of Math. 164 No. 3, (2006), 731-824.

\bibitem[GST]{GST} J.-M. Gambaudo, S. van Strien \&  C. Tresser. H\'enon-like
maps with strange attractors: there exist $C^\infty$ Kupka-Smale
diffeomorphisms on $S^2$ with neither sinks nor sources.
Nonlinearity  2 (1989), 287-304.

\bibitem[H]{H} P. Hartman. On local homeomorphisms of Euclidean spaces, Bol. Soc. Mexicana  5 (1960), 220-241.




\bibitem[Mi1]{Mi1} J.W. Milnor. On the concept of attractor, 
Comm. Math. Phys. 99 (1985), 177-195.

\bibitem[Mi2]{Mi2} J.W. Milnor. Attractor, Scholarpedia.




\bibitem[P]{P} J. Palis. On Morse-Smale dynamical systems, Topology  8
(1968), 385-404.

\bibitem[PS]{PS} E. R. Pujals \& M. Sambarino. Homoclinic tangencies and 
hyperbolicity for surface diffeomorphisms. Ann. Math. 51 (2000), 961-1023.

\bibitem[PSW]{PSW} C. Pugh, M. Shub,\& A. Wilkinson. H\"older Foliations,
Duke Math. J. 86 (1997), 517-546. Correction Vol. 105 (2000) 105-106.



\end{thebibliography}
\end{document}